\documentclass[11pt,reqno]{amsart}
\usepackage{graphicx}
\usepackage{amscd}
\usepackage{amsmath,amsthm}
\usepackage{amsfonts}
\usepackage{amssymb}
\usepackage{color}
\usepackage{hyperref}
\usepackage{tikz-cd}
\usepackage{float}
\usepackage{nccmath}
\usepackage{graphicx}
\usepackage{multirow}
\usepackage{geometry}

\usepackage{yhmath}
\usepackage{diagbox}
\usepackage{enumitem}

\DeclareRobustCommand{\rchi}{{\mathpalette\irchi\relax}}
\newcommand{\irchi}[2]{\raisebox{\depth}{$#1\chi$}} 
\newcommand{\hlt}[1]{\textcolor{red}{#1}}

\setlist[enumerate]{leftmargin=*, labelindent=0pt, labelwidth=1.5em}

\DeclareMathOperator{\Span}{span}

\DeclareMathOperator{\TR}{trace}
\DeclareMathOperator{\Div}{div}
\DeclareMathOperator{\ord}{ord}

\DeclareMathOperator{\supp}{Supp}
\DeclareMathOperator{\Aut}{Aut}
\DeclareMathOperator{\Id}{Id}
\DeclareMathOperator{\mult}{mult}
\DeclareMathOperator{\Diag}{diag}
\DeclareMathOperator{\Sym}{Sym}
\DeclareMathOperator{\Alt}{Alt}

\usepackage{xcolor}

\newtheorem{theorem}{Theorem}[section]
\newtheorem{lemma}{Lemma}[section]

\newtheorem{corollary}{Corollary}[section]
\newtheorem{proposition}{Proposition}[section]
\newtheorem{definition}{Definition}[section]
\newtheorem{remark}{Remark}[section]

\numberwithin{equation}{section}

\theoremstyle{definition}

\makeatletter
\@namedef{subjclassname@2020}{%
  \textup{2020} Mathematics Subject Classification}
\makeatother

\begin{document}
\bibliographystyle{plain}

\author{Quo-Shin Chi, Zhenxiao Xie \& Yan Xu}
\thanks{This work was partially supported by NSFC No. 12171473 for the last two authors. The second author was also partially supported by the Fundamental Research Funds for Central Universities, and the third author was also partially supported by by NSFC No. 12301065 and Natural Science Foundation of Tianjin (22JCQNJC00880).}
\title{Classification of sextic curves  in the Fano 3-fold $\mathcal{V}_5$ with rational Galois covers in ${\mathbb P}^3$}

\begin{abstract} 
In this paper, we classify sextic curves in the Fano $3$-fold $\bf \mathcal{V}_5$ (the smooth quintic del Pezzo $3$-fold) that admit rational Galois covers in the complex ${\mathbb P}^3$. We show that the moduli space of such sextic curves is of complex dimension $2$ 
through the invariants of the engaged Galois groups for the explicit constructions. This raises the intriguing question of understanding the moduli space of sextic curves in ${\mathcal V}_5$ through their Galois covers in ${\mathbb P}^3$.
\end{abstract}

\maketitle

\vskip -0.6cm
{\footnotesize 
{\bf MSC(2020): Primary 53C42, 53C55; Secondary 14M15, 14J45, 14H45}

{\bf Keywords.} holomorphic curves, linear section, quintic del Pezzo 3-fold
}

\section{Introduction} The paper, representation theory of certain Galois groups in nature,  grew out of a quest to understand smooth rational curves in general, and, in particular, those of constant curvature in the complex Grassmannian $G(2,5)$, which are closely tied to the $\sigma$-model theory in particle physics (see e.g., \cite{Din-Zakrzewski1981}, \cite{ Delisle-Hussin-Zakrzewski2013}, \cite{Marsh}).


It has been a long-standing topic in Differential Geometry to study minimal surfaces to find, through the holomorphic data they enjoy, their deep relations to and applications in Differential Equations and Physics. Thanks to the foundational works of Din-Zakrzewski \cite{Din-Zakrzewski1980}, Chern-Wolfson \cite{Chern-Wolfson1987}, Eells-Wood \cite{Burstall-Wood1989}, Ulenbeck \cite{Uhlenbeck1989}, and others, it is known that any minimal 2-sphere immersed in compact symmetric spaces can be associated with a holomorphic curve in certain complex ambient spaces.

Among all minimal $2$-spheres, the classification of those with constant curvature in various symmetric spaces has consistently attracted attention. While such spheres exhibit rigidity in real and complex space forms \cite{Calabi1967, Bolton-Jensen-Rigoli-Woodward1988, Bando-Ohnita1987}, their classification in general symmetric spaces remains largely open. So far, only the case of the complex Grassmannian $G(2,4)$ has been fully resolved \cite{Chi-Zheng1989, ZQLi-ZuHuanYu1999}. For $G(2,5)$, a partial classification result for constantly curved holomorphic $2$-spheres was done in \cite{Jiao-Peng2004, Jiao-Peng2011}, where only those of degree no larger than $5$ appear.  Furthermore, a conjecture in \cite{Delisle-Hussin-Zakrzewski2013} suggests that holomorphic $2$-spheres of degrees $7$ to $9$ in $G(2,5)$ cannot have constant curvature.

For the case of the critical degree $6$ in $\mathrm{G}(2,5)$, only one example existed prior to our recent work \cite{Chi-Xie-Xu2024}, in which we constructed a real $2$-parameter family of constantly curved holomorphic $2$-spheres of degree $6$ by utilizing the structure of the Fano $3$-fold ${\mathcal{V}}_5$ in $\mathrm{G}(2,5)$. 
We say that a ${\mathbb P}^1$ of degree 6 (referred to as a {\em sextic curve} henceforth) in $\bf {\mathcal V}_5$ is {\em generally ramified} if it is tangent to the $1$-dimensional $PSL_2$-invariant orbit of $\bf {\mathcal V}_5$; otherwise, we say that it belongs to the {\em exceptional transversal family}. In that paper, we showed that a holomorphic $2$-sphere of constant curvature in $G(2, 5)$, $PGL(5)$-equivalent to a generally ramified 
sextic curve, is such that the $PGL(5)$-equivalence is induced by a diagonal matrix, so that all such $2$-spheres can be completely classified to form a semialgebraic moduli space of real dimension 2. Then it is natural to ask whether exceptional transversal family exists. This is the motivation of the current paper. 

To this end, we have established in \cite{Chi-Xie-Xu2024} that a sextic curve in $\bf {\mathcal V}_5$ admits a Galois cover in ${\mathbb P}^3$ with the Galois group a subgroup of the projective binary octahedral group isomorphic to $S_4$, where we view ${\mathbb P}^3$ as the closure of $PSL_2$ and the latter covers the $3$-dimensional $PSL_2$-orbit of $\bf {\mathcal V}_5$ with $S_4$ as the deck transformation group.

In the following, employing the representation theory of finite groups and their invariants, we will 
classify all sextic curves in $\bf \mathcal{V}_5$ each admitting a rational Galois cover in ${\mathbb P}^3$.

\begin{theorem}
The moduli space of sextic curves in the Fano $3$-fold $\bf \mathcal{V}_5$ each admitting a rational Galois cover in ${\mathbb P}^3$ is of {\rm (}complex{\rm )} dimension $2$, among which the space of exceptional transversal family is of dimension $1$. 
\end{theorem}
Our classification via explicit construction reduces to finding a 2-plane in a k-fold copy of a 1-dimensional irreducible representation in $V_n \otimes \mathbb{C}^2$ for $k \geq 2$, where $V_n$ denotes the space of binary forms of degree $n$. 
The invariant theory of finite groups plays more of a pivotal role in our work 
as the group order gets larger. 
To this end, it is worth remarking that the well-known projection formula for calculating a summand in the canonical decomposition of a $G$-module of a finite group $G$, in our situation, turns out to engage in calculations so long that one needs to practically resort to a symbolic software for an answer without conceptualizing a coherent pattern. In contrast, our approach is to find a summand we seek by expressing a basis of it in terms of the generating invariants of $G$-representations, for which Proposition \ref{key basis prop} is the engine and from which the summand structure manifests itself.

 We note that Takagi and Zucconi in \cite{Takagi-1, Takagi-2} studied the moduli space of sextic curves in ${\mathcal{V}}_5$, 
 where they proved, modulo the ${PSL}_2$-action, that it has dimension $9$, which infers by the preceding theorem that the Galois lifts of generic sextic curves in $\bf {\mathcal V}_5$ are curves of higher genera with nontrivial automorphism groups. This naturally raises the intriguing question of understanding  the moduli space of sextic curves in  $\bf {\mathcal V}_5$ 
 via their Galois covers in ${\mathbb P}^3$. 

The paper is organized as follows. In section 2, we review some necessary background information about the representattion theory of finite groups in general and our work on the sextic curves in the Fano ${\mathcal V}_5$ \cite{Chi-Xie-Xu2024}. Section 3 explains the underlying idea of the classification scheme through the representations of the finite subgroups of $SU(2)$. In sections 4 through 6, we work out the cyclic cases. In sections 7 through 10, the dihedral cases are carried out, and  we classify in sections 11 and 12, respectively, the cases $A_4$ and $S_4$. 

\section{Preliminaries}

\subsection{Irreducible representations of $SL_2(\mathbb{C})$.} ~

Let $V_n$ be the the space of binary forms of degree $n$ in two variables $u$ and $v$, on which $SL_2({\mathbb C})$, to be denoted by $SL_2$ henceforth,
acts by 
\begin{small}$$
\aligned
&SL_2\times V_n \rightarrow V_n,\quad (A,f) &\mapsto  (A\cdot f)(u,v)\triangleq f (A^{-1}\cdot (u,v)^T). 
\endaligned
$$\end{small}
The spaces $V_n$ for $n\in\mathbb{Z}_{\geq 0}$ exhaust all finite-dimensional irreducible representations of $SL_2$. Furthermore, the symplectic form $\det(a,b)$ for $a,b\in(V_1)^* (\simeq {\mathbb C}^2$) induces a canonical $SL_2$-isomorphism $(V_1)^* \cong V_1$ via $e \mapsto \det(e,\cdot)$, where $SL_2$ acts on the row spaces $V_1$ and $(V_1)^*$ by right matrix multiplication, so that we may also work  with $V_n\otimes V_1$ henceforth as now $V_n\otimes V_1\simeq V_n\otimes (V_1)^*$ (in matrix terms, it is because {\small$A^{T}=JA^{-1}J^{-1}, J\triangleq \begin{pmatrix}0&-1\\1&0\end{pmatrix}, A\in SL_2$}).

\begin{lemma}\label{eigen} For $A\in SL_2$, 
the standard tensor product action  on 
$V_n\otimes {\mathbb C}^2$, where $p_n\otimes e_1+q_n\otimes e_2$ is written $(p_n,q_n)$ as usual, is given in matrix form by 
\begin{equation}\label{AB}\small{
A\otimes A: (p_n, q_n) \longmapsto (A\cdot p_n, A\cdot q_n)\, A^T}
\end{equation}
with matrix multiplication on the right. In particular, $p_n\otimes e_1 + q_n\otimes e_2$ spans a $1$-dimensional invariant subspace of a subgroup $G\subset SL_2$ with character $\chi$ {\rm (}see the next subsection for definition{\rm )} in $V_n\otimes {\mathbb C}^2$ if and only if, by \eqref{AB},
\begin{equation}\label{pqn}\small{
(A\cdot p_n,A\cdot q_n)=\chi(A)\, (p_n,q_n)\,(A^{-1})^T},
\end{equation}
\end{lemma}

The lemma is straightforward to verify. It is also directly checked that
\begin{equation}\label{V1 equi to natural action}
(A\cdot u,A\cdot v)=(u,v)(A^{-1})^T,\quad\forall~A\in  SL_2(\mathbb{C}).
\end{equation}
The Clebsch-Gordan formula states that  (assume $m\geq n$) 
\begin{small}$$ 
V_m\otimes V_n\cong V_{m+n}\oplus V_{m+n-2}\oplus\cdots \oplus V_{m-n}.
$$\end{small} 
Moreover, for any given number $p\in [0, n]$, the projection $V_m\otimes V_n  \rightarrow V_{m+n-2p}$ can be formulated by 
\begin{equation}\label{transvectant}\medmath{
(f,h)\mapsto (f,h)_{p}\triangleq \frac{(m-p)!(n-p)!}{m!n!}\sum_{i=0}^{p}(-1)^{i}\binom{p}{i}\frac{\partial^p f}{\partial u^{p-i}\partial v^i}\frac{\partial^p h}{\partial u^{i}\partial v^{p-i}},}
\end{equation}
which is $SL_2(\mathbb{C})$-equivariant and is called the \emph{$p$-th transvectant} \cite[Eq (2.1), p.~16]{MP}. 

\subsection{Representation of finite groups}~

In this subsection, we introduce some basic facts about the representation theory of finite groups 
 \cite{Serre1977}. The complex number field is enforced in the sequel.

Let $G$ be a finite group. Let $\rho:G\rightarrow GL(V)$ be a linear representation  over the vector space $V$. The function $\rchi_\rho(s):=\TR(\rho_s)$ for all  $s\in G$
is the \emph{character} of $\rho$.

Let $\rho^1:G\rightarrow GL(W_1)$ and $\rho^2:G\rightarrow GL(W_2)$ be two linear representations of $G$, and let $\rchi_1$ and $\rchi_2$ be their characters. Then
The character $\rchi$ of the direct sum representation $W_1\oplus W_2$ is equal to $\rchi_1+\rchi_2$, while
the character $\psi$ of the tensor product representation $W_1\otimes W_2$ is equal to $\rchi_1\cdot \rchi_2$.

A complex-valued function $f$ on $G$ is called a \emph{class function} if $f(tst^{-1})=f(s),~\forall t,s\in G$. For example, the characters of representations of $G$ are class functions. Let ${\mathcal C}$ be the space of class functions on $G$. Equip ${\mathcal C}$ with an inner product 
\begin{small}$$(\phi|\psi)\triangleq\frac{1}{|G|}\sum\limits_{t\in G}\phi(t)\overline{\psi(t)}.$$\end{small}

The number of irreducible representations of $G$ {\rm (}up to isomorphism{\rm )} is equal to the number of conjugacy classes of $G$, while
all the irreducible characters $\rchi_1,\rchi_2,\ldots,\rchi_h$ form an orthonormal basis of ${\mathcal C}$.  Moreover, two representations with the same character are isomorphic, and, in particular, $V$ with character $\phi$ is irreducible if and only if $(\phi|\phi)=1$.
 
 Suppose now a linear representation $V$ of $G$ with character $\rchi_V$ decomposes into a direct sum of irreducible representations 
$$V=n_1 W_1\oplus \cdots \oplus n_h W_h,$$
where $\{W_1,\cdots, W_h\}$ consists of non-isomorphic irreducible representations appearing in $V$, and $n_i=(\rchi_V|\rchi_{W_i})$ denotes the number of irreducible representations isomorphic to $W_j$ in the decomposition whose direct sum is denoted by  $n_j W_j$.

\begin{theorem}
Let $\rho:G\rightarrow GL(V)$ be a linear representation of a finite group $G$.
\begin{enumerate}
\item[{\rm(1)}] The 
decomposition 
$V=n_1 W_1\oplus \cdots \oplus n_h W_h$ does not depend on the initially chosen decomposition $W_1,\cdots,W_h$ 
of $V$ into irreducible representations.

\item[{\rm(2)}] The projection $p_i$ of $V$ onto $n_i W_i$ associated with this decomposition is given by 
{\small\begin{equation}\label{general projection formula}
p_i\triangleq \frac{n_i}{g}\sum\limits_{t\in G} \overline{\rchi_i(t)}\rho_t.
\end{equation}}

\end{enumerate}
\end{theorem}

\begin{remark}
In practice, however, applying the projection formula \eqref{general projection formula} in computations is rather involved. Later, we will employ group invariants to devise an easier method, see Proposition {\rm \ref{key basis prop}{\rm }}.
\end{remark}



\subsection{Finite subgroups of $SO(3)$ and $SU(2)$}

The following is well-known.

\begin{theorem}\cite[p.~20] 
{GaborToth2002}
Up to conjugation, the only finite subgroups of $SO(3)$ are the cyclic groups $C_d$, the dihedral groups $D_d$, the tetrahedral group $A_4$, the octahedral group $S_4$, and the icosahedral group $A_5$.
\end{theorem}

Note that $\pi:SU(2)\rightarrow SO(3)\cong SU(2)/\{\pm \Id\}$ is a double cover. Given a finite subgroup $G$ of $SO(3)$, the inverse image $G^\ast$ of $G$ under $\pi$ is called the \emph{binary group associated to $G$}.

\begin{theorem}\cite[p.~36] 
{GaborToth2002}
Any finite subgroup of $SU(2)$ is either cyclic $C_d$ {\rm (}for all $d\in \mathbb{Z}_{\geq 0}${\rm )}, or conjugate to one of the binary subgroups $D_d^\ast,~A_4^\ast,~S_4^\ast$, or $A_5^\ast$.
\end{theorem}

Since these binary polyhedral groups $G^\ast$ are also subgroups of $SL_2(\mathbb{C})$, one obtains representations of $G^\ast$ on the space of binary forms $V_n$. Indeed, given $A\in G^\ast$ with eigenvalues $\lambda,\lambda^{-1}$, it is directly checked that
\begin{equation}\label{formula to compute char on Vn}
\rchi_{V_n}(A)=(\lambda^{n+1}-\lambda^{-n-1})/(\lambda-\lambda^{-1}).
\end{equation}

Let $v$ be an irreducible character of $G^\ast$. Denote by $m_{h,v}=(\rchi_{V_n}|v)$ the multiplicity of $v$ in the representation $G^\ast$ in the space $V_n$ of binary forms. Springer \cite{Springer} gave an explicit formula for the Poincar\'e series
\begin{small}$$
P_v(T) =\sum_{h=0}^\infty m_{h,v}\, T^h.
$$\end{small}

\begin{proposition}\cite[p.~105, 4.2. p.~107, 4.4]
{Springer}\label{poincareEven}
Let $G^\ast$ be one of $D_{2n}^\ast,~A_4^\ast,~S_4^\ast$. Then any irreducible representation of degree $1$ of $G^\ast$ only occurs in $V_{2m}$ due to that its Poincar\'e series is an even function.

\end{proposition}

\subsection{Invariants}

Let $G$ be a finite subgroup of $SU(2)$. A non-zero homogeneous binary polynomial $f\in V_n$ is called an \emph{invariant} of $G$ if $\Span\{f\}$ is stable under the action of $G$. Thus, $\Span\{f\}$ is an irreducible representation with degree $1$ of $G$; denote its character by $\rchi_f$ for simplify. Furthermore, we call $f$ an \emph{absolute invariant} if 
$\rchi_f\equiv 1$. 



For later reference, we use $x$ and $y$ as the variables in binary forms. The invariants of finite subgroup of $SU(2)$ are classified as polynomials of fundamental invariants, which are closely related to the geometry of regular polyhedra depicted in Figure~1 for easier visulization.
\begin{figure}[htbp!]
\centering
\includegraphics[width=0.65\textwidth]{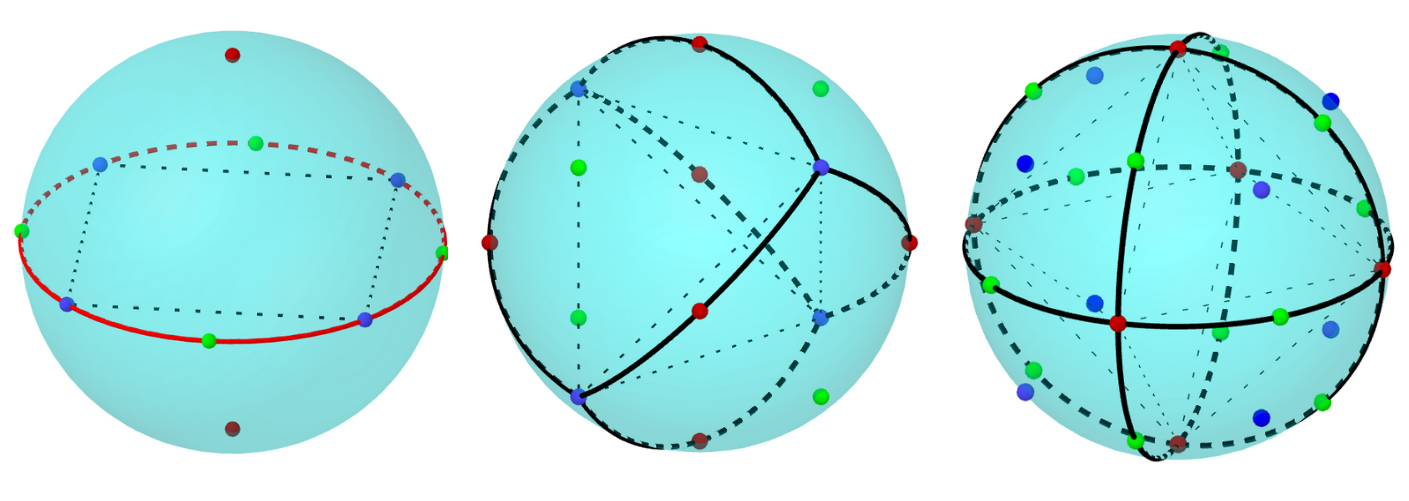}
\caption{Regular Polyhedra Inscribed in the Sphere}\label{DihedronFig}
\end{figure}

We identify the complex plane ${\mathbb C}$ as the equatorial plane of the unit sphere $S^2$ centered at the origin. Stereographic projection through the north pole of $S^2$ identifies it with the extended complex plane, through which $w:=x/y$ corresponds to a point on $S^2$.

For the binary cyclic group $C_l^*$, $x$ and $y$ are fundamental invariants, which imply that any monomial $x^iy^{n-i}$ 
is an invariant.

The fundamental invariants of the binary dihedral group $D_l^*$ are 
$$\medmath{\alpha_d\triangleq (x^d-y^d)/2 
,\quad\quad\quad ~~~~~~\beta_\triangleq (x^d+y^ d)/2
,\quad\quad\quad~~~~~~\gamma\triangleq xy.}$$
Their zeros (by setting $\omega:=x/y$) can be geometrically interpreted through the regular polygon structure shown in the upper-left sphere above, 
‌where the blue vertices‌ correspond to the zeros of $\alpha_d$, the green  edge midpoints‌ mark the zeros of $\beta_d$, while the red dots‌ at the centroids of the hemispherical faces represent the vanishing points of $\gamma$. 

For the binary tetrahedral group $A_4^*$, the fundamental invariants are 
\[\medmath{\Phi\triangleq x^4-2\sqrt{3}ix^2y^2+y^4, \quad \Psi\triangleq x^4+2\sqrt{3}ix^2y^2+y^4, \quad\Omega\triangleq xy(x^4-y^4).}\]
To see their zeros, we consider the tetrahedron shown in the middle sphere. 
The ‌blue vertices‌ correspond to the zeros of $\Phi$, the red edge midpoints‌ mark the zeros of $\Omega$, while the green dots‌ at the centroids of the four faces represent the vanishing points of $\Psi$.

For the binary octahedral group $S_4^\ast$,  
$\{\Omega,~ \Phi\Psi, ~ {\Phi^3+\Psi^3}\}$ 
constitutes its set of fundamental invariants. Their zeros are geometrically encoded by the regular octahedron in the upper-right sphere, where $\Omega$ vanishes at the six vertices (marked in red), $\Phi\Psi$ vanishes at the eight face centroids (in blue), while $\Phi^3+\Psi^3$ vanishes at the twelve edge midpoints (in green).

The fundamental invariants and their algebraic relations are summarized in Table 1. We do not list the invariants of $A_5^*$ as we do not need them in the sequel. See \cite[p.~42]{GaborToth2002} instead. 
\vskip -0.2cm
{\scriptsize\begin{table}[hbtp]
\begin{tabular}{|c|c|c|c|c|c|}
\hline
Platonic          & $G^\ast$ & $\iota_0$ & $\iota_1$ & $\iota_2$ & Relation\\ \hline
Dihedron & $D_d^\ast$ & $\alpha_d$ & $\beta_d$ & $\gamma$ & $\alpha_d^2-\beta_d^2+\gamma^d=0$     \\ \hline
Tetrahedron & $A_4^\ast $ & $\Phi$ & $\Omega$ & $\Psi$ & $\Phi^3+12\sqrt{3}i\Omega^2-\Psi^3=0$ \\ \hline
Octahedron & $S_4^\ast$ & $\Omega$ & $\frac{\Phi^3+\Psi^3}{2}$ & $\Phi\Psi$ & $108\Omega^4+(\frac{\Phi^3+\Psi^3}{2})^2-(\Phi\Psi)^3=0$ \\ \hline
\end{tabular}
\vspace{0.15cm}
\caption{Fundamental Invariants}\label{invariants and relation}
\end{table}}
\vskip -0.5cm
These polyhedral subgroups of $SU(2)$ are closely tied to rational branched Galois coverings onto ${\mathbb P}^1$. Indeed, 
let $M$ be a compact Riemann surface. Consider the branched covering $\varphi:M\rightarrow \mathbb{P}^1$ and its group of covering transformations of $G=\{\sigma\in \Aut(M)~|~\varphi\circ\sigma=\varphi\}$. We call $\varphi$ a branched \emph{Galois covering} if $|G|=\deg \varphi$.

When $M=\mathbb{P}^1$, which is to be enforced henceforth, the branched Galois coverings,  up to M\"obius transformation, were classified by Klein \cite{Klein1956} as given in Table \ref{Rational Galois Coverings} below. We will utilize the following simple lemma in section~\ref{sec-D2} through section~\ref{sec-S4}. 
\begin{lemma}{\label{linear system of D2}}
Let $G$ be any group from Table $2$ {\rm (}with $C_n$ excluded{\rm )} and $\varphi$ its associated Galois covering. Let $\mathrm{Nu}(\varphi)$ and $\mathrm{De}(\varphi)$ be the numerator and denominator of $\varphi$, respectively, and $\iota_1$ be the corresponding fundamental 
invariant listed in Table~$1$. 
Then, the zeros of the linear system $s_1\mathrm{Nu}(\varphi)+s_2\mathrm{De}(\varphi)$ {\rm (}$[s_1,s_2]\neq 0${\rm )} are $(1)$ zeros of $\mathrm{Nu}(\varphi)$ if $s_2=0$, $(2)$ zeros of $\mathrm{De}(\varphi)$ if $s_1=0$, $(3)$ zeros of $\iota_1^2$ if $s_1=-s_2$, or $(4)$ a principal orbit of $G$ if $s_1s_2(s_1+s_2)\neq 0$. \end{lemma}
\vskip -0.3cm
{\small\begin{table}[H]
\begin{tabular}{|c|c|c|c|c|c|}
\hline
$G$          & $C_n$ & $D_j$ & $A_4$ & $S_4$ & $A_5$\\ \hline
$\varphi([x,y])$ &   $\frac{x^n}{y^n}$    &  -$\frac{\alpha_j^2}{\gamma^j}$      & $\frac{\Phi^3}{\Psi^3}$      &  $\frac{(\Phi\Psi)^3}{108\Omega^4}$    & $\frac{\mathcal{H}^3}{1728\mathcal{I}^5}$  \\ \hline
\end{tabular}
\vspace{0.15cm}
\caption{Rational Branched Galois Coverings}\label{Rational Galois Coverings}
\vskip -0.25cm
\end{table}}
\begin{proposition}\label{key basis prop}
Let $G^*$ be a finite subgroup of $SU(2)$ with the irreducible representation on $V_1$. Let $f\in V_n$ be an invariant of $G^*$ with character $\rchi_f$, and $\Span\{g_1,g_2,\ldots,g_k\}\subset V_m$ be an irreducible representation of $G^\ast$ with character $\rchi^\prime$. Then we have the following conclusions.
\begin{enumerate}
\item[\rm{(1)}] $\Span\{\frac{\partial f}{\partial u},\frac{\partial f}{\partial v}\}\subset V_{n-1}$ is an irreducible representation of $G^\ast$ with character $\rchi_f\rchi_{V_1}$. In particular, $-\frac{\partial f}{\partial v}\otimes e_1+\frac{\partial f}{\partial u}\otimes e_2$ spans a $1$-dimensional invariant space with character $\chi_f$ in $V_n\otimes {\mathbb C}^2$. Likewise, $u\otimes e_1+v\otimes e_2$ spans a $1$-dimensional invariant subspace.

\item[\rm{(2)}] $\Span\{fg_1,fg_2,\ldots,fg_k\}\subset V_{n+m}$ is an irreducible representation of $G^\ast$ with character $\rchi_f\rchi^\prime$.
\end{enumerate}

\end{proposition}

\begin{proof}
(1). Note that the $1$st transvectant is $SL_2$-invariant; see \eqref{transvectant}. We have
\begin{small}$$
(A\cdot f,A\cdot g)_1=(\rchi_f(A)\,f,A\cdot g)_1=\rchi_f(A)\,(f,A\cdot g)_1.$$\end{small}
Thus, for any $A\in G^\ast$, by \eqref{V1 equi to natural action} we derive
\begin{small}\begin{equation}\label{symp}(A\cdot (f,u)_1,A\cdot (f,v)_1)=\chi_f(A)\, ((f,A\cdot u)_1,(f,A\cdot v)_1)=\chi_f(A)\, ((f,u)_1,(f,v)_1)\,(A^{-1})^T.\end{equation}\end{small}
Moreover, by \eqref{transvectant}, $(f,u)_1=-\frac{1}{n}\frac{\partial f}{\partial v},~(f,v)_1=\frac{1}{n}\frac{\partial f}{\partial u}$, so that the character of the engaged representation is $\rchi_f\rchi_{V_1}$. It is irreducible since
\begin{small}$$(\rchi_f\rchi_{V_1}|\rchi_f\rchi_{V_1})=\frac{1}{|G^\ast|}\sum\limits_{A\in G^\ast}|\rchi_f(A)\rchi_{V_1}(A)|^2=\frac{1}{|G^\ast|}\sum\limits_{A\in G^\ast}|\rchi_{V_1}(A)|^2=1.$$\end{small}
Here we also use the fact the character of a $1$-dimensional complex representation takes values in $S^1$. The second sentence is a consequence of \eqref{symp} and \eqref{pqn}, and the last follows from \eqref{V1 equi to natural action}.

The statement in (2) is due to that the action of $SL_2$ on the algebra of all binary forms is an algebra homomorphism; that the representation is irreducible is verified similarly as in (1). 
\end{proof}


\subsection{The Fano 3-fold $\bf {\bf \mathcal{V}_5}$}\label{muf3f}~

The Fano 3-fold $\bf \mathcal{V}_5$ has been studied by Mukai and Umemura in \cite{Mukai-Umemura1983} from the viewpoint of algebraic group actions. By considering the action of $PSL_2$ on ${\mathbb P}(V_6)$, one sees that $\bf \mathcal{V}_5$ is the closure of $PSL_2 \cdot uv(u^4-v^4)$ (for a realization of $\bf \mathcal{V}_5$ as a generic linear section of $G(2,5)$, see \cite[Sec 3.1]{Chi-Xie-Xu2024}). In the same paper, they obtained the following orbit decomposition structure on $\bf \mathcal {V}_5$. 

\begin{theorem}\label{OrbitDecomp}
{\small\begin{align*}
\aligned
{\bf \mathcal{V}_5} =\overline{PSL_2 \cdot uv(u^4-v^4)}
 =PSL_2 \cdot uv(u^4-v^4) \sqcup PSL_2 \cdot u^5v \sqcup PSL_2 \cdot u^6.
\endaligned
\end{align*}}
\end{theorem}

Meanwhile, through the invariants and covariants of the binary sextic (see \eqref{transvectant}), the above orbits have another $SL_2$-invariant characterization. 
\begin{proposition}\label{orbits defined by transvectant}
Given $f=\sum_{i=0}^{6}\sqrt{\tbinom{6}{i}}a_i u^{6-i}v^i$ defining $[f]\in \mathbb{P}(V_6)$, we have
\begin{enumerate}
\item[\rm{(1)}] $[f]$ lies in ${\bf {\mathcal V}_5}
$ if and only if the $4$-th transvectant $(f,f)_4=0$,

\item[\rm{(2)}] $[f]$ lies in the closed $2$-dim orbit $\overline{PSL_2 \cdot u^5v}$ if and only if the $4$-th and $6$-th transvectants $(f,f)_4$ and $(f,f)_6$ vanish, and

\item[\rm{(3)}] $[f]$ lies in the $1$-dim orbit $PSL_2 \cdot u^6$ if and only if the $2$nd transvectant $(f,f)_2=0$.
\end{enumerate}

\end{proposition}
For later purposes, we quote the following well-known calculation 
\begin{equation}\label{eQ}{\small
(f,f)_6=
2a_0a_6 - 2a_1a_5 + 2a_2a_4 - a_3^2,}
\end{equation}
and consider the $5$-quadric $Q_5$ defined by 
{\small\begin{equation}\label{5-quadric}
Q_5\triangleq\{[f]\in {\mathbb P}(V_6): (f,f)_6=0\}.
\end{equation}}

 In the above orbit decomposition, the open orbit $PSL_2 \cdot uv(u^4-v^4)$ of dimension $3$ is parameterized as {\rm (}under the basis $\{\sqrt{\binom{6}{i}}u^{6-i}v^i\}_{i=0}^6${\rm )}
{\small\begin{equation}\label{f}
\aligned
&f_1: PSL_2\mapsto {\mathbb P}(V_6), \quad
[\begin{pmatrix}
a & b\\
c & d
\end{pmatrix}] \mapsto [\begin{pmatrix}
a & b\\
c & d
\end{pmatrix} \cdot uv(u^4-v^4)]=[a_0:a_1:\cdots:a_6],\\
&a_0\triangleq -\sqrt {6}{d}^{5}c+\sqrt {6}d{c}^{5},\quad a_1\triangleq {d}^{4} \left( ad+5\,bc \right) -c^4(5\,ad+bc),\\
&a_2\triangleq -b{d}^{3} \left( ad+2\,bc \right) \sqrt {10}+a{c}^{3} \left( 2\,ad+bc
 \right) \sqrt {10}
, \\&a_3\triangleq {b}^{2}{d}^{2} \left( ad+bc \right) \sqrt {30}-{a}^{2}{c}^{2} \left( a
d+bc \right) \sqrt {30}
,\\
&a_4\triangleq -{b}^{3}d \left( 2\,ad+bc \right) \sqrt {10}+{a}^{3}c \left( ad+2\,bc
 \right) \sqrt {10}
,\\&a_5\triangleq b^4(5\,ad+bc)-{a}^{4} \left( ad+5\,bc \right) ,\quad
a_6\triangleq -\sqrt {6}{b}^{5}a+\sqrt {6}b{a}^{5}.
\endaligned
\end{equation}}

Furthermore, the isotropy group $G_0$ of the open orbit 
 is the projective binary octahedral group of order $24$, isomorphic to $S_4$, consisting of the following elements 
{\scriptsize
\begin{align}
\label{isotropy group}
\begin{split}
T_{1,k}:&=\begin{pmatrix}
  \xi_k & 0 \\
  0 & 1/\xi_k \\
\end{pmatrix},~~T_{2,k}:=\begin{pmatrix}
  0 & \xi_k \\
  -\frac{1}{\xi_k} & 0 \\
\end{pmatrix},~~T_{3,k}:=
1/\sqrt{2}\cdot\begin{pmatrix}
  1/\xi_k & -1/\xi_k \\
  \xi_k & \xi_k \\
\end{pmatrix},\quad\quad (\xi_k\triangleq e^{2k\pi\sqrt{-1}/8},~k
= 0,1,\ldots,3),\\
T_{4,k}:&=\frac{1}{\sqrt{2}}\cdot\begin{pmatrix}
  \sqrt{-1}/\xi_k & -1/\xi_k \\
  \xi_k & -\sqrt{-1}\xi_k \\
\end{pmatrix},~~ T_{5,k}:=\frac{1}{\sqrt{2}}\cdot\begin{pmatrix}
  -1/\xi_k & -1/\xi_k \\
  \xi_k & -\xi_k \\
\end{pmatrix},~~ T_{6,k}:=\frac{1}{\sqrt{2}}\cdot\begin{pmatrix}
  -\sqrt{-1}/\xi_k & -1/\xi_k \\
  \xi_k & \sqrt{-1}\xi_k\\
\end{pmatrix}.
\end{split}
\end{align}}
\!\!Note that $G_0$ is generated by $T_{1,1}$ and $T_{3,0}$ while $S_4$ is generated by the cycles $(1234)$ and $(2134)$, and they correspondingly share the same generator structures. 
Therefore, we obtain an isomorphism $J$ of $G_0$ with $S_4$ given in the table beside Figure 2. {\rm (}Note also that $G_0$ is invariant under taking matrix transpose.{\rm )}

\begin{proposition}\label{correspondence}
$S_4$ has $5$ conjugacy classes. We point out their geometric meanings with respect to vertices, edges, and faces of a regular octahedron {\rm (}see Figure $2$ for a pictorial description{\rm )}.
\begin{enumerate}
\item[\rm{(1)}] $\Id$, order $1$.
\item[\rm{(2)}] $[(12)]$, order $2$, consisting of $6$ half turns around the axes through the midpoint of opposite edges.
\item[\rm{(3)}] $[(12)(34)]$, order $2$, consisting of $3$ half turns around the axis through the opposite vertices.
\item[\rm{(4)}] $[(123)]$, order $3$, consisting of $8$ one-third turns around the axes through the centers of the opposite faces.
\item[\rm{(5)}] $[(1234)]$, order $4$, consisting of $6$ quarter turns around the axes through the opposite vertices.  
\end{enumerate} 
\end{proposition}

\begin{figure}[htbp]
\begin{minipage}{0.3\textwidth}
\includegraphics[width=\linewidth]{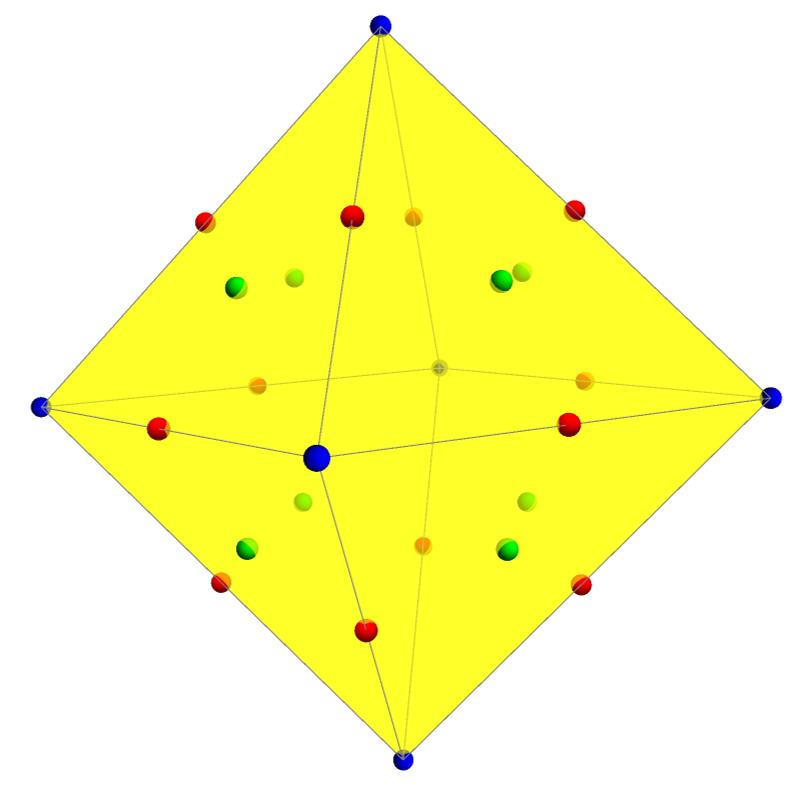}
\caption{Octahedron}\label{Octahedron}
\end{minipage}
\hfill
\begin{minipage}{0.6\textwidth}
\centering
\scriptsize
\begin{tabular}{|l|l|l|l|}
\hline
$g\in G_0$ & $J(g)$ & $g\in G_0$ & $J(g)$ \\ \hline
$T_{1,0}=\Id$  &  $\Id$   & $T_{1,1}$  & $(1234)$    \\ \hline
$T_{1,2}=T_{1,1}^2 $ &  $(13)(24)$   & $T_{1,3}=T_{1,1}^3$  &   $(1432)$    \\ \hline
$T_{2,0}=-T_{3,0}^2$  &  $(23)(14)$   & $T_{2,1}=-T_{1,1}T_{3,0}^2$  & $(24)$    \\ \hline
$T_{2,2}=-T_{1,1}^2T_{3,0}^2$  &  $(12)(34)$   & $T_{2,3}=-T_{1,1}^3T_{3,0}^2$  &   $(13)$    \\ \hline
$T_{3,0}$  &  $(2134)$   & $T_{3,1}=-T_{1,1}^3T_{3,0}$  &   $(124)$    \\ \hline
$T_{3,2}=-T_{1,1}^2T_{3,0}$  &  $(23)$   & $T_{3,3}=-T_{1,1}T_{3,0}$  &   $(143)$    \\ \hline
$T_{4,0}=T_{1,1}T_{3,0}T_{1,1}$  &  $(12)$   & $T_{4,1}=T_{3,0}T_{1,1}$  &   $(243)$    \\ \hline
$T_{4,2}=-T_{1,1}^3T_{3,0}T_{1,1}$  &  $(1423)$   & $T_{4,3}=-T_{1,1}^2T_{3,0}T_{1,1}$  &   $(134)$    \\ \hline
$T_{5,0}=T_{3,0}^3$  &  $(2431)$   & $T_{5,1}=-T_{1,1}^3T_{3,0}^3$  &   $(234)$    \\ \hline
$T_{5,2}=-T_{1,1}^2T_{3,0}^3$  &  $(14)$   & $T_{5,3}=-T_{1,1}T_{3,0}^3$  &   $(132)$    \\ \hline
$T_{6,0}=T_{1,1}T_{3,0}^3T_{1,1}$  &  $(34)$   & $T_{6,1}=T_{3,0}^3T_{1,1}$  &   $(142)$    \\ \hline
$T_{6,2}=-T_{1,1}^3T_{3,0}^3T_{1,1}$  &  $(1324)$   & $T_{6,3}=-T_{1,1}^2T_{3,0}^3T_{1,1}$  &   $(123)$    \\ \hline
\end{tabular}
\end{minipage}
\end{figure}

\subsection{Galois covering of sextic curves in $\bf \mathcal{V}_5$}\label{sec-para}

To understand better the construction and structure of sextic curves in the Fano $3$-fold $\bf \mathcal{V}_5$, we look at it from the Galois point of view. 

Let $F:\mathbb{P}^1\rightarrow {\bf \mathcal{V}_5 }$ be a sextic curve that does not lie in the closed $2$-dimensional orbit $\overline{PSL_2 \cdot u^5v}$. Then we identify the projectivization of the space of $2\times 2$ nonzero (complex) matrices with ${\mathbb C}P^3$ by 
\begin{small}$$\iota: [\begin{pmatrix}a&b\\c&d\end{pmatrix}] \mapsto [a:b:c:d].$$\end{small}
Via $\iota$, the subset of $2\times 2$ matrices of zero determinant defines the following $PSL_2$-invariant hyperquadric $Q_2$ of dimension $2$,
{\small
\begin{equation}\label{2q}
Q_2\triangleq \{[a:b:c:d]\in{\mathbb C}P^3~|~ ad-bc=0\}.
\end{equation}
}
\!\!Note that we can identify $PSL_2$ with $\mathbb{P}^3 \setminus Q_2$. 

\begin{theorem}\cite[Theorem 5.2]{Chi-Xie-Xu2024}\label{lift of curve not in 2dim orbit}
Let $F:\mathbb{P}^1\rightarrow {\bf {\mathcal V}_5}\subset G(2,5)$ be a sextic curve. If $F$ does not lie in the closed $2$-dimensional orbit $\overline{PSL_2\cdot u^5v}$, then there exists a compact connected Riemann surface $g:M\rightarrow \mathbb{P}^3$ covering $F$ as in the following commutative diagram
{\small\begin{equation}\label{diagram2}
\begin{tikzcd}
M\arrow{r}{g} \arrow{d}{\varphi}&\mathbb{P}^3 \arrow[dashed]{d}{f_1}\\
\mathbb{P}^1  \arrow{r}{F} & \mathbb{P}^6
\end{tikzcd}
\end{equation}}

Moreover, $\varphi:M\rightarrow \mathbb{P}^1$ is a branched Galois covering, and the group of covering transformations $G\triangleq\{\sigma\in \Aut(M)~|~\varphi\circ\sigma=\varphi\}$ is a subgroup of $S_4$ isomorphic to the isotropy group at $uv(u^4-v^4)$ given in  \eqref{isotropy group}.
\end{theorem}


Furthermore, the realization of the group of covering transformations $G=\{\sigma\in \Aut(M)~|~\varphi\circ\sigma=\varphi\}$ as a subgroup $\widetilde{G}$ of $S_4^*$ is given by 
\begin{small}$$G\rightarrow \widetilde{G}\subset S_4,\quad \sigma\mapsto C_\sigma\triangleq g(\sigma(q))^{-1} g(q)\in SL_2,~~\forall q\in M,$$\end{small}
where $C_\sigma$ is well-defined due to that $M$ is connected and the isotropy group $S_4$ is finite.


For the map $g=\begin{pmatrix}
a & b \\
c & d
\end{pmatrix}:M\rightarrow \mathbb{P}^3$ as in Theorem \ref{lift of curve not in 2dim orbit}, to be referred to henceforth as a {\bf Galois lift} of $F$, we associate it with two meromorphic functions
{\small\begin{equation}\label{eq-mathrmw}\mathrm{x}:=c/a,~~~\mathrm{w}:= b/a.
\end{equation}}
We employ the geometry of the octahedron to study the Galois covering $\varphi$.

The quadric $Q_2$ defined in \eqref{2q} is a saddle surface in ${\mathbb P}^3$ isomorphic to $\mathbb{P}^1\times \mathbb{P}^1$ by the parametrization {\scriptsize$\begin{pmatrix}
1 & \mathrm{w}\\
\mathrm{x} & \mathrm{w}\mathrm{x}
\end{pmatrix}$}, where each pair $(\mathrm{w},\mathrm{x})$ determines uniquely a point $p\in Q_2$, through which there passes a unique $\mathrm{w}$-ruling {\small\begin{equation}\label{ruling-w}L_{\mathrm{w}}\triangleq\{(\mathrm{w},\mathrm{x})~|~\mathrm{x}\in\mathbb{P}^1\}\end{equation}}

Referring to Figure 2, projected onto the enclosing sphere from its center followed by the stereographic projection, 
the above $26$ centers (i.e., the 6 vertices, 12 edge centers, and 8 face centers) turn out to be the roots of the polynomial equation 
\begin{equation}\label{26}
(\mathrm{w}^5-\mathrm{w})(\mathrm{w}^8+14\mathrm{w}^4+1)(\mathrm{w}^{12}-33\mathrm{w}^8-33\mathrm{w}^4+1)=0,
\end{equation}
where we also count $\mathrm{w}=\infty$ as a root. We point out that the above three polynomial factors take exactly the 26 centers as their roots, respectively. Furthermore, via this correspondence, the symmetric group $S_4$ (isomorphic to the projective binary octahedral group) acts on the $w$-rulings of $Q_2$, and is exactly the action of the isotropic group in  \eqref{isotropy group} acting on  {\scriptsize$\begin{pmatrix}
1 & \mathrm{w}\\
\mathrm{x} & \mathrm{w}\mathrm{x}
\end{pmatrix}$}
by right matrix multiplication, where the $26$ centers are related to the $26$ distinct eigenvectors of these isotropy matrices. 

Now, we introduce two important divisors to study the Galois covering~\eqref{diagram2}. In the following, we denote the degree of the covering $\varphi:M\rightarrow \mathbb{P}^1$ by $d$, and the degree of the Galois lift $g:M\rightarrow \mathbb{P}^3$ by $k$.

Let ${\mathcal Q}$ be the intersection divisor defined by $g$ and the quadric $Q_2$. By  Bezout's theorem, we have $\deg({\mathcal Q})=2k$. 

We say that a hypersurface $G=0$ of degree $t$ in ${\mathbb P}^6$ is {\bf \emph{generic}} if it does not contain the curve $F$ and it cuts out a divisor on $F$ whose support lives in $V=\mathbb{P}^1\setminus F^{-1}(Q_5)$. Projective normality  \cite[pp.230-231]{Miranda1995}) of the rational normal curve $F$ warrants the existence of generic hypersurfaces. 

 A generic hyperplane $H=\sum_{i=0}^6 c_i\, a_i=0$ in ${\mathbb P}^6$ with coordinates $[a_0: \cdots:a_6]$ cuts $\gamma=F(M)$ in a divisor $D_H$ of degree $6$ whose support lies in $V$, while $f_1$ pulls the hyperplane $H=0$ back to a hypersurface of degree $6$ in ${\mathbb P}^3$ that cuts $g$ in a divisor ${\mathcal D}$ of degree $6k$ by Bezout's theorem. Since $\varphi|_U$ is a covering map of degree $d$ over $V$, the divisor ${\mathcal D}$ contains the pullback divisor ${\mathcal D}_0\triangleq \varphi^\ast(D_H)$ of degree $6d$. Define their difference by ${\mathcal F}$, 
{\small\begin{equation}\label{computationsInterDiv}
{\mathcal F}\triangleq {\mathcal D}-{\mathcal D}_0.
\end{equation}}
\!\!We denote the support of a divisor $\bf D$ by $\supp {\bf D}$. In particular, \eqref{computationsInterDiv} infers
{\small\begin{equation}\label{degree-difference}
\deg({\mathcal F})= 6(k-d).
\end{equation}}
\!\!$\mathcal{F}$ is the fixed part of the intersection divisors of $g$ with the hypersurfaces of degree $6$ in ${\mathbb P}^3$ obtained by the coordinates of $f_1$ given in \eqref{f}, namely, 
{\small\begin{equation}\label{min}
\mathcal{F}=\min\limits_{0\leq i\leq 6}\{g^\ast(a_i\circ f_1)\}.
\end{equation}}

\noindent Note, by \eqref{min}, that $p\in {\text Supp }\, {\mathcal F}$ is determined by setting $g^\ast(a_i\circ f_1)=0,0\leq i\leq 6$, which deduces that $p$ lies in one of the six ${\mathrm w}$-rulings (see \eqref{ruling-w}).
\begin{equation}\label{free}
 \{L_\mathrm{w}\;|\: \mathrm{w}=0, \infty, ~\text{or}~ \mathrm{w}^4=1\}   
 \end{equation}
  in ${\mathbb P}^3$, and vice versa. We refer to these six lines as the {\bf free lines}, to be labeled as ${\mathcal L}_1, {\mathcal L_2},{\mathcal L}_j$ with $\mathrm{w}=(\sqrt{-1})^{j-3}, 3\leq j\leq 6,$
  in order.


 \begin{proposition}\cite[Proposition 5.1]{Chi-Xie-Xu2024}\label{charc on 2dim orbit}
 Let $F:\mathbb{P}^1\rightarrow {\bf {\mathcal V}_5}\subset \mathbb{P}^6$ be a sextic curve not lying in the closed $2$-dimensional orbit $\overline{PSL_2\cdot u^5v}$, and $g:M\rightarrow \mathbb{P}^3$ be the Galois lift of $F$ in the commutative diagram \eqref{diagram2}. Then
$\mathcal{F}\leq \mathcal{Q}.$
Moreover, for any given $p\in \supp \mathcal Q$, 
\begin{enumerate}
\item[\rm{(1)}] if $\mathrm{w}(p)$ is not associated with any of the $6$ vertices, i.e., does not satisfy the first factor of \eqref{26}, then $\ord_p(\mathcal{F})=0$ and $f_1\circ g (p)$ lies in the $1$-dimensional orbit $PSL_2\cdot u^6$, and
\item[\rm{(2)}] if $\mathrm{w}(p)$ is associated with one of the $6$ vertices, then $\ord_p(\mathcal{F})>0$ and $f_1\circ g (p)$ lies in the $1$-dimensional 
orbit if and only if $\ord_p(\mathcal{F})<\ord_p(\mathcal{Q})$. 
\end{enumerate}
 \end{proposition}
\begin{corollary}\cite[Corollary 5.1]{Chi-Xie-Xu2024}\label{Cor4.1}
Assume the same setting as in Proposition {\rm \ref{charc on 2dim orbit}}. We have 
$\deg \varphi\leq \deg g\leq \frac{3}{2}\deg \varphi.$
Moreover, $\deg g= \deg \varphi$ if and only if $\mathcal{F}=0$, while $\deg g= \frac{3}{2}\deg \varphi$ if and only if $\mathcal{F}=\mathcal{Q}$.
\end{corollary} 
 

From the proof of the above proposition, we obtain the following local criterion to be useful in the following construction of concrete examples.

\begin{proposition}\label{Critical Prop Free lines}
Let $g(z):\Delta \rightarrow \mathbb{P}^3$, given by {\scriptsize$z\mapsto \begin{bmatrix}
a(z) & b(z) \\
c(z) & d(z)
\end{bmatrix}$}, be a holomorphic map from the unit disk $\Delta$ centered at $0$. Assume that {\scriptsize$g(0)=\begin{pmatrix}
a_0 & 0 \\
c_0 & 0
\end{pmatrix}$},
and 
\begin{small}$$b(z)=b_nz^n+o(z^n),~~d(z)=d_nz^n+o(z^n),$$\end{small}
where $(a_0,c_0)\neq (0,0),~(b_n,d_n)\neq (0,0)$ for some $n\geq 1$. Then $\mathcal{F}(0)=n$. 
Moreover, $\mathcal{Q}(0)=\mathcal{F}(0)$ if and only if {\scriptsize$\det \begin{pmatrix}
a_0 & b_n \\
c_0 & d_n
\end{pmatrix}\neq  0$}.
\end{proposition}

\subsection{Generally Ramified Family vs. Exceptional Transversal Family}
By Theorem 4.1 in \cite{Chi-Xie-Xu2024}, a sextic curve $F$ in $\bf \mathcal{V}_5$ is ramified (in the sense of harmonic sequences) at a point $q$  if and only if the tangent line of $F$ at $q$ lies in $\bf \mathcal{V}_5$. An important class of lines in $\bf \mathcal{V}_5$ is given by the rulings of the tangent developable surface $\bf{S}$ (i.e., the closed $2$-dimensional $PSL_2$-orbit), which are exactly the tangent lines of the $1$-dimensional orbit $PSL_2\cdot u^6$; in particular, that there is a unique line though $q$ in the $1$-dimensional orbit implies that $F$ is ramified at $q$ if and only if $F$ is tangent to the $1$-dimensional orbit at $q$.  Our investigation of various examples and Galois analysis have prompted the following definition. 

\begin{definition} \label{def-generally}
We say that a sextic curve $F$ in $\bf {\mathcal V}_5$ is in the {\bf generally ramified family} if $F$ is ramified at the $1$-dimensional orbit $PSL_2\cdot u^6$ somewhere;  otherwise we say that $F$ lies in the {\bf{exceptional transversal family}}.
\end{definition}

In terms of intersection divisors, we obtained necessary and sufficient conditions for the curve $F$ to belong to the generally ramified family. Define the multiplicity of $\varphi$ at $p$ to be $\varphi: s\mapsto s^{\mult_p(\varphi)}$ for a local uniformizing parameter $s$ with $s(p)=0.$ 
\begin{theorem}\label{thm}\cite[Theorem 5.3]{Chi-Xie-Xu2024}
Let $F:\mathbb{P}^1\rightarrow {\bf {\mathcal V}_5}$ 
be a sextic curve which is not contained in the closed $2$-dimensional $PSL_2$-orbit. 
$F$ belongs to the generally ramified family if and only if 
\begin{enumerate}
\item[\rm{(1)}] there exists a point $p\in \supp \mathcal Q \setminus \supp \mathcal F$ such that either $\mult_{p}(\varphi)=1$, or $\ord_p(\mathcal{Q})\geq 2$ and $\mathrm{w}(p)$ is associated with one of the $12$ edge centers and $8$ face centers, or 
\item[\rm{(2)}] there exists a point $p\in M$ such that $0<\ord_p(\mathcal{F})<\ord_p({\mathcal Q})$. 
\end{enumerate}
\end{theorem}

One readily observes that $\mult_{p}(\varphi)=1$ whenever $\mathrm{w}(p)$ is not associated with any of the $26$ centers of the octahedron. Consequently, the ramified points of $\varphi$ are associated with these $26$ centers. In particular, if $\mult_{p}(\varphi)=4$, then $p\in \supp ~\mathcal{F}$.

\section{The Idea of  Construction}

In this section, we introduce our main idea for the correspondence between the classification of rational Galois lifts and the irreducible representations of degree $1$. Because we employ binary forms in the construction, the binary subgroups $G^\ast$ in $SU(2)$, rather than the subgroups of $SO(3)$, are to be invoked.

\begin{theorem}\label{main construction thm}
Let $\varphi:\mathbb{P}^1\rightarrow \mathbb{P}^1$ be a branched Galois covering of degree $d$ with the covering  group $G\subset SO(3)$. 
\begin{enumerate}
\item[\rm{(1)}] Let $F:P^1\mapsto \mathcal{V}_5$ be a sextic curve with rational Galois lift $g$ 
of degree $n$ that 
satisfies the following commutative diagram 
{\small\begin{equation}\label{diagram3}
\begin{tikzcd}
\mathbb{P}^1 \arrow{r}{g} \arrow{d}{\varphi}&\mathbb{P}^3 \arrow[dashed]{d}{f_1}\\
\mathbb{P}^1  \arrow{r}{F} & \mathbb{P}^6
\end{tikzcd}
\end{equation}}
\!\!as generally given in Theorem {\rm \ref{lift of curve not in 2dim orbit}}. Then under the natural action of $G^\ast$, the lines spanned by the row vectors of $g$ 
correspond to a pair of isomorphic yet mutually independent 1-dimensional irreducible representations within $V_n\otimes \mathbb{C}^2$. 
\item[\rm{(2)}] Conversely, let $kW$ ($k\geq 2$) denote the direct sum of $1$-dimensional irreducible representations of $G^\ast$ within $V_n\otimes \mathbb{C}^2$, each isomorphic to the representation $W$. 
Then any two independent vectors in $kW$ induce a unique curve $F:\mathbb{P}^1\rightarrow \mathbb{P}^6$ satisfying {\rm \ref{diagram3}}. 
\end{enumerate}
 \end{theorem}
(1) The idea is that for any $\sigma\in G$, and any $[x,y]\in \mathbb{P}^1$, we must have $g([x,y])$ and $g(\sigma([x,y]))$ lie in the same fiber of $f_1$ due to \eqref{diagram3}. Hence there is a unique matrix $A\in S_4^*$ such that, by \eqref{AB},
\begin{small}$$g(\sigma([x,y])) A^{T}=g([x,y]),~~\forall [x,y]\in \mathbb{P}^1,$$\end{small}
as $S_4^*$ is discrete and closed under transposing.

 Then, our problem of constructing a sextic curve in $\bf {\mathcal V}_6$ admitting a rational Galois covering is to find isomorphisms $\phi_1:G^\ast \rightarrow SU(2)$, and $\phi_2:G^\ast \rightarrow S_4^\ast \subseteq SU(2)$, such that for every $t\in G^\ast $, we have 
\begin{equation}\label{eigenspace problem}
\medmath{(\phi_1(t)\cdot g)(x,y)\,\phi_2(t)^{T}=\lambda(t)\, g(x,y),}
\end{equation}
for some complex number $\lambda(t)$, so that \eqref{diagram3} is validated via projectivization. In particular, the row vectors $(a,b)$ and $(c,d)$ of $g$ are both eigenvectors within $V_n\otimes \mathbb{C}^2$ of the $G^\ast$-action 
$$\medmath{
t\cdot \big(p(x,y), q(x,y) \big)=\big( (\phi_1(t)\cdot p)(x,y), (\phi_1(t)\cdot q)(x,y)\big) \,\phi_2(t)^{T},\quad \forall t\in G^\ast,
}$$
where $\lambda(t)$ is the character function of the two isomorphic and independent $1$-dimensional irreducible representations generated by the vectors $(a,b)$ and $(c,d)$, respectively.

(2) The converse is then also made clear accordingly. 

The Galoisness of $\varphi$ then dictates that 
there be a unique holomorphic map $F:\mathbb{P}^1\rightarrow \mathcal{V}_6,~z\mapsto [a_0(z),\ldots,a_6(z)]$ (\cite[p.~78, Thm 3.4]{Miranda1995}), such that the diagram \eqref{diagram3} commutes. 

In concrete computations, we know that the Galois covering $\varphi$ is given by $z=\epsilon(x,y)/\delta(x,y)$, where $\epsilon,~\delta$ are two invariants of $G^\ast$, which enables us to solve a system of linear equations
$(a_0(\frac{\epsilon}{\delta}),\ldots,a_6(\frac{\epsilon}{\delta}))\delta^6=f_1\circ g([x,y])$
to find the polynomials $a_0,\ldots,a_6$ (another method is to use the Normal Form in the Groebner basis theory).
 


However, we caution that the constructed curves $F$ downstairs may not be linearly full in $\mathbb{P}^6$ to form a subvariety of the parametrized space.
Moreover, $g$ must at the same time satisfy the required constraints on the free divisor $\mathcal{F}$ to make sure that the curve $F$ downstairs is sextic, to be detailed in the following sections.


In passing, 
we remark that since the symmetry group of the Fano $\mathcal{V}_5$ is $PSL_2$, we define two holomorphic curves $g,~\widetilde{g}:\mathbb{P}^1\rightarrow \mathbb{P}^3$ to be {\bf $PSL_2$-\emph{equivalent}} if there exist $A\in SL_2$ and $B\in S_4^\ast$, such that 
$AgB^T=\widetilde{g}.$
The actions of $SL_2$ on the left and $S_4^\ast$ on the right manifest geometrically by observing that any two bases of a $2$-plane in $kW$ said in item (2) of Theorem \ref{main construction thm} differ by an $A$, while any two conjugate representations of $G^*\subset S_4^*$ differ by a $B$. 

\begin{remark}
If we work in $V_n\otimes V_1$ instead of $V_n\otimes {\mathbb C}^2=V_n\otimes (V_1)^*$, then the eigenvalue problem \eqref{eigenspace problem} is changed to 
{\small\[\medmath{(\phi_1(t)\cdot \tilde{g})(x,y)\,\phi_2(t)^{-1}=\lambda(t)\, \tilde{g}(x,y).}\]}
Moreover, $\tilde{g}(x,y)$ is equivalent to $g(x,y)$ in the convention in Theorem \eqref{main construction thm} by $g(x,y)\triangleq\tilde{g}(x,y)\cdot J$, where {\scriptsize$J\triangleq \begin{pmatrix}0&-1\\1&0\end{pmatrix}=T_{2,4}$}. Therefore, they are $PSL_2$-equivalent and give the same curves downstairs.

\end{remark}

Thanks to Theorem \ref{main construction thm} and the equivalence property, 
the problem of constructing a sextic curve in $\bf {\mathcal V}_5$ admitting a rational Galois covering is transformed to finding a $2$-plane that sits in a $k$-copy of an $1$-dimensional irreducible representation in $V_n\otimes \mathbb{C}^2$ with $k\geq 2$.

We outline the approach before completing the classification as follows. 
Given a covering $\varphi$ with its associated Galois group $G$, we first identify a Galois covering $G^*$ within $S_4^*$, and set  $\phi_1$ on the left-hand side of \eqref{eigenspace problem} as the identity map from $G^*$ to $G^*$. Determining $\phi_2$ on the right hand side requires more careful analysis, where all conjugacy classes of $G^*$ in $S_4^*$ need to be enumerated to account for outer automorphisms. Next, we stratify by $\deg g \in \mathbb{Z}^+$, where the degree range follows from Corollary~\ref{Cor4.1} combined with the $\mathcal{Q}$ and $\mathcal{F}$ terms at the  ramified points of $\varphi$. We then compute the decomposition of $V_n \otimes \mathbb{C}^2$ for each candidate degree $\deg g = n$. For non-cyclic cases, our analysis reduces to the $2$-dimensional irreducible representations in the decomposition of $V_n$, whose explicit bases are constructible via invariants using Proposition~\ref{key basis prop}. Finally, we obtain all possible $k$-copies ($k \geq 2$) of a 1-dimensional irreducible representation $W$ in $V_n \otimes \mathbb{C}^2$. The Galois lift $g$ then admits the parameterization $g = M(\xi_1, \xi_2)$,
where $M \in G(2,k)$, and $\xi_1$, $\xi_2$ are column vectors of degree-$n$ polynomials constructed from the explicit basis obtained in the previous step. The constraints on $M$ are determined by studying $\det(g)$ and the divisor $\mathcal{F}$.

\section{Cyclic $C_2$ of order $2$}

In \eqref{diagram3}, assume that $\deg \varphi=2$. Up to M\"obius transformations, by Table \ref{Rational Galois Coverings}, $\varphi:\mathbb{P}^1\rightarrow \mathbb{P}^1$ is given by
 $w\mapsto z=w^2$. The associated Galois group $C_2$ 
consists of two elements generated by 
{\small\begin{equation}\label{cyclic2}
\sigma: \mathbb{P}^1\rightarrow \mathbb{P}^1,\quad\quad w\mapsto -w.
\end{equation}}

\begin{theorem}
In the case of $C_2$, let $g([x,y]):\mathbb{P}^1\rightarrow \mathbb{P}^3$ be a holomorphic curve satisfying diagram \eqref{diagram3}. Then up to M\"obius transformations on $\mathbb{P}^1$ and the $PSL_2$-equivalence, we have $\deg g = 2$ or $3$. Moreover, the following classification holds.
\begin{enumerate}
\item[\rm{(1)}] If  $\deg g=2$, then
\begin{small}$${g=\begin{pmatrix}
1 & 0 & a_3 \\
0 & 1 & b_3
\end{pmatrix}\begin{pmatrix}
x^2 & xy & y^2\\
-\xi_1 x^2 & \xi_1xy & -\xi_1 y^2
\end{pmatrix}^T,}$$\end{small} where $\xi_1=\exp(\frac{\pi i}{4})$. Such curves form a $\mathbb{P}^1$-parameterized family with $[a_3,b_3]$ as the parameter. The sextic curve $F$ downstairs determined by $g$ in \eqref{diagram3} is  generally ramified.

\item[\rm{(2)}] If $\deg g=3$, then \begin{small}$$g=\begin{pmatrix}
1 & 0 & a_3 & a_4 \\
0 & 1 & b_3 & b_4
\end{pmatrix}\begin{pmatrix}
x^3 & 0 & xy^2 & 0\\
0 & x^2y & 0 & y^3
\end{pmatrix}^T,$$\end{small} where 
{\scriptsize$\begin{pmatrix}
1 & 0 & a_3 & a_4 \\
0 & 1 & b_3 & b_4
\end{pmatrix}\in G(2,4)$} is the unique plane in $\mathbb{C}^4$ perpendicular to $(1,-1,1,-1)$ and $(i{t}^3,-{t}^2,i{t},-1)$ under the usual bilinear inner product. Such curves constitute a $\mathbb{C}\setminus\{0\}$-parameterized family with $t$ as the parameter. The sextic curve $F$ downstairs determined by $g$ in \eqref{diagram3} belongs to the exceptional transversal family. 
\end{enumerate}
\end{theorem}
\begin{proof}
It follows from  Corollary \ref{Cor4.1} that $2\leq \deg(g)\leq 3$. 
Corresponding to  \ref{cyclic2}, we choose $r=T_{1,2}$ given in \eqref{isotropy group} generating $C_4$ in $S_4^\ast$ and double covering $\sigma$. This gives $\phi_1$ on the left-hand side of \eqref{eigenspace problem}. To determine $\phi_2$ on the right-hand side, observe that the conjugacy classes of elements of order 
$4$ in $S_4^\ast $ are represented by $T_{1,2}$ and $T_{2,1}$. They impose distinct geometric features in that $T_{2,1}$ does not fix any free line while $T_{1,2}$ fixes the free lines $\mathcal{L}_1$ and $\mathcal{L}_2$; see 
Proposition \ref{correspondence} and the table beside Figure 2. We separate it into two cases. 

 \textbf{Case (1)}:  Assume that $\deg g=2$. Then $\deg \mathcal{F}=0$ by Corollary \ref{Cor4.1}. 
 This implies that 
 the fixed points $g(0)$ and $g(\infty)$ do not lie on the free lines. 
 Therefore, we choose $\phi_2$ 
such that $\phi_2(r)=T_{2,1}$. 

 It is well-known that $C_4=\langle r\rangle$ have $4$ irreducible representations $W_h$ of degree $1$, and the corresponding characters are given by 
\begin{small}$$\chi_h(r^k)=i^{hk},~0\leq h,k\leq 3,\quad i=\sqrt{-1}.$$\end{small}
Moreover, $\rchi_h\cdot \rchi_l=\rchi_{h+l}$, where the addition is taken modulo $4$. Through $\phi_1$, it acts on $V_2$ by $(r\cdot f)(x,y):=f(T_{1,2}^{-1}(x,y)^T)$, and we can decompose $V_2$ into 
\begin{small}$$V_2=\Span_{\mathbb{C}}\{xy\}\oplus\Span_{\mathbb{C}}\{x^2,y^2\}, $$\end{small}
and, accordingly, the character is decomposed as $\rchi_{V_2}=\rchi_0+2\rchi_2$. Moreover, through $\phi_2$,
\begin{equation}\label{decompse of C2 in case the C2}{\scriptsize
\mathbb{C}^2=\Span_{\mathbb{C}}\{e_1+\xi_1 e_2\}\oplus \Span_{\mathbb{C}}\{e_1-\xi_1 e_2\},\quad \rchi_{\mathbb{C}^2}=\chi_1+\chi_3.}
\end{equation}
Thus $\rchi_{V_2\otimes \mathbb{C}^2}=\rchi_{V_2}\cdot \rchi_{\mathbb{C}^2}=3\rchi_1+3\rchi_3$, and $V_2\otimes \mathbb{C}^2$ is decomposed into the following two subspaces
\begin{equation}\label{V2 times C2}{\scriptsize
\begin{split}
3W_1=& \Span_{\mathbb{C}}\{x^2\otimes(e_1-\xi_1 e_2), xy\otimes (e_1+\xi_1 e_2),y^2\otimes (e_1-\xi_1e_2)\},\\
3W_3=& \Span_{\mathbb{C}}\{x^2\otimes (e_1+\xi_1e_2),xy\otimes (e_1-\xi_1e_2),y^2\otimes (e_1+\xi_1e_2)\}.
\end{split}}
\end{equation}
A plane in $3W_1$ is spanned by the two rows of the following matrix 
\begin{equation}\label{two plane in C2 deg 2}\begin{small}
g=\begin{pmatrix}
a_1 & a_2 & a_3 \\
b_1 & b_2 & b_3
\end{pmatrix}\begin{pmatrix}
x^2 & xy & y^2\\
-\xi_1 x^2 & \xi_1xy & -\xi_1 y^2
\end{pmatrix}^T.
\end{small}
\end{equation}
 Meanwhile, a plane $\widetilde{g}$ in $3W_3$ is equivalent to \eqref{two plane in C2 deg 2} by $g=-i\, \widetilde{g}\,T_{1,2}$. So, it suffices for us to use the plane in $3W_1$ in the following. Notice that $g$ is parametrized by {\scriptsize$\begin{pmatrix}
a_1 & a_2 & a_3 \\
b_1 & b_2 & b_3
\end{pmatrix}\in G(2,3)\cong \mathbb{P}^2$} with the Pl\"ucker coordinates $p_{ij}:=a_ib_j-a_jb_i,~1\leq i<j\leq 3$.
By the Cauchy-Binet formula, 
we have \begin{small}$$\det g=(1+i)\sqrt{2}xy(p_{12}x^2-p_{23}y^2),$$\end{small} which implies that one of $p_{12},~p_{23}$ does not vanish as $\det g\not\equiv 0$. By interchanging $x$ and $y$, we may assume that $p_{12}\neq 0$. Then  $\sqrt{p_{23}/p_{12}}\in \supp\mathcal{Q}$. If $p_{23}\neq 0$, up to a M\"obius transformation on $\mathbb{P}^1$ that preserves $0$ and $\infty$,  
we may assume $p_{23}/p_{12}=1$. It follows that after multiplying an $SL_2$ matrix on the left, 
we can parametrize $g$ as 
{\scriptsize$\begin{pmatrix}
1 & 0 & -1 \\
0 & 1 & b_3
\end{pmatrix}\in G(2,3)$.}
If $p_{23}=0$, we can conduct a scaling $(x,y)\mapsto (\lambda x,y/\lambda)$ such that $(a_2, b_2)=(a_3, b_3)$. Then after multiplying an $SL_2$ matrix on the left, 
we can parametrize $g$ as 
{\scriptsize$\begin{pmatrix}
1 & 0 & 0 \\
0 & 1 & 1
\end{pmatrix}\in G(2,3)$.} 

In both cases, each respective sextic curve $F$ downstairs is computed by substituting $(x,y)=(\sqrt{z},1)$ into $g\cdot uv(u^4-v^4)$, i.e., by calculating $f_1$ in \eqref{f} with the explicitly given $g$. Moreover, since 
$1\in \supp \mathcal{Q}\setminus\supp \mathcal{F}$ with $\mult_1(\varphi)=1$ in the generic case or $\mathcal{Q}(0)\geq 3$ in the limit case, we obtain that the sextic curve $F$ belongs to the general ramified family by Theorem~\ref{thm}.

\textbf{Case (2)}: Assume that $\deg g=3$. Then $\mathcal{F}=\mathcal{Q}$ with degree $\deg \mathcal{F}=6$ by Corollary \ref{Cor4.1} and \eqref{degree-difference}. It follows that the ramified points $0$ and $\infty$ are mapped to points on the free lines by $g$. Therefore, we choose $\phi_2$ such that $\phi_2(r)=T_{1,2}$. 

Through $\phi_1$, $V_3$ is decomposed into
{\small
\[V_3=\Span_{\mathbb{C}}\{x^3,xy^2\}\oplus\Span_{\mathbb{C}}\{x^2y,y^3\}, \]}
\!\!and, accordingly, the character is decomposed as $\rchi_{V_3}=2\rchi_1+2\rchi_3$. Moreover, through $\phi_2$,
\begin{small}\begin{equation*}\label{decompse of C2 in case the C2 deg 3}
\mathbb{C}^2=\Span_{\mathbb{C}}\{e_1\}\oplus \Span_{\mathbb{C}}\{e_2\},\quad\rchi_{\mathbb{C}^2}=\chi_1+\chi_3.
\end{equation*}\end{small}
\!\!Thus, {\small$\rchi_{V_3\otimes \mathbb{C}^2}=\rchi_{V_3}\cdot \rchi_{\mathbb{C}^2}=4\rchi_0+4\rchi_2$}, and $V_3\otimes \mathbb{C}^2$ is decomposed into 
\begin{small}\begin{align*}\label{V2 times C2}
\begin{split}
4W_0=& \Span_{\mathbb{C}}\{x^3\otimes e_2, x^2y\otimes e_1,xy^2\otimes e_2,y^3\otimes e_1\},\\
4W_2=& \Span_{\mathbb{C}}\{x^3\otimes e_1,x^2y\otimes e_2, xy^2\otimes e_1,y^3\otimes e_2\}.
\end{split}
\end{align*}\end{small}
\!\!A plane in $4W_2$ is spanned by the two rows of the following matrix
{\small\begin{equation}\label{two plane in C2 deg 3}
g=\begin{pmatrix}
a_1 & a_2 & a_3 & a_4 \\
b_1 & b_2 & b_3 & b_4
\end{pmatrix}\begin{pmatrix}
x^3 & 0 & xy^2 & 0\\
0 & x^2y & 0 & y^3
\end{pmatrix}^T.
\end{equation}}
\!\!A plane $\widetilde{g}$ in $4W_0$ is equivalent to \eqref{two plane in C2 deg 3} by $g(x,y)= \widetilde{g}(-x,y)T_{2,0}$, so that we use the plane $g$ in $4W_2$ in the following. Notice that $g$ is parametrized by {\scriptsize$\begin{pmatrix}
a_1 & a_2 & a_3 & a_4 \\
b_1 & b_2 & b_3 & b_4
\end{pmatrix}\in G(2,4)$} with the Pl\"ucker coordinates $p_{ij}= a_ib_j-a_jb_i,~1\leq i<j\leq 4$.

Note also that $p_{13}\neq 0$; 
otherwise, one of the entries $a$ and $c$ of $g$ will vanish after multiplying an $SL_2$ matrix on the left, which  
contradicts the fact that all entries of $g$ are not zero (lest the first or the last coordinate of the sextic curve $F$ would vanish by  \eqref{f}). Similarly, $p_{24}\neq 0$. In particular, we have $(a_2,b_2)\neq (0,0)\neq (a_3,b_3)$. Hence by Proposition \ref{Critical Prop Free lines} and $\mathcal{Q}=\mathcal{F}$, we obtain that 
\begin{small}$$\mathcal{Q}(0)=\mathcal{F}(0)=1,\quad \mathcal{Q}(\infty)=\mathcal{F}(\infty)=1,$$\end{small}
which further implies that $p_{12}\neq 0$ and $p_{34}\neq 0$, to not to violate the equalities. Then we may assume that $g$ is parametrized by {\scriptsize$\begin{pmatrix}
1 & 0 & a_3 & a_4 \\
0 & 1 & b_3 & b_4
\end{pmatrix}\in G(2,4)$} with $a_4b_3\neq 0$, by multiplying an $SL_2$ matrix on the left. 

Since $\deg \mathcal{F}=6$, up to M\"obius transformations, there exists $t\in\mathbb{C}\setminus\{0\}$ such that 
\begin{small}$$\mathcal{Q}=(0)+(\infty)+(1)+(-1)+(t)+(-t).$$\end{small} 
It is easy to verify that $\pm 1$ and $\pm t$ are mapped to points on the free lines ${\mathcal L}_{j},~3\leq j\leq 6$. A direct computation shows that $\mathrm{w}(\sigma(p))=-\mathrm{w}(p)$ (${\mathrm w}$ given in \eqref{26}). This implies either $g(1)\in \mathcal{L}_3$ 
while $g(-1)\in \mathcal{L}_5$, 
or $g(1)\in \mathcal{L}_4$ 
while $g(-1)\in \mathcal{L}_6$, up to changing $x$ to $-x$. By multiplying $T_{1,3}^{-1}$ on the right, we assume that $g(1)\in \mathcal{L}_3$. 

When $t^2\neq 1$ (the generic case), then a similar discussion as above implies either $g(t)\in \mathcal{L}_3$ or $g(t)\in \mathcal{L}_4$ 
up to a change of $t$ to $-t$. If $g(t)\in \mathcal{L}_4$, 
then we obtain the perpendicular conclusion in Case (2) from 
\begin{equation}\label{eq-L3}\begin{small}g(p)\in \mathcal{L}_3 ~(\text{vs}.~\mathcal{L}_4) \Leftrightarrow g(p)(1,-1)^T=(0,0)^T ~(\text{vs}.~g(p)(1,\sqrt{-1})^T=(0,0)^T). \end{small} \end{equation}
The sextic curves $F$ downstairs are computed by substituting $(x,y)=(\sqrt{z},1)$ into {\scriptsize$g\cdot uv(u^4-v^4)/xy(x^2-y^2)(x^2-t^2y^2)$,} in which the denominator gives the divisor 
${\mathcal F}$. 
If $g(t)\in \mathcal{L}_3$, 
then similar computations demonstrate that $F$ sits in ${\mathbb P}^5$, to be ruled out; in fact, $F=[f_0:\cdots:f_6]$ satisfies the linear constraint
\begin{small}$$
-2(t+1)^3\sqrt{5}f_0+(t+1)^2\sqrt{30}f_1-2\sqrt{3}(t+1)f_2+f_3=0.
$$\end{small}
\vskip-0.6mm
When $t^2=1$ (the non-generic case), we have $\mathcal{F}(1)=2$. In addition to the perpendicular condition imposed by $g(1)\in \mathcal{L}_3$, 
there is another constraint requiring $g'(1)\in \mathcal{L}_3$, 
which arises from Proposition~\ref{Critical Prop Free lines}. Such a curve is uniquely determined as $a_3=-3,~ a_4=-2,~b_3=2,~b_4=1$. The sextic curve $F$ downstairs is computed by substituting $(x,y)=(\sqrt{z},1)$ into {\scriptsize$g\cdot uv(u^4-v^4)/xy(x^2-y^2)^2$}. This case is also ruled out since $F$ sits in ${\mathbb P}^5$; we may also obtain the conclusion by letting $t\rightarrow 1$ 
in the generic case (when $g(1),~g(t)$ both lie in ${\mathcal L}_3$). 
\end{proof}
\begin{remark} In the proof of the above theorem, we have seen that the non-genric case 
is always the limit of the generic case. 
Similarly, it suffices to prove the generic cases in the sequel. {\rm (}Proposition {\rm \ref{Critical Prop Free lines}} is necessary for the non-generic cases.{\rm )}
\end{remark}

\section{Cyclic $C_3$ of order $3$}

In \eqref{diagram3}, assume that $\deg \varphi=3$. 
Up to M\"obius transformations, by Table \ref{Rational Galois Coverings}, $\varphi:\mathbb{P}^1\rightarrow \mathbb{P}^1$ is given by
 $w\mapsto z=w^3$. The associated Galois group of $\varphi$ consists of three elements generated by 
{\small\begin{equation}\label{cyclic3}
\sigma: \mathbb{P}^1\rightarrow \mathbb{P}^1,\quad\quad w\mapsto \exp(\frac{-2\pi\sqrt{-1}}{3})w.
\end{equation}}

\begin{theorem}
In the case of $C_3$, let $g([x,y]):\mathbb{P}^1\rightarrow \mathbb{P}^3$ be a holomorphic curve satisfying diagram \eqref{diagram3}. Then up to M\"obius transformations on $\mathbb{P}^1$ and the $PSL_2$-equivalence, we have $\deg g=3$ or $4$. Moreover, the following classification holds.
\begin{enumerate}
\item[\rm{(1)}] If $\deg g=3$, then \begin{small}$$g=\begin{pmatrix}
a_1 & a_2 & a_3 \\
b_1 & b_2 & b_3
\end{pmatrix}\begin{pmatrix}
x^3 & xy^2 & y^3\\
\nu_2 x^3 & \nu_1 xy^2 & \nu_2 y^3
\end{pmatrix}^T,$$\end{small} where $\nu_1\triangleq\frac{-1+i}{1-\sqrt{3}},~\nu_2\triangleq\frac{-1+i}{1+\sqrt{3}}$. Such curves form a $\mathbb{C}$-parameterized family with $p_{13}=a_1b_3-a_3b_1$ as the parameter. The sextic curve $F$ downstairs in \eqref{diagram3} determined by $g$ is  generally ramified.
\item[\rm{(2)}] If $\deg g=4$, then $$g=\begin{pmatrix}
1 & 0 & a_3 & a_4 \\
0 & 1 & b_3 & b_4
\end{pmatrix}\begin{pmatrix}
x^4 & x^3y & xy^3 & y^4\\
\nu_1 x^4 & \nu_2 x^3y & \nu_1 xy^3 & \nu_2 y^4
\end{pmatrix}^T,$$
where {\scriptsize$\begin{pmatrix}
1 & 0 & a_3 & a_4 \\
0 & 1 & b_3 & b_4
\end{pmatrix}\in G(2,4)$} 
is the unique $2$-plane which is perpendicular to $({\nu}_1,{\nu}_2,{\nu}_1,{\nu}_2)$ and $({\nu_1t^4},{\nu_2t^3},{\nu_1t},{\nu_2})$ {\emph{(}}or $({t}^4,{t}^3,{t},1)$ in the non-generic case as $t\to 1${\emph{)} under the bilinear inner} product.
This class depends on one freedom $t$ and belongs to the exceptional transversal family. 

\end{enumerate}

\end{theorem}

\begin{proof}
It follows from  Corollary \ref{Cor4.1} that $3\leq \deg(g)\leq 4$. 
Corresponding to  \eqref{cyclic3}, we choose $r=\Diag\{\eta,\eta^5\}$ generating a $C_6$ in $SU(2)$ and double covering $\sigma$, where $\eta=\exp(\frac{\pi i}{3})$. This gives $\phi_1$ on the left-hand side of \eqref{eigenspace problem}. Since in $S_4^*$, up to conjugation, the element of order $6$ is $T_{3,1}$, we choose $\phi_2$ on the right hand side of \eqref{eigenspace problem} such that $\phi_2(r)=T_{3,1}$.  
It is well-known that $C_6:=\langle r\rangle$ has $6$ irreducible representations $W_h$ of degree $1$, and the corresponding characters are given by 
\begin{small}$$\rchi_h(r^k)=\eta^{hk},~0\leq h,k\leq 5,$$\end{small}
where $\eta:=\exp\frac{i\pi}{3}$. Moreover, $\rchi_h\cdot \rchi_l=\rchi_{h+l}$, where the addition is taken modulo $6$.

\textbf{Case (1)}: Assume that $\deg g=3$, then $\deg \mathcal{F}=0$. 
Through $\phi_1$, $V_3$ is decomposed into
\begin{small}$$V_3=\Span_{\mathbb{C}}\{xy^2\}\oplus\Span_{\mathbb{C}}\{x^3,y^3\}\oplus\Span_{\mathbb{C}}\{x^2y\},$$ \end{small}
and, accordingly, the character is decomposed as {\small$\rchi_{V_3}=\rchi_1+2\rchi_3+\rchi_5$. Through $\phi_2$}, we have 
\begin{equation}\label{decompse of C2 in case the C3}\medmath{
\mathbb{C}^2=\Span_{\mathbb{C}}\{e_1+\nu_1 e_2\}\oplus \Span_{\mathbb{C}}\{e_1+\nu_2 e_2\},\quad\rchi_{\mathbb{C}^2}=\chi_1+\chi_5,}
\end{equation}
where $\nu_1\triangleq (-1+i)/(1-\sqrt{3}),\nu_2\triangleq (-1+i)/(1+\sqrt{3})$. Thus {\small$\rchi_{V_3\otimes \mathbb{C}^2}=\rchi_{V_3}\cdot \rchi_{\mathbb{C}^2}=2\rchi_0+3\rchi_2+3\rchi_4$, and $V_3\otimes \mathbb{C}^2$} is decomposed into the following three subspaces
{\scriptsize\begin{align*}\label{V3 times C2 in Case C3}
\begin{split}
 2W_0=& \Span_{\mathbb{C}}\{x^2y\otimes (e_1+\nu_1e_2),xy^2\otimes(e_1+\nu_2e_2) \},\\
3W_2=& \Span_{\mathbb{C}}\{x^3\otimes (e_1+\nu_2e_2),xy^2\otimes (e_1+\nu_1e_2),y^3\otimes (e_1+\nu_2e_2)\},\\
3W_4=& \Span_{\mathbb{C}}\{x^3\otimes (e_1+\nu_1e_2),x^2y\otimes (e_1+\nu_2e_2),y^3\otimes (e_1+\nu_1e_2)\}.
\end{split}
\end{align*}}
\!\!We rule out $2W_0$ as vectors in it have common divisors. 
A plane in $3W_2$ is spanned by the two rows of the following matrix (the coefficient matrices before $x^3$ and $y^3$ are not zero)
{\small\begin{equation}\label{two plane in 3W2 deg 3}
g=\begin{pmatrix}
a_1 & a_2 & a_3 \\
b_1 & b_2 & b_3
\end{pmatrix}\begin{pmatrix}
x^3 & xy^2 & y^3\\
\nu_2 x^3 & \nu_1 xy^2 & \nu_2 y^3
\end{pmatrix}^T.
\end{equation}}
\!\!A plane $\widetilde{g}(x,y)$ in $3W_4$ is equivalent to \eqref{two plane in 3W2 deg 3} by $g(x,y)= \frac{\sqrt{3}-1}{\sqrt{2}}\widetilde{g}(y,x)T_{2,3}$, so that we use the plane \eqref{two plane in 3W2 deg 3} in the following. Using the Pl\"ucker coordinates $p_{ij}=a_ib_j-a_jb_i,~1\leq i<j\leq 3$, 
we have \begin{small}$$\det g=(1-\sqrt{-1})\sqrt{3}xy^2(p_{12}x^3-p_{23}y^3).$$\end{small} 
\!\!Thus, up to M\"obius transformations on $\mathbb{P}^1$, 
the divisor $\mathcal{Q}$ equals either {\small$(0)+2(\infty)+\varphi^{-1}(1)$} (the generic case when $p_{12}=p_{23}\neq 0$), {\small$(0)+5(\infty)$}, or {\small$4(0)+2(\infty)$} (the last two are limit cases corresponding to $p_{12}=0$ or $p_{23}=0$, respectively). In the last two cases, we may assume that $p_{13}=1$ by scaling $y\mapsto \lambda y$ so that these cases give singleton curves.

In all the three cases, the sextic curves $F$ downstairs 
are computed by $g([z^\frac{1}{3},1])\cdot uv(u^4-v^4)$. Since 
$\ord_{\infty}(\mathcal{Q})\geq2$, $F$ is tangent to the $1$-dimensional orbit at $F(\infty)$ by Theorem~\ref{thm}. 

\textbf{Case (2)}: Assume that $\deg g=4$. 
Now $V_4$ is decomposed into
{\scriptsize \[V_4=\Span_{\mathbb{C}}\{x^2y^2\}\oplus\Span_{\mathbb{C}}\{x^4,xy^3\}\oplus\Span_{\mathbb{C}}\{x^3y,y^4\}, \]}
\!\!and, accordingly, the character is decomposed as {\small$\rchi_{V_4}=\rchi_0+2\rchi_2+2\rchi_4$}. 
It follows from \eqref{decompse of C2 in case the C3} that {\small$\rchi_{V_4\otimes \mathbb{C}^2}=\rchi_{V_4}\cdot \rchi_{\mathbb{C}^2}=3\rchi_1+4\rchi_3+3\rchi_5$}, and $V_4\otimes \mathbb{C}^2$ is decomposed into the following three subspaces
{\scriptsize\begin{align*}\label{V4 times C2 in C3 deg 4}
\begin{split}
3W_1=& \Span_{\mathbb{C}}\{x^4\otimes (e_1+\nu_2e_2),x^2y^2\otimes(e_1+\nu_1e_2), xy^3\otimes (e_1+\nu_2e_2)\},\\
4W_3=& \Span_{\mathbb{C}}\{x^4\otimes(e_1+\nu_1e_2),x^3y\otimes(e_1+\nu_2e_2),xy^3\otimes(e_1+\nu_1e_2),y^4\otimes(e_1+\nu_2e_2)\},\\
3W_5=& \Span_{\mathbb{C}}\{x^3y\otimes(e_1+\nu_1e_2),x^2y^2\otimes(e_1+\nu_2e_2),y^4\otimes(e_1+\nu_1e_2)\}.
\end{split}
\end{align*}}
\!\!We rule out $3W_1$ and $3W_5$ as vectors in them have common divisors. 
A plane in $4W_3$ is spanned by the two rows of the following matrix (the coefficient matrices before $x^4$ and $y^4$ are not zero)
{\small\begin{equation}\label{two plane in C3 deg 4}
g=\begin{pmatrix}
a_1 & a_2 & a_3 & a_4 \\
b_1 & b_2 & b_3 & b_4
\end{pmatrix}\begin{pmatrix}
x^4 & x^3y & xy^3 & y^4\\
\nu_1 x^4 & \nu_2 x^3y & \nu_1 xy^3 & \nu_2 y^4
\end{pmatrix}^T.
\end{equation}}
\!\!Since $\deg\mathcal{F}=6$, while g(0) and $g(\infty)$ do not lie on the free lines, by scaling, we may assume that {\small$\mathcal{F}=\varphi^{-1}(1)+\varphi^{-1}(t^3)$} (the generic case), where $t^3\neq 0,1,\infty$, or {\small$\mathcal{F}=2\varphi^{-1}(1)$}. (We discuss the generic case only. The non-generic case is the limit of the former by letting $t\rightarrow 1$). Moreover, since 
{\small$\mathcal{Q}=\mathcal{F}+(0)+(\infty)$}, from Proposition \ref{charc on 2dim orbit} and Theorem~\ref{thm}, we obtain that the sextic curve $F$ downstairs 
belongs to the exceptional transversal family.

Write $p_{ij}=a_ib_j-a_jb_i,~1\leq i,j\leq 4$. 
From \begin{small}$$\det g=(-1+\sqrt{-1})\sqrt{3}xy(p_{12}x^6+(p_{14}-p_{23})x^3y^3+p_{34}y^6),$$\end{small}
we have $p_{12}\neq 0$ so that we may assume that {\scriptsize$\begin{pmatrix}
a_1 & a_2 \\
b_1 & b_2
\end{pmatrix}=\Id_2$} after multiplying an $SL_2$ matrix. 

By computing the induced action of $\phi_2(r)$ on the roots ${\mathrm w}$ in \eqref{eq-mathrmw}, it is straightforward to verify that the orbit $\varphi^{-1}(1)$ is mapped to points on three distinct free lines, which are either $\{\mathcal{L}_1, \mathcal{L}_4, \mathcal{L}_5\}$ or $\{\mathcal{L}_2, \mathcal{L}_3, \mathcal{L}_6\}$. Up to a M\"obius transformation on $\mathbb{P}^1$, 
we may assume that $g(1)$ lies on ${\mathcal L}_1$ or ${\mathcal L}_2$. By further multiplying $T_{2,3}$ on the right, 
we may assume $g(1)\in {\mathcal L}_1$. Likewise, interchanging $t$ with another point of $\varphi^{-1}(t^3)$, we may assume $g(t)\in {\mathcal L}_1$ or ${\mathcal L}_2$. The conclusion then follows from the characterization  
\begin{equation}\label{eq-L12}\begin{small}g(p)\in \mathcal{L}_1 ~(\text{resp}.~\mathcal{L}_2) \Leftrightarrow g(p)(0,1)^T=(0,0)^T ~(\text{resp}.~g(p)(1,0)^T=(0,0)^T). \end{small} \end{equation}

In both cases, the sextic curves $F$ downstairs are computed by substituting $(x,y)=(z^{\frac{1}{3}},1)$ into {\scriptsize$g\cdot uv(u^4-v^4)/(w^3-1)(w^3-t^3)$}.
\end{proof}

\section{Cyclic $C_4$ of order $4$}

In \eqref{diagram3}, assume that $\deg \varphi=4$ and $G=C_4$. 
Up to M\"obius transformations, by Table \ref{Rational Galois Coverings}, $\varphi:\mathbb{P}^1\rightarrow \mathbb{P}^1$ is given by
 $w\mapsto z=w^4$. 
 The associated Galois group of $\varphi$ consists of three elements generated by 
{\small\begin{equation}\label{cyclic4}
\sigma: \mathbb{P}^1\rightarrow \mathbb{P}^1,\quad\quad w\mapsto -\sqrt{-1}w.
\end{equation}}

\begin{theorem}
In the case of $C_4$, let $g([x,y]):\mathbb{P}^1\rightarrow \mathbb{P}^3$ be a holomorphic curve satisfying diagram \eqref{diagram3}. 
Then $\deg g=5$, and up to M\"obius transformations on $\mathbb{P}^1$, 
$g$ is equivalent to 
\begin{small}$$\begin{pmatrix}
1 & a_2& 0 & a_4 \\
0 & b_2 & 1 & b_4
\end{pmatrix}\begin{pmatrix}
x^5 & 0 & xy^4 & 0 \\
0 & x^4y & 0 & y^5
\end{pmatrix}^T.$$\end{small}
where {\scriptsize$ \begin{pmatrix}
1 & a_2& 0 & a_4 \\
0 & b_2 & 1 & b_4
\end{pmatrix}\in G(2,4)$} 
is a $2$-plane that is perpendicular to $(1,-1,1,-1)$ under the bilinear inner product. This class depends on two freedoms $a_4,~b_4$ and belongs to the generally ramified family.
\end{theorem}

\begin{proof}
It follows from  Corollary \ref{Cor4.1} that $4\leq \deg(g)\leq 6$. If $\deg g=4$, then $\mathcal{F}=0$. But this contradicts the fact that $0, \infty\in \mathcal{F}$ since $\mult_0(\varphi)=\mult_\infty(\varphi)=4$. 

Corresponding to  \eqref{cyclic3}, we choose $r=T_{1,1}$ generating a $C_8$ in $SU(2)$ and double covering $\sigma$. This gives $\phi_1$ on the left-hand side of \eqref{eigenspace problem}. Since in $S_4^*$, up to conjugation, the element of order $8$ is $\pm T_{1,1}$, we choose $\phi_2$ on the right-hand side of \eqref{eigenspace problem} such that $\phi_2(r)=T_{1,1}$ and leave the other case to be discussed similarly in Remark \ref{choose of -T11 in C4} below.  
It is well-known that $C_8:=\langle r\rangle$ have $8$ irreducible representations $W_h$ of degree $1$, and the corresponding characters are given by 
\begin{small}$$\rchi_h(r^k)=\xi_{hk},~0\leq h,k\leq 7,$$\end{small}
where $\xi_l:=\exp(\frac{\sqrt{-1}l\pi}{4})$. Moreover, $\rchi_h\cdot \rchi_l=\rchi_{h+l}$, where the addition is taken modulo $8$.

\textbf{Case (1)}: Assume $\deg g=5$. Then $\deg \mathcal{F}=6$. 
Through $\phi_1$, $V_5$ is decomposed into
\begin{small}$$V_5=\Span\{x^2y^3\}\oplus\Span\{x^5,xy^4\}\oplus\Span\{x^4y,y^5\}\oplus\Span\{x^3y^2\},$$\end{small}
and, accordingly, the character is decomposed as $\rchi_{V_5}=\rchi_1+2\rchi_3+2\rchi_5+\rchi_7$. 
Through $\phi_2$, 
{\small\begin{equation}\label{decompse of C2 in deg 5 c4}
\mathbb{C}^2=\Span\{e_1\}\oplus \Span\{e_2\},~~~\rchi_{\mathbb{C}^2}=\chi_1+\chi_7.
\end{equation}}
\!\!Thus {\small$\rchi_{V_5\otimes \mathbb{C}^2}=\rchi_{V_5}\cdot \rchi_{\mathbb{C}^2}=2\rchi_0+3\rchi_2+4\rchi_4+3\rchi_6$}, and {\small$V_5\otimes \mathbb{C}^2$} is decomposed into 
\begin{equation}\label{direct sum decomposition C4}
{\scriptsize\begin{split}
2W_0=& \Span\{x^3y^2\otimes e_1,x^2y^3\otimes e_2\},\quad 3W_2=\Span\{x^5\otimes e_2,x^2y^3\otimes e_1,xy^4\otimes e_2\},\\
4W_4=& \Span\{x^5\otimes e_1, x^4y\otimes e_2,xy^4\otimes e_1,y^5\otimes e_2\},\quad 3W_6= \Span\{x^4y\otimes e_1,x^3y^2\otimes e_2,y^5\otimes e_1\}.
\end{split}}
\end{equation}
We exclude $2W_0$, $3W_2$ and $3W_6$ because the vectors in each case have common factors. 
A plane in $4W_4$ is spanned by the two rows of the following matrix
\begin{equation}\label{two plane in C4 deg 5}{\small
g=\begin{pmatrix}
a_1 & a_2 & a_3 & a_4 \\
b_1 & b_2 & b_3 & b_4
\end{pmatrix}\begin{pmatrix}
x^5 & 0 & xy^4 & 0\\
0 & x^4y & 0 & y^5
\end{pmatrix}^T.}
\end{equation}
Consider the Pl\"ucker coordinates $p_{ij}=a_ib_j-a_jb_i,~1\leq i<j\leq 4$. 
Similar to the discussion for \eqref{two plane in C2 deg 3}, we conclude that $(a_2,b_2)\neq (0,0)\neq (a_3,b_3)$. Hence by Proposition \ref{Critical Prop Free lines} 
we obtain that {\small$\mathcal{F}(0)=\mathcal{F}(\infty)=1$}. From $\deg \mathcal{F}=6$, by scaling, we may assume {\small$\mathcal{F}=(0)+(\infty)+\varphi^{-1}(1)$}. Moreover, by multiplying an $SL_2$ matrix, we may also assume
{\scriptsize$\begin{pmatrix}
a_1 & a_3 \\
b_1 & b_3
\end{pmatrix}=\Id_2$}.

It is straightforward to verify that the orbit {\small$\varphi^{-1}(1)$} is mapped to points on the free lines ${\mathcal L}_{j},~3\leq j\leq 6$. A direct computation shows that ${\mathrm w}(\sigma(p))=i\,{\mathrm w}(p)$ (${\mathrm w}$ given in \eqref{26}). This implies that up to a M\"obius transformation on $\mathbb{P}^1$, we may assume $g(1)\in \mathcal{L}_3$. So, the perpendicular conclusion follows from \eqref{eq-L3}.  

The sextic curve $F$ downstairs is computed by substituting $(x,y)=(\sqrt[4]{z},1)$ into {\scriptsize$g\cdot uv(u^4-v^4)/xy(x^4-y^4)$}. Note that either $\supp\mathcal{Q}=\supp\mathcal{F}$ but $\mathcal{Q}\neq \mathcal{F}$, or there exists $p\in \supp\mathcal{Q}\setminus\supp\mathcal{F}$ such that $\mult_p(\varphi)=1$. Therefore, by Theorem~\ref{thm}, we conclude that $F$ always belongs to the generally ramified family.

\textbf{Case (2)}:  Assume that $\deg g=6$. 
Through $\phi_1$, $V_6$ is decomposed into
{\scriptsize\[V_6=\Span\{x^3y^3\}\oplus\Span\{x^6,x^2y^4\}\oplus\Span\{x^5y,xy^5\}\oplus\Span\{x^4y^2,y^6\}, \]}
\!\!and, accordingly, the character is decomposed as {\small$\rchi_{V_6}=\rchi_0+2\rchi_2+2\rchi_4+2\rchi_6$}. It follows from \eqref{decompse of C2 in deg 5 c4} that {\small$\rchi_{V_6\otimes \mathbb{C}^2}=\rchi_{V_6}\cdot \rchi_{\mathbb{C}^2}=3\rchi_1+4\rchi_3+4\rchi_5+3\rchi_7,$} and {\small$V_6\otimes \mathbb{C}^2$} is decomposed into
\begin{equation}{\scriptsize
\begin{split}
3W_1=& \Span\{x^6\otimes e_2,x^3y^3\otimes e_1,x^2y^4\otimes e_2\},\quad 4W_3=\Span\{x^6\otimes e_1,x^5y\otimes e_2,x^2y^4\otimes e_1,xy^5\otimes e_2\},\\
4W_5=& \Span\{x^5y\otimes e_1,x^4y^2\otimes e_2,xy^5\otimes e_1,y^6\otimes e_2\},\quad 3W_7= \Span\{x^4y^2\otimes e_1,x^3y^3\otimes e_2,y^6\otimes e_1\}.
\end{split}}
\end{equation}
Because $x^6,y^6$ does not appear simultaneously in any one of the above four subspaces, the case $\deg g=6$ does not occur.
\end{proof}

\begin{remark}\label{choose of -T11 in C4}
If we choose $r$ corresponding to $-T_{1,1}$ on the right, then {\small$\rchi_{\mathbb{C}^2}=\rchi_3+\rchi_5=\rchi_4(\rchi_1+\rchi_7)$} {\rm (}compare \eqref{decompse of C2 in deg 5 c4}{\rm )}. Thus {\small$\rchi_{V_5}\cdot \rchi_{\mathbb{C}^2}=\rchi_{V_5}\cdot (\rchi_1+\rchi_7)\cdot \rchi_4$}, and the direct sum decomposition \eqref{direct sum decomposition C4} still holds, although the subscripts may change so that we derive the same formula \eqref{two plane in C4 deg 5}. 
In the following, we will often encounter this phenomenon, in which case a discussion on a single choice suffices.
\end{remark}

\section{Dihedral $D_2$ of order $4$}\label{sec-D2}

In \eqref{diagram3}, assume that $\deg \varphi=4$ and $G=D_2$. Recall that the invariants of $D_2$ are generated by the three invariants $\alpha_2,\beta_2,
\gamma$ with the relation given in Table 1. 
Up to M\"obius transformations, by Table \ref{Rational Galois Coverings}, $\varphi:\mathbb{P}^1\rightarrow \mathbb{P}^1$ is given by
 $$\varphi:w\longmapsto z=-\alpha_2^2/\gamma^2=-(w^2-2+\frac{1}{w^2})/4.$$ The associated Galois group of $\varphi$ consists of three elements generated by 
{\small\begin{equation}\label{D2GaloisGroup}
\sigma: \mathbb{P}^1\rightarrow \mathbb{P}^1,\quad\quad w\mapsto -w;\quad \tau: \mathbb{P}^1\rightarrow \mathbb{P}^1,\quad\quad w\mapsto 1/w.
\end{equation}}
\!\!The ramified points of $\varphi$ are 
$(0,\infty),~(1,-1),~(i,-i),$ in pairs,
with multiplicities $2$. They are the zeros of $\gamma,~\alpha_2,~\beta_2$, respectively, and are assigned to $\infty,~0$ and $1$, respectively under $\varphi$. Moreover, $0$ and $\infty$ are fixed under $\sigma$, $1$ and $-1$ are fixed under $\tau$, while $i$, and $-i$ are fixed under $\tau\circ \sigma$. Up to M\"obius transformations acting on both copies of $\mathbb{P}^1$, the three pairs can be arbitrarily permuted without affecting the map 
$\varphi$. 

On the left-hand side of \eqref{eigenspace problem}, in correspondence to \eqref{D2GaloisGroup}, we select $A=T_{1,2}$ and $B=T_{2,2}$ as double coverings of 
$\sigma$ and $\tau$, respectively.
They generate a subgroup isomorphic to the quaternion group,
{\small\begin{equation}\label{gp}
G^\ast\triangleq \{\pm \Id,\pm T_{1,2},\pm T_{2,2},\pm T_{2,0}\}\subset S_4^*.
\end{equation}}
It has $5$ conjugacy classes $\{\Id\},\{-\Id\},\{\pm T_{1,2}\},\{\pm T_{2,2}\},\{\pm T_{2,0}\}$; hence $G^\ast$ has $5$ irreducible representations.
\vskip -0.25cm
{\scriptsize\begin{table}[H]
\begin{center}
\begin{tabular}{|c|c|c|c|c|c|c|}
\hline
         &  $\Id$ & $-\Id$ & $\{\pm T_{1,2}\}$ & $\{\pm T_{2,2}\}$ & $\{\pm T_{2,0}\}$ \\ \hline
$\psi_1$ &  $1$   & $1$    & $1$               & $1$               & $1$               \\ \hline
$\psi_2$ &  $1$   & $1$    & $1$               & $-1$               & $-1$               \\ \hline
$\psi_3$ &  $1$   & $1$    & $-1$               & $1$               & $-1$               \\ \hline
$\psi_4$ &  $1$   & $1$    & $-1$               & $-1$               & $1$               \\ \hline
$\psi_5$     & $2$   & $-2$    & $0$               & $0$               & $0$               \\ \hline
\end{tabular}
\end{center}
\end{table}}
The first four irreducible representations of degree $1$ are obtained by the double covering $G^\ast\rightarrow G^\ast/\{\pm \Id\}=D_2$. The last representation is obtained by multiplications of $G^\ast$ on the column vectors of $\mathbb{C}^2$ (also equals $\rchi_{V_1}$). Immediately, we obtain that 
\begin{equation}\label{multiplication table for D2 char}{\scriptsize
\psi_i\psi_5=\psi_5,~~1\leq i\leq 4;~~~\psi_5^2=\psi_1+\psi_2+\psi_3+\psi_4,}
\end{equation}
where $\mathbb{C}^2\otimes \mathbb{C}^2$ are decomposed into four irreducible subspaces of dimension $1$
\begin{equation}
\label{key decomposition for D2}{\scriptsize
\begin{split}
W_1&=\Span\{e_1\otimes e_2-e_2\otimes e_1\},\quad W_2=\Span\{e_1\otimes e_2+e_2\otimes e_1\}\\
W_3&=\Span\{e_1\otimes e_1-e_2\otimes e_2\},\quad W_4=\Span\{e_1\otimes e_1+e_2\otimes e_2\}
\end{split}}
\end{equation}
with characters $\psi_i,~1\leq i\leq 4$, respectively

To determine $\phi_2$ on the right-hand side of \eqref{eigenspace problem}, observe that up to conjugation, there are two types of  $D_2$ in $S_4$, namely, 
\begin{small}$$
H_1\triangleq\{\Id,(12)(34),(13)(24),(14)(23)\},\quad\quad H_2\triangleq\{\Id,(13),(24),(13)(24)\}.
$$\end{small}
 Note that $H_1$, normal in $S_4$, is double covered by the group $G^\ast$ in \eqref{gp}, and $H_2$, not normal in $S_4$, is double covered by the group \begin{small}$$G_2\triangleq\{\pm \Id, \pm T_{1,2},\pm T_{2,3},\pm T_{2,1}\}\subset S_4^*.$$\end{small}  Both $G^\ast$ and $G_2$, isomorphic to each other, have a half-turn around the axis through the opposite vertices, so that one of the three pairs of ramified points, say, $\{g(0),g(\infty)\}$,  lies on the free lines. The geometry of $G^\ast$ and $G_2$ differ by the two lemmas below.

{
\begin{lemma}\label{first type D2}
Under the assumption that $\{g(0),g(\infty)\}$ lies on the free lines, up to conjugation {\rm (}by elements in $S_4^*${\rm )}, there is a unique isomorphism $\Gamma:G^\ast\rightarrow G^\ast$ given by the identity map. Moreover, all the three pairs of ramified points lie on the free lines, and we may assume that {\small$\mathcal{Q}(0)\geq \mathcal{Q}(1)\geq \mathcal{Q}(i)$}.
\end{lemma}

\begin{proof} The outer automorphism group of the quaternion group is isomorphic to $S_3$ obtained through $G^\ast$ as follows. Since $G^\ast$ is normal in $S_4^*$, the conjugation $\Gamma: k \mapsto gkg^{-1},\forall k\in G^\ast,$ by a $g\in S_4^*$ gives rise to an outer automorphism of $G^\ast$ when $g$ runs through the factor group $S_4^*/G^\ast\simeq S_3$. 
Now consider $\tilde{\Gamma}\triangleq g^{-1}\Gamma g$. 
Geometrically, this means that after a $SL_2$-coordinate change of ${\mathbb C}^2$ in $V_s\otimes {\mathbb C}^2$ by $g\in S_4^*$, we may assume that $\Gamma$ is the identity.

Because $\pm T_{1,2},\pm T_{2,2}$ and $\pm T_{2,0}$ are all half-turns around the axis through the vertices of opposite sides, all the three pairs of ramified points lie in the free lines. Since we can permute them in any order by M\"obius transformations, we may assume that {\small$\mathcal{Q}(0)\geq \mathcal{Q}(1)\geq \mathcal{Q}(i)$}. 
\end{proof}

\begin{lemma}\label{second type D2}
Under the assumption that $\{g(0),g(\infty)\}$ lies on the free lines, up to conjugation {\rm (}by elements in $S_4^*${\rm )}, there is a unique isomorphism $\theta:G^\ast\rightarrow G_2$ determined by $\theta(T_{1,2})=T_{1,2}$ and $\theta(T_{2,2})=T_{2,1}$. Moreover, the pairs $\{g(1),g(-1)\}$ and $\{g(i),g(-i)\}$ do not lie in the free lines. Meanwhile, the character of $\theta:G^\ast\rightarrow G_2\subseteq SU(2)$ is the above $\psi_5$.
\end{lemma}

\begin{proof}  Unlike $G^\ast$, $G_2$ is not normal in $S_4^*$. Because $\pm T_{1,2}$ are half-turns around the axis through the opposite vertices, and $\pm T_{2,2}$ and $\pm T_{2,0}$ are half-turns around the axis through the midpoints of opposite edges, while we assume that $\{g(0),g(\infty)\}$ lie on the free lines, it follows that there are only two choices of $\theta(T_{1,2})$, namely, $\theta(T_{1,2})=\pm T_{1,2}$. 

Adopting the quaternion notation, let ${\bf I}\triangleq T_{2,0}$ and ${\bf J}\triangleq T_{2,2}$. We have ${\bf K}=T_{1,2}={\bf I}{\bf J}$. Likewise, let ${\bf i}\triangleq T_{2,1}$, ${\bf j}\triangleq T_{2,3}$, and ${\bf k}\triangleq T_{1,2}$; we have ${\bf i}{\bf j}={\bf k}$ while ${\bf K}={\bf k}$, so that $\theta({\bf K})=\pm {\bf k}$ to obey the above geometric constraint. 
 
 We may assume $\theta({\bf K})={\bf k}$; otherwise, if $\theta({\bf K})=-{\bf k}$, then $\widetilde{\theta}\triangleq{\bf i}\theta {\bf i}^{-1}:G^\ast\rightarrow G_2$ maps ${\bf K}$ to ${\bf k}$. Consequently, $\theta({\bf I})\theta({\bf J})={\bf k}$, so that by quaternion multiplication, either $\theta({\bf I})=\pm {\bf i}$ and $\theta({\bf J})=\pm {\bf j}$, or $\theta({\bf I})=\pm {\bf j}$ and $\theta({\bf J})=\mp {\bf i}.$ In the former case, if $\theta({\bf I})=-{\bf i},$ $\theta({\bf J})=-{\bf j}$, then ${\bf k}\theta {\bf k}^{-1}$ sends ${\bf I}$ to ${\bf i}$ and ${\bf J}$ to ${\bf j}$.  In the latter case, if $\theta({\bf I})={\bf j}$ and $\theta({\bf J})=-{\bf i}$, then ${\bf k}\theta {\bf k}^{-1}$ sends ${\bf I}$ to $-{\bf j}$ and ${\bf J}$ to ${\bf i}$. In any event, we may assume that either $\theta$ maps ${\bf I},{\bf J},{\bf K}$ to ${\bf i}, {\bf j},{\bf k}$, respectively, or sends ${\bf I},{\bf J},{\bf K}$ to $- {\bf j}, {\bf i}, {\bf k}$, respectively.
 
 However, in the former case, the conjugation $\sigma^x: g\mapsto xgx^{-1},g\in G_2,x\triangleq T_{1,1}^3,$ fixes $G_2$ in $S_4^*$, fixes ${\bf K}$ by commutativity, and sends ${\bf i}$ to $-{\bf j}$ and ${\bf j}$ to ${\bf i}$. so that $\sigma^x\circ \theta=x\theta x^{-1}$ sends ${\bf I}$ to $-{\bf j}$ and ${\bf J}$ to ${\bf i}$, 
which is what we are after. 

\end{proof}

\begin{theorem}
In the case of $D_2$, let $g([x,y]):\mathbb{P}^1\rightarrow \mathbb{P}^3$ be a holomorphic curve satisfying diagram \eqref{diagram3}. Then $\deg g=5$, and up to M\"obius transformations on $\mathbb{P}^1$, 
$g$ is equivalent to 
\begin{enumerate}
\item[\rm{(1)}] 
{\small\begin{equation}\label{g d2 case2}
g=\begin{pmatrix}
a_1 & 1 & 0\\
b_1 & 0 & 1
\end{pmatrix}\begin{pmatrix}
x\alpha_2^2 & x\gamma^2 & -x\alpha_2\beta_2 \\
y\alpha_2^2 & y\gamma^2 & y\alpha_2\beta_2
\end{pmatrix}^T,
\end{equation}}
which depends on two freedoms $a_1,~b_1$ and belongs to the generally ramified family, or

\item[\rm{(2)}] 
{\small\begin{equation}\label{g d2 case3}
g=\begin{pmatrix}
a_1 & 1 & 0\\
b_1 & 0 & 1
\end{pmatrix}\begin{pmatrix}
x\alpha_2^2 & x\gamma^2 & -x\alpha_2\beta_2 \\
\xi_1y\alpha_2^2 & \xi_1y\gamma^2 & \xi_1y\alpha_2\beta_2
\end{pmatrix}^T,
\end{equation}}
where $\xi_1=\exp(\frac{\pi i}{4})$ and {\scriptsize$ \begin{pmatrix}
a_1 & 1 & 0\\
b_1 & 0 & 1
\end{pmatrix}\in G(2,3)$} 
is a unique $2$-plane that is perpendicular to {\scriptsize${((t^2-1)^2(\xi_1-t),4(\xi_1-t)t^2,(t^4-1)(\xi_1+t))}$} under the usual bilinear inner product. This class depends on one freedom $t$ and belongs to the exceptional transversal family.

\end{enumerate}

\end{theorem}

\begin{proof}
It follows from  Corollary \ref{Cor4.1} that $4\leq \deg(g)\leq 6$. By Proposition \ref{poincareEven} and the Clebsch-Gordan formula {\small$V_{2k}\otimes \mathbb{C}^2\cong V_{2k}\otimes V_1\cong V_{2k+1}\oplus V_{2k-1}$} (see \eqref{transvectant}), we rule out the cases when $\deg g=4$ and $6$. Therefore, we assume that $\deg g=5$ in the following so that $\deg \mathcal{F}=6$.

\textbf{Case (1)}: On the right-hand side of \eqref{eigenspace problem}, we choose $\phi_2$ to be the identity map   $\Id:G^\ast\rightarrow G^\ast$ as stated in Lemma \ref{first type D2}. Then all the three pairs of ramified points lie on the free lines, so that {\small$\mathcal{F}=(\alpha_2)+(\beta_2)+(\gamma)$} and we may assume that {\small$\mathcal{Q}(0)\geq \mathcal{Q}(1)\geq \mathcal{Q}(i)$}.

Note that $\{\alpha_2^2, \alpha_2\beta_2, \gamma^2\}$ are three linearly independent invariants in $V_4$. By using item (2) of Proposition \ref{key basis prop}, we obtain the following decomposition, 
\begin{small}\begin{equation}\label{V5 D2}{
V_5=\Span\{y\alpha_2^2,-x\alpha_2^2\}\oplus\Span\{y\gamma^2,-x\gamma^2\}\oplus \Span\{y\alpha_2\beta_2,x\alpha_2\beta\}, ~~~~\quad \rchi_{V_5}=3\psi_5.}
\end{equation}\end{small}
It follows from \eqref{key decomposition for D2} and \eqref{V5 D2} that 
{\small$\rchi_{V_5\otimes \mathbb{C}^2}=3 \psi_5^2=3\psi_1+3\psi_2+3\psi_3+3\psi_4,$} and {\small $V_5\otimes \mathbb{C}^2$} is decomposed into the following four subspaces 
\begin{equation}\label{direct sum decomposition D2 deg 5 first kind}{\scriptsize
\begin{split}
3W_1=& \Span\{x\alpha_2^2\otimes e_1+y\alpha_2^2\otimes e_2,x\gamma^2\otimes e_1+y\gamma^2\otimes e_2,-x\alpha_2\beta_2\otimes e_1+y\alpha_2\beta_2\otimes e_2\},\\
3W_2=& \Span\{-x\alpha_2^2\otimes e_1+y\alpha_2^2\otimes e_2,-x\gamma^2\otimes e_1+y\gamma^2\otimes e_2,x\alpha_2\beta_2\otimes e_1+y\alpha_2\beta_2\otimes e_2\},\\
3W_3=& \Span\{y\alpha_2^2\otimes e_1+x\alpha_2^2\otimes e_2, y\gamma^2\otimes e_1+x\gamma^2\otimes e_2, y\alpha_2\beta_2\otimes e_1-x\alpha_2\beta_2\otimes e_2\},\\
3W_4=& \Span\{y\alpha_2^2\otimes e_1-x\alpha_2^2\otimes e_2,y\gamma^2\otimes e_1-x\gamma^2\otimes e_2,y\alpha_2\beta_2\otimes e_1+x\alpha_2\beta_2\otimes e_2\}.
\end{split}}
\end{equation}
A plane in $3W_1$ is spanned by the two rows of the following matrix
\begin{equation}\label{two plane in the eigenspace D2 deg 5 first type}{\small
g=\begin{pmatrix}
a_1 & a_2 & a_3 \\
b_1 & b_2 & b_3 
\end{pmatrix}\begin{pmatrix}
x\alpha_2^2 & x\gamma^2 & -x\alpha_2\beta_2 \\
y\alpha_2^2 & y\gamma^2 & y\alpha_2\beta_2
\end{pmatrix}^T}
\end{equation}
The plane in $\widetilde{g}(x,y)$ in $3W_2$ is equivalent to $g$ by changing $y$ to $-y$, while the plane in $3W_{2+i}$ is equivalent to the plane in $3W_{i}$ by interchanging $x$ and $y$, for $i=1,2$. So, we use \eqref{two plane in the eigenspace D2 deg 5 first type} in the following computation. 

Using Pl\"ucker coordinates $p_{ij}=a_ib_j-a_jb_i,~1\leq i<j\leq 3$, we obtain \begin{small}$$\det g=2\alpha_2\beta_2\gamma(p_{13}\alpha_2^2+p_{23}\gamma^2).$$\end{small} 
Let $t$ be a solution of $p_{13}\alpha_2^2+p_{23}\gamma^2=0$ with 
$p_{13}p_{23}(p_{13}-p_{23})\neq 0$, and let $\varphi^{-1}(\varphi(t))$ denote the divisor consisting of $4$ distinct points of 
the principal orbit 
through $t$. Then by Lemma \ref{linear system of D2} and the assumption that {\small$\mathcal{Q}(0)\geq \mathcal{Q}(1)\geq \mathcal{Q}(i)$}, we obtain that {\small$\mathcal{Q}=(\alpha_2)+(\beta_2)+(\gamma)+\varphi^{-1}(\varphi(t))$} in the generic case, or {\small$\mathcal{Q}=(\alpha_2)+(\beta_2)+3(\gamma)$} in the non-generic case that arises when $p_{13}\rightarrow 0$. 
In both cases, $p_{23}\neq 0$, and so we may assume that $g$ is given as in \eqref{g d2 case2}. 

The sextic curves $F$ downstairs in both cases are computed by {\scriptsize $g\cdot uv(u^4-v^4)/xy(x^4-y^4)$} and invariant theory (the Galoisness of $\varphi$ dictates that the coordinates of $F$ be invariants of $\alpha_2,\beta_2$, and $\gamma$). $F$ is tangent to the $1$-dimensional orbit at $F(\varphi(t))$, or $F(\infty)$, respectively.

\textbf{Case (2)}: Now, on the right-hand side of \eqref{eigenspace problem} we choose $\phi_2$ to be the map $\theta:G^\ast\rightarrow G_2$ as stated in the Lemma \ref{second type D2}. It is equivalent to Case (1) under the new basis $(\widetilde{e}_1,\widetilde{e}_2):=(e_1,\xi_1e_2)$, where $\xi_1=\exp(\frac{\pi i}{4})$, so that {\small$V_5\otimes \mathbb{C}^2$} is decomposed similarly as \eqref{direct sum decomposition D2 deg 5 first kind} by changing $e_2$ to $\xi_1e_2$.

 Similarly, we need only consider the plane in $3W_1$ spanned by the two rows of the following matrix 
\begin{small}$$g=\begin{pmatrix}
a_1 & a_2 & a_3 \\
b_1 & b_2 & b_3 
\end{pmatrix}\begin{pmatrix}
x\alpha_2^2 & x\gamma^2 & -x\alpha_2\beta_2 \\
\xi_1y\alpha_2^2 & \xi_1y\gamma^2 & \xi_1y\alpha_2\beta_2
\end{pmatrix}^T,$$\end{small} in which case, only $g(0),g(\infty)$ lie on the free lines, while $g(1),g(-1),g(i),g(-i)$ do not.

Note that \begin{small}$$\det g=2\xi_1\alpha_2\beta_2\gamma(p_{13}\alpha_2^2+p_{23}\gamma^2).$$\end{small} 
By Lemma \ref{linear system of D2} and the restriction on free divisor $\mathcal{F}$, we obtain that {\small$\mathcal{Q}=(\alpha_2)+(\beta_2)+3(\gamma)$} and {\small$\mathcal{F}=3(\gamma)$} in the non-generic case,  or {\small$\mathcal{Q}=(\alpha_2)+(\beta_2)+(\gamma)+\varphi^{-1}(\varphi(t))$} and {\small$\mathcal{F}=\gamma+\varphi^{-1}(\varphi(t))$} in the genric case, where $\varphi^{-1}(\varphi(t))$ consists of $4$ distinct points constituting the principal orbit of $D_2$ through the point $t$ (a solution of {\small$p_{13}\alpha_2^2+p_{23}\gamma^2=0$}). In both cases, $p_{23}\neq 0$, so that we may assume that $g$ is in the form \eqref{g d2 case3}. For the generic case, right multiplication by {\small$T_{2,1}^T$} (corresponding to interchanging $t$ with other points in $\varphi^{-1}(\varphi(t))$) induces the permutation $(12)(3465)$ on free lines; we may thus assume that $g(t)$ lies on ${\mathcal L}_2$ or ${\mathcal L}_3$. The sextic curve $F$ is computed by {\scriptsize$g\cdot uv(u^4-v^4)/xy(x^2-t^2y^2)(x^2-y^2/t^2)$} and invariant theory.

If $g(t)$ lies on ${\mathcal L}_2$, then we can solve $a_1$ and $b_1$ as rational polynomials of $t$ from the perpendicular condition arising from \eqref{eq-L12}. A direct computation shows that $F$ sits in ${\mathbb P}^5$, to be ruled out. In fact, $F=[f_0:\cdots:f_6]$ satisfies the linear constraint
{\scriptsize\[
(t^2+1)^3f_0-\sqrt{3/2}(t^2+1)^2(t^2-1)\, f_1+\sqrt{3/5}(t^2+1)(t^2-1)^2\,f_2-\sqrt{1/20}(t^2-1)^3f_3=0.\]} 
The non-generic case $\mathcal{Q}=\alpha_2+\beta_2+3\gamma_2$ is ruled out by letting $t\rightarrow 0$ in this case.

If $g(t)$ lies on $\mathcal{L}_3$, then we obtain the perpendicular condition from \eqref{eq-L3}. 
$F$ belongs to the exceptional transversal family by Theorem \ref{thm}. \end{proof}

\section{Digression into Dicyclic Groups}\label{digress} A dicyclic group $Dic_n$ of order $4n, n\geq 3,$ is one generated by two elements $x$ and $y$ with the presentation $x^{2n}=1, y^2=x^n, yxy^{-1}=x^{-1}$. The last two relations give that all elements in $Dic_n$ are of the form $x^k$ or $x^ky, 0\leq k\leq 2n-1$. We have the product rule
\begin{equation}\label{product-rule}{\small
(x^ky)(x^ly)=x^{k-l+n},\quad (x^{l}y)(x^k)=x^{l-k}y,\quad x^kyx^{-k}=x^{2k}y.}
\end{equation}
The first two of which follow from the last one that is in turn an inductive result of rewriting $y^2=x^n$ as $ y^{-1}=x^{-n}y.$ The first identity implies that $(x^ly)^2=x^n$ so that its order is $4$.

For each pair $(a,b)\in  {\mathbb Z}_{2n}^{\times}\times {\mathbb Z}_{2n}$, where ${\mathbb Z}_{2n}^{\times}$ is the multiplicative group modulo $2n$ consisting of all $k\in {\mathbb Z}_{2n}$ prime to $2n$, define the map $\phi:Dic_n\rightarrow Dic_n$ by 
\begin{equation}\label{automorphism}{\small
x^m\mapsto x^{am},\quad x^ky\mapsto x^{ak+b}y.}
\end{equation}
Then it is checked that $\phi$ is an automorphism of $Dic_n$ satisfying $\phi : x\mapsto x^a$ and $y\mapsto x^b y$. Conversely, any automorphism $\phi$ must map $x$ to $x^a$ for some $a$ prime to $2n$, because as mentioned above $x^ly$ has order $4$ whereas $\phi(x)$ has order $2n>4$, while $\phi$ maps $y$ in general to $x^by$ for some $b$ for the same reason. 

We conclude by \eqref{automorphism} that the automorphism group of $Dic_n$ is the affine transformation group consisting of the affine transformations $L_{(a,b)}(t) := at+b, (a,b)\in {\mathbb Z}_{2n}^{\times}\times {\mathbb Z}_{2n}.$ Moreover, all inner automorphisms are those of the form $L_{(\pm 1, 2m)}$, corresponding to the conjugation $h\mapsto ghg^{-1}, h\in G$ for $g=x^m$ or $g=x^my$ associated with the sign $\pm$, respectively. The outer automorphism group is defined by the equivalence relation
\begin{small}$$
L_{(a,b)}\sim L_{(c,d)}\quad \text{if and only if}\quad d = \pm b \; (\text{mod} \;2)\;\text{and}\; c=\pm a.
$$\end{small}

\section{Dihedral $D_3$ of order $6$}

In \eqref{diagram3}, assume that $\deg \varphi=6$ and $G=D_3$. Recall that the invariants of $D_3$ are generated by the three fundamental invariants $\alpha_3,\beta_3,
\gamma$ with the relation given in Table 1.  
Up to M\"obius transformations, by Table \ref{Rational Galois Coverings}, $\varphi:\mathbb{P}^1\rightarrow \mathbb{P}^1$ is given by
$$\varphi:w\longmapsto z=-\alpha_3^2/\gamma^3.$$ 
The associated Galois group of $\varphi$ consists of three elements generated by 
\begin{equation}\label{D3GaloisGroup}{\small
\sigma: \mathbb{P}^1\rightarrow \mathbb{P}^1,\quad\quad w\mapsto \exp(-2\pi\sqrt{-1}/3)\cdot w;\phantom{1111}\tau: \mathbb{P}^1\rightarrow \mathbb{P}^1,\quad\quad w\mapsto 1/w.}
\end{equation}
The ramified points of $\varphi$ are 
\begin{small}$$(0,\infty),~(1,\eta,\eta^2),~(-1,-\eta,-\eta^2),$$\end{small}
with multiplicities $3,~2,~2$ respectively, where $\eta:=\exp(2\pi i/3)$. They are the zeros of $\gamma,~\alpha_3,~\beta_3$, respectively, and are mapped to $\infty,~0$ and $1$, respectively, under $\varphi$. Moreover, $\pm1$ are fixed under $\tau$. 

On the left-hand side of \eqref{eigenspace problem}, in correspondence to \eqref{D3GaloisGroup}, we select $B:=T_{2,2}$ and $A:=\Diag\{\exp(\pi i/3),\exp(-\pi i/3)\}$ as double coverings of 
$\tau$ and $\sigma$, respectively.
They generate a subgroup \begin{small}$$G^\ast\triangleq \{A^k,~BA^k,~0\leq k\leq 5\},\quad A^6 = 1,\; B^2= A^3,\; BAB^{-1}=A^{-1},$$\end{small} in $SU(2)$ isomorphic to the dicyclic group of order 12. It has $6$ conjugacy classes.
{\scriptsize\[\{\Id\},\quad\{-\Id\},\quad\{A,A^5\},\quad\{A^2,A^4\},\quad\{B,BA^2,BA^4\},\quad\{BA,BA^3,BA^5\},\]}
\!\!Hence, $G^\ast$ has $6$ irreducible representations.
{\scriptsize\begin{table}[H]
\begin{center}
\begin{tabular}{|c|c|c|c|c|c|c|c|}
\hline
         &  $\Id$ & $-\Id$ & $\{A,A^5\}$ & $\{A^2,A^4\}$ & $\{BA^{2k}\}$ & $\{BA^{2k+1}\}$ \\ \hline
$\psi_1$ &  $1$   & $1$    & $1$               & $1$               & $1$         & $1$      \\ \hline
$\psi_2$ &  $1$   & $-1$    & $-1$               & $1$               & $i$ & $-i$               \\ \hline
$\psi_3$ &  $1$   & $1$    & $1$               & $1$               & $-1$      & $-1$         \\ \hline
$\psi_4$ &  $1$   & $-1$    & $-1$               & $1$               & $-i$ & $i$              \\ \hline
$\psi_5$     & $2$   & $2$    & $-1$               & $-1$               & $0$       & $0$        \\ \hline
$\psi_6$     & $2$   & $-2$    & $1$               & $-1$               & $0$       & $0$        \\ \hline
\end{tabular}
\end{center}
\end{table}}
Note that $H:=\{\Id,A^2,A^4\}$ is a normal subgroup of $G^\ast$, and $G^\ast$ is a semi-direct product of $H$ and $\{\Id,B,B^2,B^3\}\cong C_4$. So, the first four irreducible representations of degree $1$ are obtained by $G^\ast\rightarrow G^\ast/H=C_4$. Note that $\psi_2,\psi_3$ and $\psi_4$ also arise from 
the action of $G^\ast$ on the invariants $\beta_3,\gamma$, and $\alpha_3$, respectively. The last representation $\psi_6$ is obtained by multiplication of $G$ on the column vectors of $\mathbb{C}^2$ (also equal to $\rchi_{V_1}$). Finally, $\psi_5$ is obtained by the tensor product of $\psi_2$ with $\psi_6$. Under the natural basis $\{e_1,e_2\}$ of $\mathbb{C}^2$, we realize $\psi_5$ as the character of the natural action on the column vectors given by $\psi_2(A)A,~\forall A\in G^\ast$. 

Immediately, we obtain 
\begin{equation}\label{eq-chaD3}{\scriptsize
\begin{split}
\psi_1\cdot \psi_6&=\psi_3\cdot \psi_6=\psi_6,\quad\psi_2\cdot \psi_6=\psi_4\cdot \psi_6=\psi_5,\\
\psi_5\cdot\psi_6&=\psi_2+\psi_4+\psi_6,\quad \psi_6\cdot\psi_6=\psi_1+\psi_3+\psi_5.
\end{split}}
\end{equation}
Moreover, two irreducible subspaces of dimension $1$ in $\mathbb{C}^2\otimes \mathbb{C}^2$ (the case $\psi_6\cdot\psi_6$) are
{\small\begin{equation}\label{key decomposition for D3}{\scriptsize
W_1=\Span\{e_1\otimes e_2-e_2\otimes e_1\},\quad W_3=\Span\{e_1\otimes e_2+e_2\otimes e_1\}.}
\end{equation}}
\!\!This formula also holds for $\psi_5\cdot\psi_6$, such that the first invariant space has character $\psi_2$, and the other has character $\psi_4$.

To determine $\phi_2$, on the right-hand side of \eqref{eigenspace problem}, observe that up to conjugation, there is a unique $D_3$ in $S_4$, i.e., $H_1\triangleq \langle (124),(24)\rangle$, which is double covered by the dicyclic group \begin{small}$$G^*\triangleq \{T_{3,1}^k,~T_{2,1}T_{3,1}^k,~0\leq k\leq 5\},\quad T_{3,1}^6=1, \;T_{2,1}^2=T_{3,1}^3, \;T_{2,1}T_{3,1}T_{2,1}^{-1}=T_{3,1}^{-1},$$\end{small} in $S_4^*$. After projection, all elements of $G^*$, except for the unit element, either fix the face centers or the edge midpoints. So, all ramified points do not lie on the free lines. 

\begin{lemma}\label{first and second isom of D3}
Up to conjugation, there are two isomorphisms $\gamma_i$ from $G^\ast$ to $G^*$ determined by $\gamma_i(A)=T_{3,1},~\gamma_i(B)=(-1)^{i+1}T_{2,1}$, where $i=1,2$. The character of $G^\ast$ acting on column vectors in $\mathbb{C}^2$ by matrix multiplication on the left through $\gamma_i$ is $\psi_6$.
\end{lemma}
{\begin{proof} We work with $G^\ast$, whose outer automorphism group is ${\mathbb Z}_2$. Explicitly, with $n=3$, the data in Section \ref{digress} 
imply that the the outer automorphism group is generated by the automorphism with $(a, b)=(1,3)$, i.e., by $A\mapsto A$ and $B\mapsto A^3B=-B$.
\end{proof}}

\begin{theorem}
In the case $D_3$, let $g([x,y]):\mathbb{P}^1\rightarrow \mathbb{P}^3$ be a holomorphic curve satisfying diagram \eqref{diagram3}. 
Then $\deg g=7$, and up to M\"obius transformations on $\mathbb{P}^1$, 
$g$ is equivalent to
\begin{equation}\label{deg7 D3 case}{\small
g=\begin{pmatrix}
a_1 & a_2 & a_3\\
b_1 & b_2 & b_3
\end{pmatrix}\begin{pmatrix}
\alpha_3\beta_3(\frac{\nu_1}{\xi_1}y+x) & \gamma^3(\frac{\nu_1}{\xi_1}y-x) & \alpha_3^2(\frac{\nu_1}{\xi_1}y-x) \\
\alpha_3\beta_3\nu_1(\frac{\nu_2}{\xi_1}y+x) & \gamma^3\nu_1(\frac{\nu_2}{\xi_1}y-x) & \alpha_3^2\nu_1(\frac{\nu_2}{\xi_1}y-x) 
\end{pmatrix}^T,}
\end{equation}
where $\xi_1=\exp(\frac{\pi i}{4})$, and $\nu_1=\frac{-1+\sqrt{-1}}{1-\sqrt{3}},~\nu_2=\frac{-1+\sqrt{-1}}{1+\sqrt{3}}$. Moreover, {\scriptsize$ \begin{pmatrix}
a_1 & a_2 & a_3\\
b_1 & b_2 & b_3
\end{pmatrix}\in G(2,3)$} 
is the unique two-plane which is perpendicular to 
{\scriptsize $
{(\alpha_3\beta_3\nu_1(\frac{\nu_2}{\xi_1}y+x) ,\gamma^3\nu_1(\frac{\nu_2}{\xi_1}y-x) , \alpha_3^2\nu_1(\frac{\nu_2}{\xi_1}y-x))|_{(x,y)=(t,1)} },
$}
under the usual bilinear inner product. This class depends on one freedom $t$ and belongs to the exceptional transversal family.
\end{theorem}

\begin{proof}
We choose $\phi_2$ to be the isomorphism $\gamma_1$ stated in the Lemma \ref{first and second isom of D3}. The choice of $\gamma_2$ is similar. The action of $\phi_2$ on $\mathbb{C}^2$ is equivalent to the irreducible representation $\psi_6$ under the new basis $E_1\triangleq e_1+\nu_1e_2,~E_2\triangleq \frac{\nu_1}{\xi_1}(e_1+\nu_2e_2)$ (constructed by diagonalizing $T_{3,1}$), where \begin{small}$$\xi_1=\exp(\pi i/4),\quad\nu_1=(-1+\sqrt{-1})/(1-\sqrt{3}),\quad\nu_2=(-1+\sqrt{-1})/(1+\sqrt{3}).$$\end{small}

 It follows from  Corollary \ref{Cor4.1} that $6\leq \deg(g)\leq 9$. Since $\{g(0),g(\infty)\}$ are among the centers of faces, and $\{g(\eta^k),~0\leq k\leq 5\}$, where $\eta=\exp(\frac{\pi i}{3})$, are among the midpoints of edges, eight in all, by \eqref{degree-difference}, we have {\small$$\deg \mathcal{Q}-\deg \mathcal{F}=2\deg g-6(\deg g-6)=36-4\deg g\geq 8.$$} \!\!Therefore, $6\leq \deg g\leq 7$. Thus the cases $\deg g=8,9$ do not occur.

\textbf{Case (1)}: Assume that $\deg g=6$. 
Through $\phi_1$, we have 
$\rchi_{V_6}=\psi_1+2\psi_3+2\psi_5$ (computed by \eqref{formula to compute char on Vn} or by the Poincar\'e series \cite{Springer}), where 
\begin{equation}\label{W5}\medmath{
2W_5=\Span\{y\alpha_3\gamma,-x\alpha_3\gamma\}\oplus \Span\{y\beta_3\gamma,x\beta_3\gamma\},}
\end{equation}
constructed by the invariant $\alpha_3\gamma$ (vs. $\beta_3\gamma$) of degree $5$ with character $\psi_2$ (vs. $\psi_4$) and (2) in Proposition \ref{key basis prop}. 
We rule out this case because vectors in $2W_5$ have common factors.

\textbf{Case (2)}: Assume that $\deg g=7$. Through $\phi_1$, we have
{\small$\rchi_{V_7}=\psi_2+\psi_4+3\psi_6$}, where 
\begin{small}$$3W_6=\Span\{y\alpha_3\beta_3,-x\alpha_3\beta_3\}\oplus \Span\{y\gamma^3,x\gamma^3\}\oplus \Span\{y\alpha_3^2,x\alpha_3^2\}.$$\end{small}
constructed from invariants $\{\alpha_3\beta_3, \gamma^3, \alpha_3^3\}$ in $V_6$ by Proposition \ref{key basis prop}. It follows from \eqref{eq-chaD3} and \eqref{key decomposition for D3} that 
{\small$\rchi_{V_7\otimes \mathbb{C}^2}=\rchi_{V_7}\cdot \psi_6=3\psi_1+3\psi_3+5\psi_5$}, and the first two invariant subspaces are given by the following 
{\scriptsize\begin{align*}\label{direct sum decomposition D3 deg 7 first kind}
\begin{split}
3W_1=& \Span\{y\alpha_3\beta_3y\otimes E_2-(-x\alpha_3\beta_3)\otimes E_1,~y\gamma^3\otimes E_2-x\gamma^3\otimes E_1,~y\alpha_3^2\otimes E_2-x\alpha_3^2\otimes E_1\},\\
3W_3=& \Span\{y\alpha_3\beta_3y\otimes E_2+(-x\alpha_3\beta_3)\otimes E_1,~y\gamma^3\otimes E_2+x\gamma^3\otimes E_1,~y\alpha_3^2\otimes E_2+x\alpha_3^2\otimes E_1\}.
\end{split}
\end{align*}}
\!\!Since $3W_3$ differs from $3W_1$ only by a coordinate change from $y$ to $-y$, we focus our analysis on the space $3W_1$. A plane in $3W_1$ is spanned by the two rows of the matrix \eqref{deg7 D3 case} (after expanding $E_i$ into $e_j$). 
Note that 
\begin{small}$$\det g=\frac{2(1+\sqrt{-1})\sqrt{6}}{\sqrt{3}-1}\alpha_3\beta_3\gamma(p_{12}\gamma^3+p_{13}\alpha_3^2).$$\end{small}

By Lemma \ref{linear system of D2} and the fact that zeros of $\alpha_3,~\beta_3,~\gamma$ do not lie on the free lines, we must have {\small$\mathcal{Q}=\alpha_3+\beta_3+\gamma+\varphi^{-1}(\varphi(t))$}, where $\varphi^{-1}(\varphi(t))$ consists of $6$ distinct points which is the principal orbit of $D_3$ through the point $t$, and {\small$\mathcal{F}=\varphi^{-1}(\varphi(t))$} (here, $t\in\mathbb{C}\setminus\{0,1\}$ is a solution of $p_{12}\gamma^3+p_{13}\alpha_3^2=0$). 

Note that $g(t)$ lies on one of the six free lines ${\mathcal L}_i$. Meanwhile, the action of multiplying {\small $T_{3,1}^T$} on the right gives a permutation $(145)(263)$, and the action of multiplying {\small$T_{2,1}^T$} on the right gives a permutation $(12)(34)(56)$ on the free lines. Therefore, 
we may assume that $g(t)$ lies on ${\mathcal L}_1$. Then the perpendicular conclusion follows from \eqref{eq-L12}.
The sextic curve $F$ downstairs is computed by {\scriptsize$g\cdot uv(u^4-v^4)/(x^3-t^3y^3)(x^3-t^{-3}y^3)$} and invariant theory. By Theorem~\ref{thm}, $F$ belongs to the exceptional transversal family.
\end{proof}

\section{Dihedral $D_4$ of order $8$}

In \eqref{diagram3}, assume that $\deg \varphi=8$ and $G=D_4$. Recall that the invariants of $D_4$ are generated by the three fundamental invariants $\alpha_4,\beta_4,
\gamma$ with the relation given in Table 1. 
Up to M\"obius transformations, by Table \ref{Rational Galois Coverings}, $\varphi:\mathbb{P}^1\rightarrow \mathbb{P}^1$ is given by
 $$\varphi:w\longmapsto z=-\alpha_4^2/\gamma^4=-(w^4-2+\frac{1}{w^4})/4.$$ Thus the Galois group of $\varphi$ consists of three elements generated by 
{\small\begin{equation}\label{D4GaloisGroup}
\sigma: \mathbb{P}^1\rightarrow \mathbb{P}^1,\quad\quad w\mapsto -\sqrt{-1}\cdot w;\phantom{1111}\tau: \mathbb{P}^1\rightarrow \mathbb{P}^1,\quad\quad w\mapsto 1/w.
\end{equation}}
\!\!The ramified points of $\varphi$ are 
{\small\[(0,\infty),~(\pm 1,\pm i),~(\pm \xi_1,\pm \xi_3),\]}
\!\!with multiplicities $4,~2,~2$ respectively, where $\xi_k:=\exp(k\pi i/4)$. They are the zeros of $\gamma,~\alpha_4,~\beta_4$, respectively, and are mapped to $\infty,~0$ and $1$, respectively, under $\varphi$. Moreover, $\pm1$ are fixed under $\tau$, $\pm\xi_1$ are fixed under $\tau\sigma$. Up to M\"obius transformations acting on both copies of $\mathbb{P}^1$, we can interchange $(\pm 1, \pm i)$ and $(\pm \xi_1,\pm\xi_3)$. 

On the left-hand side of \eqref{eigenspace problem}, in correspondence to \eqref{D3GaloisGroup}, we select $B\triangleq T_{2,2}$ and $A\triangleq T_{1,1}$ as double coverings of 
$\tau$ and $\sigma$, respectively.
Note that they generate a subgroup \begin{footnotesize}
$$G^*\triangleq\{T_{1,1}^k,~T_{2,2}T_{1,1}^k,~0\leq k\leq 7\}=\{T_{1,k},~T_{2,k},~0\leq k\leq 7\},\quad T_{1,1}^8=1,\, T_{2,2}^2=T_{1,1}^4, \,T_{2,2}T_{1,1}T_{2,2}^{-1}=T_{1,1}^{-1},$$\end{footnotesize} the generalized quaternion group of order 16, in $S_4^*$ . It has $7$ conjugacy classes 
{\scriptsize\[\{\Id\},\{-\Id\},\{T_{1,1},T_{1,1}^7\},\{T_{1,1}^2,T_{1,1}^6\},\{T_{1,1}^3,T_{1,1}^5\},\{T_{2,2}T_{1,1}^{2k},~0\leq k\leq 3\},\{T_{2,2}T_{1,1}^{2k+1},~0\leq k\leq 3\}\]}
\!\!and $7$ irreducible representations.
{\scriptsize\begin{table}[H]
\begin{center}
\begin{tabular}{|c|c|c|c|c|c|c|c|c|}
\hline
         &  $\Id$ & $-\Id$ & $\{T_{1,1},T_{1,1}^7\}$ & $\{T_{1,1}^2,T_{1,1}^6\}$ & $\{T_{1,1}^3,T_{1,1}^5\}$ & $\{T_{2,2}T_{1,1}^{2k}\}$ & $\{T_{2,2}T_{1,1}^{2k+1}\}$ \\ \hline
$\psi_1$ &  $1$   & $1$    & $1$               & $1$               & $1$         & $1$    & $1$  \\ \hline
$\psi_2$ &  $1$   & $1$    & $1$               & $1$               & $1$ & $-1$ & $-1$               \\ \hline
$\psi_3$ &  $1$   & $1$    & $-1$               & $1$               & $-1$      & $1$ & $-1$         \\ \hline
$\psi_4$ &  $1$   & $1$    & $-1$               & $1$               & $-1$ & $-1$ & $1$              \\ \hline
$\psi_5$     & $2$   & $2$    & $0$               & $-2$               & $0$       & $0$  & $0$      \\ \hline
$\psi_6$     & $2$   & $-2$    & $\sqrt{2}$               & $0$               & $-\sqrt{2}$       & $0$   & $0$     \\ \hline
$\psi_7$     & $2$   & $-2$    & $-\sqrt{2}$               & $0$               & $\sqrt{2}$       & $0$   & $0$     \\ \hline
\end{tabular}
\end{center}
\end{table}}
So, the first five irreducible representations are obtained by $G^*\rightarrow G^*/\{\pm\Id\}=D_4$. Note that $\psi_2,\psi_3$, and $\psi_4$ also arise from 
the action of $G^*$ on the invariants $\gamma, \beta_4$ and $\alpha_4$, respectively. The representation $\psi_6$ is obtained by multiplications of $G^*$ on the column vectors of $\mathbb{C}^2$. Finally, $\psi_7$ is obtained by the tensor product of $\psi_3$ with $\psi_6$. Under the natural basis $\{e_1,e_2\}$ of $\mathbb{C}^2$, we realize $\psi_7$ as the character of the natural action on the column vectors given by $\psi_3(A)A,~\forall A\in G^*$. 
Immediately, we obtain 
\begin{equation}\label{multi table for D4}{\scriptsize
\begin{split}
\psi_1\cdot \psi_6&=\psi_2\cdot \psi_6=\psi_6,\quad\psi_3\cdot \psi_6=\psi_4\cdot \psi_6=\psi_7,\\
\psi_5\cdot\psi_6&=\psi_6+\psi_7,\quad\psi_6\cdot\psi_6=\psi_1+\psi_2+\psi_5,\quad\psi_7\cdot\psi_6=\psi_3+\psi_4+\psi_5.
\end{split}}
\end{equation}
Moreover, two irreducible subspaces of dimension $1$ in {\small$\mathbb{C}^2\otimes \mathbb{C}^2$} (the case {\small$\psi_6\cdot\psi_6$}) are
{\small\begin{equation}\label{key decomposition for D4}
W_1=\Span\{e_1\otimes e_2-e_2\otimes e_1\},\quad W_2=\Span\{e_1\otimes e_2+e_2\otimes e_1\}.
\end{equation}}
\!\!This formula also holds for $\psi_7\cdot\psi_6$, such that the first invariant space has character $\psi_3$, and the other has character $\psi_4$.

On the right-hand side of \eqref{eigenspace problem}, 
up to conjugation, there is a unique $D_4$ in $S_4$ given by
\begin{small}$$H_1\triangleq\langle (14)(23),(12)(34),(24)\rangle.$$\end{small}
Note that $H_1$ is double covered by the same group $G^*$ above.  

\begin{lemma}\label{four isoms of D4}
Up to conjugation, there are four isomorphisms $\gamma_i$ from $G^*$ to $G^*$ determined by $\gamma_1=\Id$, and
\begin{footnotesize}
$$\aligned\gamma_{2}(T_{1,1})=T_{1,1},\;\gamma_{2}(T_{2,2})=T_{2,3};\quad \gamma_{3}(T_{1,1})=-T_{1,1},\;\gamma_{3}(T_{2,2})=T_{2,2};\quad 
\gamma_{4}(T_{1,1})=-T_{1,1},\;\gamma_{4}(T_{2,2})=T_{2,3}.
\endaligned$$\end{footnotesize}
Only one pair $\{g(\pm 1),g(\pm i)\}$ or $\{g(\pm\xi_1),g(\pm \xi_3)\}$ lies on the free lines. 
If we assume that 
$\{g(\pm 1),g(\pm i)\}$ lies on the free lines,
then only two isomorphisms, $\gamma_1$ and $\gamma_3$, remain. The characters of $G^*$ acting on 
$\mathbb{C}^2$ by matrix multiplication on the left are $\psi_6$ and $\psi_7$, respectively.
\end{lemma}
\begin{proof} The outer automorphism group of $G^*$ is the Klein-4 group. Explicitly, with $n=4$ the data in Section \ref{digress} 
imply that the outer automorphism group of $G^*$ is generated by the four automorphisms given by $(a,b)= (1,0),(1,1),(5,0),(5,1)$, which are the ones above.
\end{proof}

\begin{theorem}
In the case of $D_4$, let $g([x,y]):\mathbb{P}^1\rightarrow \mathbb{P}^3$ be a holomorphic curve satisfying diagram \eqref{diagram3}. Then up to M\"obius transformations on $\mathbb{P}^1$ and the $PSL_2$-equivalence, we have $\deg g =9$ or $11$. Moreover,  the following classification holds. 
\begin{enumerate}
\item[\rm{(1)}] If $\deg g=9$, then 
{\small\begin{equation}\label{deg 9 D4 case}
g=\begin{pmatrix}
a_1 & a_2 & a_3 \\
b_1 & b_2 & b_3
\end{pmatrix}\begin{pmatrix}
x\alpha_4^2  & x\gamma^4 & -\alpha_4\gamma\frac{\partial \beta_4}{\partial y}\\
y\alpha_4^2 & y\gamma^4 & \alpha_4\gamma\frac{\partial \beta_4}{\partial x}
\end{pmatrix}^T.
\end{equation}}
This class depends on two freedoms and belongs to the generally ramified family.

\item[\rm{(2)}] If  $\deg g=11$, then
\begin{equation}\label{deg 11 D4 case}{\small
g=\begin{pmatrix}
a_1 & a_2 & a_3 \\
b_1 & b_2 & b_3
\end{pmatrix}\begin{pmatrix}
-x\alpha_4\gamma^3 & x\beta_4\gamma^3 & \alpha_4\beta_4\frac{\partial \alpha_4}{\partial y} \\
y\alpha_4\gamma^3 & y\beta_4\gamma^3 & \alpha_4\beta_4\frac{\partial \alpha_4}{\partial x}
\end{pmatrix}^T,}
\end{equation}
where {\scriptsize$\begin{pmatrix}
a_1 & a_2 & a_3 \\
b_1 & b_2 & b_3
\end{pmatrix}\in G(2,3)$} is the unique plane which is perpendicular to 
\begin{equation}\begin{small}\label{perp condition deg 11 D4}
{(-(t+1)^2(t^2+1)t^3,(t^4+1)t^3,-(t^3+1)(t^4+1)(t^3+t^2+t+1))},\end{small}
\end{equation}
under the bilinear inner product. This class depends on one freedom $t$ and belongs to the exceptional transversal family.
\end{enumerate}
\end{theorem}

\begin{proof} We choose $\phi_2$ to be the isomorphism $\gamma_1$ stated in the Lemma \ref{four isoms of D4}. 
The choice of $\gamma_3$ is similar. 

It follows from  Corollary \ref{Cor4.1} that $6\leq \deg(g)\leq 9$. Since $\{g(0),g(\infty)\}$ and $\{g(\pm 1),g(\pm i)\}$ lie on the free lines, six in all, while $\{g(\pm \xi_1),g(\pm \xi_3)\}$ do not, we conclude $\deg \mathcal{F}\geq 6$, and 
\[\deg \mathcal{Q}-\deg \mathcal{F}=2\deg g-6(\deg g-8)=48-4\deg g\geq 4.\] Therefore $9\leq \deg g\leq 11$. 
Proposition \eqref{poincareEven} allows us to exclude the case when 
$\deg g=10$. 

\textbf{Case (1)}: Assume that $\deg g=9$. Then $\deg \mathcal{F}=6$ and $\mathcal{F}=(0)+(\infty)+(1)+(i)+(-1)+(-i)$. 
Through $\phi_1$, we have $\rchi_{V_9}=3\psi_6+2\psi_7$, where 
\begin{small}$$3W_6=\Span\{y\gamma^4,-x\gamma^4\}\oplus \Span\{y\alpha_4^2,-x\alpha_4^2\}\oplus \Span\{\alpha_4\gamma\frac{\partial \beta_4}{\partial x},\alpha_4\gamma\frac{\partial \beta_4}{\partial y}\},$$\end{small}
is constructed by invariant $\gamma^4,~\alpha_4^2,~\alpha_4\gamma$ of degree $8,~8,~6$ and with characters $\psi_1,~\psi_1$ and $\psi_3$, respectively (both (1) and (2) in Proposition \ref{key basis prop} are used). 
Similarly, we have \begin{scriptsize}$$2W_7=\Span\{y\alpha_4\gamma^2,x\alpha_4\gamma^2\}\oplus \Span\{y\beta_4\gamma^2,-x\beta_4\gamma^2\},$$\end{scriptsize}
which is excluded as its vectors have common factors.
It follows from \eqref{key decomposition for D4} that {\small$3W_6\otimes \mathbb{C}^2=3W_1\oplus 3W_2$}, where 
{\scriptsize \begin{align*}
3W_1=& \Span\{y\gamma^4\otimes e_2-(-x)\gamma^4\otimes e_1,~y\alpha_4^2\otimes e_2-(-x)\alpha_4^2\otimes e_1,~\alpha_4\gamma\frac{\partial \beta_4}{\partial x}\otimes e_2-\alpha_4\gamma\frac{\partial \beta_4}{\partial y}\otimes e_1\},\\
3W_2=& \Span\{y\gamma^4\otimes e_2+(-x)\gamma^4\otimes e_1,~y\alpha_4^2\otimes e_2+(-x)\alpha_4^2\otimes e_1,~\alpha_4\gamma\frac{\partial \beta_4}{\partial x}\otimes e_2+\alpha_4\gamma\frac{\partial \beta_4}{\partial y}\otimes e_1\}.
\end{align*}}
\!\!Since $3W_2$ differs from $3W_1$ only by a coordinate change from $y$ to $-y$, we focus our analysis on the space $3W_1$. A plane $g(x,y)$ in $3W_1$ is spanned by the two rows of the matrix \eqref{deg 9 D4 case}. 
Note that 
\begin{small}$$\det g=4\alpha_4\beta_4\gamma(p_{13}\alpha_4^2+p_{23}\gamma^4).$$\end{small} By Lemma \ref{linear system of D2}, we know that $\mathcal{Q}$ equals one of {\small$3(\alpha_4)+(\beta_4)+(\gamma),~(\alpha_4)+3(\beta_4)+(\gamma),~(\alpha_4)+(\beta_4)+5(\gamma)$}, and the generic case {\small$(\alpha_4)+(\beta_4)+(\gamma)+\varphi^{-1}(\varphi(t))$}, where {\small$\varphi^{-1}(\varphi(t))$} consists of $8$ distinct points constituting the principal orbit of $D_4$ through the point $t$, where $t$ is a solution of {\small$p_{13}\alpha_4^2+p_{23}\gamma^4=0$} with {\small$p_{13}p_{23}(p_{13}-p_{23})\neq 0$}.
We only deal with the generic case (the other three cases are its limits). 
The sextic curves $F$ are computed by {\scriptsize$g\cdot uv(u^4-v^4)/xy(x^4-y^4)$} and invariant theory. It is tangent to the $1$-dimensional orbit at $F(\varphi(t))$. 

\textbf{Case (2)}: Assume that $\deg g=11$. Then $\deg \mathcal{F}=18$ and {\small$\mathcal{F}\geq (0)+(\infty)+(1)+(i)+(-1)+(-i)$}. 
Through $\phi_1$, we have $\rchi_{V_{11}}=3\psi_6+3\psi_7$, where (constructed as Case (1)) 
\begin{scriptsize}$$\aligned &3W_7=\Span\{y\alpha_4\gamma^3,-x\alpha_4\gamma^3\}\oplus \Span\{y\beta_4\gamma^3,x\beta_4\gamma^3\}\oplus \Span\{\alpha_4\beta_4\frac{\partial \alpha_4}{\partial x},\alpha_4\beta_4\frac{\partial \alpha_4}{\partial y}\},\\
&3W_6=\Span\{y\alpha_4\beta_4\gamma,-x\alpha_4\beta_4\gamma\}\oplus \Span\{y\gamma^5,x\gamma^5\}\oplus \Span\{y\alpha_4^2\gamma,x\alpha_4^2\gamma\}.\endaligned$$\end{scriptsize}
We rule out $3W_6$ as vectors in it have common factors. It follows from \eqref{key decomposition for D4} that $3W_7\otimes \mathbb{C}^2=3W_3\oplus 3W_4$, where

{\scriptsize \begin{align*}
3W_3:=& \Span\{y\alpha_4\gamma^3\otimes e_2-(-x)\alpha_4\gamma^3\otimes e_1,y\beta_4\gamma^3\otimes e_2-x\beta_4\gamma^3\otimes e_1,\alpha_4\beta_4\frac{\partial \alpha_4}{\partial x}\otimes e_2-\alpha_4\beta_4\frac{\partial \alpha_4}{\partial y}\otimes e_1\},\\
3W_4:=& \Span\{y\alpha_4\gamma^3\otimes e_2+(-x)\alpha_4\gamma^3\otimes e_1,y\beta_4\gamma^3\otimes e_2+x\beta_4\gamma^3\otimes e_1,\alpha_4\beta_4\frac{\partial \alpha_4}{\partial x}\otimes e_2+\alpha_4\beta_4\frac{\partial \alpha_4}{\partial y}\otimes e_1\}.
\end{align*}}
\!\!Since $3W_3$ differs from $3W_4$ only by a coordinate change from  $y$ to $-y$, we focus our analysis on the space $3W_4$. A plane $g(x,y)$ in $3W_4$ is spanned by the two rows of the matrix \eqref{deg 11 D4 case}. 
It follows that 
\begin{small}$$\det g=2\alpha_4\beta_4\gamma^3(2(p_{23}-p_{13})\alpha_4^2+(2p_{23}-p_{12})\gamma^4).$$\end{small}
By Lemma \ref{linear system of D2} and zeros of $\beta_4$ do not lie on the free lines, we know that $\mathcal{Q}$ equals one of {\small$3\alpha_4+\beta_4+3\gamma,~\alpha_4+\beta_4+7\gamma$}, and the generic case {\small$\alpha_4+\beta_4+3\gamma+\varphi^{-1}(\varphi(t))$}, where {\small$\varphi^{-1}(\varphi(t))$} consists of $8$ distinct points constituting the principal orbit of $D_4$ through the point $t$.

We only deal with the generic case. Then {\small$\mathcal{F}=\alpha_4+3\gamma+\varphi^{-1}(\varphi(t))$}, where $t$ is a solution of {\small$2(p_{23}-p_{13})\alpha_4^2+(2p_{23}-p_{12})\gamma^4=0$}. 
The action of multiplying {\small$T_{1,1}^T$} on the right gives a permutation $(3654)$, while the action of multiplying {\small$T_{2,2}^T$} on the right gives a permutation $(12)(46)$ on the free lines. Therefore, 
we may assume that $g(t)$ lies on ${\mathcal L}_1$ or ${\mathcal L}_3$. In both cases, the curves $F$ are computed by $\frac{g\cdot uv(u^4-v^4)}{x^3y^3(x^4-y^4)(x^4-t^4y^4)(x^4-t^{-4}y^4)}$ and invariant theory.

If $g(t)$ lies on ${\mathcal L}_1$, then after making the normalization  $a_2=b_3=1,~a_3=b_2=0$, we can solve $a_1$ and $b_1$ as rational functions of $t$ from the perpendicular condition arising from \eqref{eq-L12}. A direct computation shows that $F$ sits in ${\mathbb P}^5$, to be ruled out. 
Actually, $F=[f_0:\cdots:f_6]$ 
satisfies a linear equation 
\begin{scriptsize}$$\frac{\sqrt{6}}{3(t^4-1)}f_0+f_1+\frac{\sqrt{10}(t^4-1)}{5}f_2+\frac{\sqrt{30}}{30}(t^4-1)^2f_3=0.$$\end{scriptsize}
By letting $t\rightarrow 0$, the non-generic case that {\small$\mathcal{Q}=\alpha_4+\beta_4+7\gamma$} is also ruled out.

If $g(t)$ lies on ${\mathcal L}_3$, then we can obtain the perpendicular condition from \eqref{eq-L3}. 
$F$ belongs to the exceptional transversal family by Theorem \ref{thm}. The non-generic case that {\small$\mathcal{Q}=3\alpha_4+\beta_4+3\gamma$} 
is obtained by letting $t\rightarrow 1$ 
in this generic case.  
\end{proof}

\section{$A_4$ of order $12$}

In \eqref{diagram3}, assume that $\deg \varphi=12$ and $G=A_4$. 
Recall that the invariants of $A_4$ are generated by the three fundamental invariants $\Phi, \Psi, \Omega$ with the relation given in Table \ref{invariants and relation}.
Up to M\"obius transformations, by Table \ref{Rational Galois Coverings}, $\varphi:\mathbb{P}^1\rightarrow \mathbb{P}^1$ is given by
 $$\varphi:w\longmapsto z=\Phi^3/\Psi^3=(\frac{w^4-2\sqrt{-3}w^2+1}{w^4+2\sqrt{-3}w^2+1})^3.$$ 
 The associated Galois group of $\varphi$ is generated by 
{\small\begin{equation}\label{A4GaloisGroup}
\sigma: \mathbb{P}^1\rightarrow \mathbb{P}^1,\quad\quad w\mapsto i(w+1)/(w-1);\phantom{1111}\tau: \mathbb{P}^1\rightarrow \mathbb{P}^1,\quad\quad w\mapsto -w.
\end{equation}}
\!\!The ramified points of $\varphi$ are three pairs of zeros of $\Phi,~\Psi$ and $\Omega$ with multiplicities $3,3$ and $2$, respectively. 
They are mapped to $0,\infty$ and $1$, respectively, under $\varphi$. 

On the left-hand side of \eqref{eigenspace problem}, in correspondence to \eqref{A4GaloisGroup}, we select $A:=T_{4,1}$ and $B:=T_{1,2}$ as double coverings of 
$\sigma$ and $\tau$, respectively. 
Note that they generate a subgroup \begin{small}$$G^*\triangleq\{T_{4,1}^k,T_{1,2}T_{4,1}^k,T_{4,1}T_{1,2}T_{4,1}^k,T_{4,1}^2T_{1,2}T_{4,1}^k,~0\leq k\leq 7\},$$\end{small} of $SU(2)$ isomorphic to $SL(2,{\mathbb Z}_3)$. It has $7$ conjugacy classes 
{\scriptsize\begin{align*}
& \{\Id\},\{-\Id\},\quad\{\pm T_{1,2},\pm T_{2,0},\pm T_{2,2}\},\quad\{-T_{3,1},T_{3,3},T_{5,1},-T_{5,3}\}\\
& \{-T_{4,1},-T_{4,3},T_{6,1},T_{6,3}\},\quad \{T_{3,1},-T_{3,3},-T_{5,1},T_{5,3}\},\quad\{T_{4,1},T_{4,3},-T_{6,1},-T_{6,3}\};
\end{align*}}
\!\!hence $G$ has $7$ irreducible representations (where $\eta=\exp(\frac{i\pi}{3})$).
{\scriptsize\begin{table}[H]
\begin{center}
\begin{tabular}{|c|c|c|c|c|c|c|c|c|}
\hline
         &  $\Id$ & $-\Id$ & $\{\pm T_{1,2},\ldots\}$ & $\{-T_{3,1},\ldots\}$ & $\{-T_{4,1},\ldots\}$ & $\{T_{3,1},\ldots\}$ & $\{T_{4,1},\ldots\}$ \\ \hline
$\rchi_0$ &  $1$   & $1$    & $1$               & $1$               & $1$         & $1$    & $1$  \\ \hline
$\rchi_1$ &  $1$   & $1$    & $1$               & $\eta^4$               & $\eta^2$ & $\eta^4$ & $\eta^2$               \\ \hline
$\rchi_2$ &  $1$   & $1$    & $1$               & $\eta^2$               & $\eta^4$      & $\eta^2$ & $\eta^4$         \\ \hline
$\rchi_3$ &  $3$   & $3$    & $-1$               & $0$               & $0$ & $0$ & $0$              \\ \hline
$\rchi_4$     & $2$   & $-2$    & $0$               & $-1$               & $-1$       & $1$  & $1$      \\ \hline
$\rchi_5$     & $2$   & $-2$    & $0$               & $-\eta^4$               & $-\eta^2$       & $\eta^4$   & $\eta^2$     \\ \hline
$\rchi_6$     & $2$   & $-2$    & $0$               & $-\eta^2$               & $-\eta^4$       & $\eta^2$   & $\eta^4$     \\ \hline
\end{tabular}
\end{center}
\end{table}}
The first four irreducible representations are obtained by $G^*\rightarrow G^*/\{\pm\Id\}=A_4$. Note that $\chi_0,\chi_1$ and $\chi_2$ also arise from 
the action of $G^*$ on the invariants $\Omega, \Psi$ and $\Phi$, respectively. The representation $\rchi_4$ is obtained by multiplications of $G^*$ on the column vectors of $\mathbb{C}^2$ (also equal to $\rchi_{V_1}$). The representations $\rchi_5$ (respectively, $\rchi_6$) is obtained by the tensor product of $\rchi_4$ with $\rchi_1$ (respectively, $\rchi_2$). Under the natural basis $\{e_1,e_2\}$ of $\mathbb{C}^2$, we realize $\rchi_5$ (respectively, $\rchi_6$) as the character of the natural action on the column vectors given by $\rchi_1(A)A,~\forall A\in G^*$ (respectively, $\rchi_2(A)A$). 
Immediately, we obtain that 
{\scriptsize \begin{align}\label{mult table for A4 binary}
\begin{split}
\rchi_0\cdot \rchi_4&=\rchi_4,\quad\rchi_1\cdot\rchi_4=\rchi_5,\quad\rchi_2\cdot\rchi_4=\rchi_6,\\
\rchi_3\cdot\rchi_4&=\rchi_4+\rchi_5+\rchi_6,\quad\rchi_4\cdot\rchi_4=\rchi_0+\rchi_3,~\rchi_5\cdot\rchi_4=\rchi_1+\rchi_3,\quad\rchi_6\cdot\rchi_4=\rchi_2+\rchi_3, 
\end{split}
\end{align}}
\!\!where the unit representation in $\rchi_4\cdot\rchi_4$ is given by $\Span\{e_1\otimes e_2-e_2\otimes e_1\}$.  

On the right-hand side, there is a unique 
{\small$A_4\triangleq\langle (14)(23),(13)(24),(243)\rangle$} in $S_4$.
Note that $G^*$ is the only group in $S_4^*$ double covering $A_4$.

\begin{lemma}\label{unique isom of A4}
Up to conjugation 
by elements in $S_4^*$, 
$\Id$ is the only automorphism of $G^*$. 
\end{lemma}

\begin{proof} It is known that the outer automorphism group of $A_4^*$ is ${\mathbb Z}_2$. Since $G^*=A_4^*$ is normal in $S_4^*$ of index 2, any element $g\in S_4^*\setminus G^*$ induces an outer automorphism via the conjugation $\theta:h\mapsto gkg^{-1}, \forall k\in G^*$. 
This means geometrically that upon a $SL_2$ coordinate change of ${\mathbb C}^2$, we may assume $\theta$ is the identity map.
\end{proof}

\begin{theorem}
In the case $A_4$, let $g([x,y]):\mathbb{P}^1\rightarrow \mathbb{P}^3$ be a holomorphic curve satisfying diagram \eqref{diagram3}. Then up to M\"obius transformations on $\mathbb{P}^1$ and the $PSL_2$-equivalence, we have $\deg g =13$ or $15$. Moreover, the following classification holds. 
\begin{enumerate}
\item[\rm{(1)}] If $\deg g=13$, then 
{\small\begin{equation}\label{deg 13 A4 case}
g=\begin{pmatrix}
a_1 & a_2 & a_3\\
b_1 & b_2 & b_3
\end{pmatrix}\begin{pmatrix}
x\Omega^2 & x\Phi^3  & -\Omega \Phi \frac{\partial \Psi}{\partial y}\\
y\Omega^2 & y \Phi^3 & \Omega\Phi\frac{\partial \Psi}{\partial x}
\end{pmatrix}^T.
\end{equation}}
\!\!This class depends on two freedoms and belongs to the generally ramified family.

\item[\rm{(2)}] If $\deg g=15$, then
{\small\begin{equation}\label{deg 15 A4 case}
g=\begin{pmatrix}
a_1 & a_2 & a_3 \\
b_1 & b_2 & b_3
\end{pmatrix}\begin{pmatrix}
x\Omega\Phi^2 & -\Omega^2\frac{\partial \Psi}{\partial y} & -\Phi^3\frac{\partial \Psi}{\partial y} \\
y\Omega \Phi^2 & \Omega^2\frac{\partial \Psi}{\partial x} & \Phi^3 \frac{\partial \Psi}{\partial x}
\end{pmatrix}^T,
\end{equation}}
\!\!where {\scriptsize$\begin{pmatrix}
a_1 & a_2 & a_3 \\
b_1 & b_2 & b_3
\end{pmatrix}\in G(2,3)$} is the unique plane which is perpendicular to 
\begin{small}\begin{equation}\label{perp condition deg 15 A4 case}
\quad((t^4-1)(2\sqrt{-3}t^2-t^4-1)^2,4t^2(t^4-1)^2(\sqrt{-3}+t^2),-4(2\sqrt{-3}t^2-t^4-1)^3(\sqrt{-3}+t^2)),
\end{equation}
\end{small}
\!\!under the bilinear inner product. This class depends on one freedom $t$ and belongs to the generally ramified family.
\end{enumerate}
\end{theorem}

\begin{proof}
We choose $\phi_2$ to be the identity stated in the Lemma \ref{four isoms of D4}. 
 It follows from  Corollary \ref{Cor4.1} that $12\leq \deg(g)\leq 18$.
Since the ramified multiplicities of the zeros of $\Phi\Psi$ are all $3$, they are mapped to the center of faces of octahedron via $g$. Since $0$, one of the zeros of $\Omega$, is invariant under $\tau:w\rightarrow -w$, $g(0)$ lies on the free lines preserved by $T_{1,2}$. 
Hence, the zeros of $\Omega$ are all mapped to the vertices. Therefore, $\deg \mathcal{F}\geq 6$, and 
\[8\leq \deg \mathcal{Q}-\deg \mathcal{F}=2\deg g-6(\deg g-12)=72-4\deg g
.\] 
Thus $13\leq \deg g\leq 16$. 
By Proposition \ref{poincareEven}, we can further rule out the cases $\deg g=14,16$.

\textbf{Case (1)}: Assume that $\deg g=13$. Then $\deg \mathcal{F}=6$ and $\mathcal{F}=(\Omega)$. 
Through $\phi_1$, we have $\rchi_{V_{13}}=3\rchi_4+2\rchi_5+2\rchi_6$,  where 
\begin{scriptsize}$$3W_4=\Span\{y\Omega^2,-x\Omega^2\}\oplus \Span\{y\Phi^3,-x\Phi^3\}\oplus \Span\{\Omega\Phi\frac{\partial \Psi}{\partial x},\Omega\Phi\frac{\partial \Psi}{\partial y}\},$$\end{scriptsize}
is constructed by invariants $\Omega^2,~\Phi^3,~\Omega\Phi$ of degree $12,~12,~10$ with characters $\rchi_0,~\rchi_0$ and $\rchi_2$, respectively ((1), (2) in Proposition \ref{key basis prop} are also used). They are equivalent to $\mathbb{C}^2$ with character $\rchi_4$ when the first polynomial maps to $e_1$ and the second maps to $e_2$. 
Similarly, 
\begin{scriptsize}$$\aligned 2W_5=\Span\{y\Phi\Psi^2,-x\Phi\Psi^2\}\oplus \Span\{\Omega\Phi\frac{\partial \Phi}{\partial x},\Omega\Phi\frac{\partial \Phi}{\partial y}\},\quad2W_6=\Span\{y\Phi^2\Psi,-x\Phi^2\Psi\}\oplus \Span\{\Omega\Psi\frac{\partial \Psi}{\partial x},\Omega\Psi\frac{\partial \Psi}{\partial y}\},\endaligned$$\end{scriptsize}
both of which are excluded as their vectors have common factors.
It follows from \eqref{mult table for A4 binary} that {\small$3W_4\otimes \mathbb{C}^2=3W_0\oplus 3W_3$}, where only 
{\scriptsize \begin{align*}\label{direct sum decomposition A4 deg 13}
\begin{split}
3W_0=& \Span\{y\Omega^2\otimes e_2-(-x)\Omega^2\otimes e_1,~y\Phi^3\otimes e_2-(-x)\Phi^3\otimes e_1,~\Omega\Phi\frac{\partial \Psi}{\partial x}\otimes e_2- \Omega\Phi\frac{\partial \Psi}{\partial y}\otimes e_1\}
\end{split}
\end{align*}}
\!\!needs to be considered. A plane $g(x,y)$ in $3W_0$ is spanned by the two rows of the matrix \eqref{deg 13 A4 case} with
\begin{small}$$\det g=4\Omega\Phi\Psi((\frac{-1}{12\sqrt{-3}}p_{13}+p_{23})\Phi^3+\frac{p_{13}}{12\sqrt{-3}}\Psi^3).$$\end{small}
It follows from Lemma \ref{linear system of D2} that {\small$\mathcal{Q}=3(\Omega)+(\Phi)+(\Psi)$}, or {\small$(\Omega)+4(\Phi)+(\Psi)$}, or {\small$(\Omega)+(\Phi)+4(\Psi)$}, or the generic case {\small$(\Omega)+(\Phi)+(\Psi)+\varphi^{-1}(\varphi(t))$}, where {\small$\varphi^{-1}(\varphi(t))$} consists of $12$ distinct points constituting the principal orbit of $A_4$ through the point $t$. Here, $t$ is a solution of {\small$(12\sqrt{-3}p_{23}-p_{13})\Phi^3+{p_{13}}\Psi^3=0.$} In all cases, we obtain the desired  curves 
downstairs by computing {\scriptsize$g\cdot uv(u^4-v^4)/xy(x^4-y^4)$} and using the invariant theory. It follows from Theorem~\ref{thm} that such curves always belong to the generally ramified family.  


\textbf{Case (2)}: Assume $\deg g=15$. Then $\deg \mathcal{F}=18$ and $\mathcal{F}>(\Omega)$. 
Through $\phi_1$, we have
{\small$\rchi_{V_{15}}=2\rchi_4+3\rchi_5+3\rchi_6$}, where  
\begin{scriptsize}$$\aligned&3W_5=\Span\{y\Omega\Phi^2,-x\Omega\Phi^2\}\oplus \Span\{\Omega^2\frac{\partial \Psi}{\partial x},\Omega^2\frac{\partial \Psi}{\partial y}\}\oplus\Span\{\Phi^3\frac{\partial \Psi}{\partial x},\Phi^3\frac{\partial \Psi}{\partial y}\},\\
&3W_6=\Span\{y\Omega\Psi^2,-x\Omega\Psi^2\}\oplus\Span\{\Omega^2\frac{\partial \Phi}{\partial x},\Omega^2\frac{\partial \Phi}{\partial y}\}\oplus \Span\{\Phi^3\frac{\partial \Phi}{\partial x},\Phi^3\frac{\partial \Phi}{\partial y}\},\\
&2W_4=\Span\{y\Omega\Phi\Psi,-x\Omega\Phi\Psi\}\oplus\Span\{\Phi^2\Psi\frac{\partial \Psi}{\partial x},\Phi^2\Psi\frac{\partial \Psi}{\partial y}\},\endaligned$$\end{scriptsize}
constructed as in Case (1). We rule out $2W_4$ as vectors in it have common factors. 
Tensoring $3W_3$ and $3W_6$ with $\mathbb{C}^2$, respectively, from (3) in Proposition \ref{key basis prop}, we obtain the following two invariant subspaces, 
{\scriptsize \begin{align*}
3W_1=& \Span\{y\Omega\Phi^2\otimes e_2-(-x)\Omega\Phi^2\otimes e_1,\Omega^2\frac{\partial \Psi}{\partial x}\otimes e_2-\Omega^2\frac{\partial \Psi}{\partial y}\otimes e_1,\Phi^3\frac{\partial \Psi}{\partial x}\otimes e_2-\Phi^3\frac{\partial \Psi}{\partial y}\otimes e_1\},\\
3W_2=& \Span\{y\Omega\Psi^2\otimes e_2+x\Omega\Psi^2\otimes e_1,\Omega^2\frac{\partial\Phi}{\partial x}\otimes e_2-\Omega^2\frac{\partial\Phi}{\partial y}\otimes e_1,\Phi^3\frac{\partial \Phi}{\partial x}\otimes e_2-\Phi^3\frac{\partial \Phi}{\partial y}\otimes e_1\}.
\end{align*}}
\!\!A plane $g(x,y)$ in $3W_1$ is spanned by the two rows of the matrix \eqref{deg 15 A4 case}. A plane $\widetilde{g}(x,y)$ in $3W_2$ is equivalent to $g(x,y)$ by $g(x,y)=\widetilde{g}(ix,-y)T_{1,1}$. So we use \eqref{deg 15 A4 case} in the following computation.
 Note that 
\begin{small}$$\det g=4\Omega\Phi^2\Psi((\frac{-p_{12}}{12\sqrt{-3}}+p_{13})\Phi^3+\frac{p_{12}}{12\sqrt{-3}}\Psi^3).$$\end{small}
It follows from Lemma \ref{linear system of D2} and the fact that the zeros of $\Phi$ and $\Psi$ do not lie on the free lines that {\small$\mathcal{Q}=3(\Omega)+2(\Phi)+(\Psi)$}, or the generic case {\small$\mathcal{Q}=(\Omega)+2(\Phi)+(\Psi)+\varphi^{-1}(\varphi(t))$}, where {\small$\varphi^{-1}(\varphi(t))$} consists of $12$ distinct points constituting the principal orbit of $A_4$ through the point $t$, where $t$ is a solution of  {\small$(12\sqrt{-3}p_{13}-p_{12})\Phi^3+{p_{12}}\Psi^3=0$}.

We only discuss the generic case. It follows from {\small$\mathcal{F}=(\Omega)+\varphi^{-1}(\varphi(t))$} that $g(t)$ lies on one of  the six free lines. The action of multiplying {\small$T_{4,1}^T$} on the right gives a permutation $(136)(254)$, while the action of multiplying {\small$T_{1,2}^T$} on the right gives a permutation $(35)(46)$ on the six free lines; so, $A_4$ acts on the six free lines transitively. Therefore, interchanging $t$ with another point of {\small$\varphi^{-1}(\varphi(t))$}, we may assume that $g(t)$ lies on the free line $\mathcal{L}_1$. Then we obtain the perpendicular conclusion \eqref{perp condition deg 15 A4 case} from \eqref{eq-L12}. 

The curves $F$ downstairs are computed by{\scriptsize $g\cdot uv(u^4-v^4)/xy(x^4-y^4)h(x,y)$} and invariant theory, where {\small${h(x,y):=(12\sqrt{-3}p_{13}-p_{12})\Phi^3+p_{12}\Psi^3}$} is a polynomial of degree $12$ whose zeros are all points of $\varphi^{-1}(\varphi(t))$. Moreover, it follows from Theorem~\ref{thm} that $F$ is tangential to the $1$-dimensional orbit at $z=0$.
The case {\small$\mathcal{Q}=3(\Omega)+2(\Phi)+(\Psi)$} is obtained by letting $t\rightarrow 0$. 
\end{proof}


\section{$S_4$ of order $24$}\label{sec-S4}

In \eqref{diagram3}, assume that $\deg \varphi=24$ and $G=S_4$. 
Recall that the invariants of $A_4$ are generated by the three fundamental invariants $(\Phi^3+\Psi^3)/2, \Phi\Psi, \Omega$ with the relation given in Table \ref{invariants and relation}.
Up to M\"obius transformations, by Table \ref{Rational Galois Coverings}, $\varphi:\mathbb{P}^1\rightarrow \mathbb{P}^1$ is given by
 $$\varphi:w\longmapsto z=\Phi^3\Psi^3/108\Omega^4=\frac{(w^8+14w^4+1)^3}{108(w(w^4-1))^4}.$$ The associated Galois group of $\varphi$ is generated by 
{\small\begin{equation}\label{S4GaloisGroup}
\sigma: \mathbb{P}^1\rightarrow \mathbb{P}^1,\quad\quad w\mapsto -\sqrt{-1}\cdot w;\phantom{1111}\tau: \mathbb{P}^1\rightarrow \mathbb{P}^1,\quad\quad w\mapsto -(w+1)/(w-1).
\end{equation}}
\!\!The ramified points of $\varphi$ are three pairs of zeros of $\Omega,~\Phi\Psi$ and $(\Phi^3+\Psi^3)/2$ with multiplicities $4,3$ and $2$, respectively. 
They are mapped to $\infty,0$ and $1$, respectively, under $\varphi$.

On the left-hand side of \eqref{eigenspace problem}, in correspondence to \eqref{A4GaloisGroup}, we select $A:=T_{1,1}$ and $B:=T_{3,0}$ as double coverings of 
$\sigma$ and $\tau$, respectively.   
Note that they generate the group \begin{small}$$G^*\triangleq\{T_{1,1}^k,T_{3,0}T_{1,1}^k,T_{3,0}^2T_{1,1}^k,T_{3,0}^3T_{1,1}^k,T_{1,1}T_{3,0}T_{1,1}^k,T_{1,1}T_{3,0}^3T_{1,1}^k,~0\leq k\leq 7\}=S_4^*$$\end{small} of $SU(2)$. It has $8$ conjugacy classes 
{\scriptsize\begin{align*}
& \{\Id\},\quad\{-\Id\},\quad\{-T_{3,1},T_{3,3},-T_{4,1},-T_{4,3},T_{5,1},-T_{5,3},T_{6,1},T_{6,3}\}\\
& \{\pm T_{1,2},\pm T_{2,0},\pm T_{2,2}\},\quad\{\pm T_{2,1},\pm T_{2,3},\pm T_{3,2},\pm T_{4,0}, \pm T_{5,2}, \pm T_{6,0}\},\\
&\{T_{3,1},-T_{3,3},T_{4,1},T_{4,3},-T_{5,1},T_{5,3},-T_{6,1},-T_{6,3}\}, \quad\{T_{1,1},-T_{1,3},T_{3,0},T_{4,2},-T_{5,0},-T_{6,2}\}\\
& \{-T_{1,1},T_{1,3},-T_{3,0},-T_{4,2},T_{5,0},T_{6,2}\},
\end{align*}}
\!\!with order $1,2,3,4,4,6,8,8$ respectively. Hence $G^*$ has $8$ irreducible representations.
{\scriptsize\begin{table}[H]
\begin{center}
\begin{tabular}{|c|c|c|c|c|c|c|c|c|c|}
\hline
         &  $\Id$ & $-\Id$ & $\{-T_{3,1}\}$ & $\{T_{1,2}\}$ & $\{T_{2,1}\}$ & $\{T_{3,1}\}$ & $\{T_{1,1}\}$  & $\{-T_{1,1}\}$\\ \hline
$\rchi_1$ &  $1$   & $1$    & $1$               & $1$               & $1$         & $1$    & $1$   & $1$ \\ \hline
$\rchi_2$ &  $1$   & $1$    & $1$               & $1$               & $-1$ & $1$ & $-1$ & $-1$               \\ \hline
$\rchi_3$ &  $2$   & $2$    & $-1$               & $2$               & $0$      & $-1$ & $0$ & $0$         \\ \hline
$\rchi_4$ &  $2$   & $-2$    & $-1$               & $0$               & $0$ & $1$ & $\sqrt{2}$  & $-\sqrt{2}$             \\ \hline
$\rchi_5$     & $2$   & $-2$    & $-1$               & $0$               & $0$       & $1$  & $-\sqrt{2}$ & $\sqrt{2}$    \\ \hline
$\rchi_6$     & $3$   & $3$    & $0$               & $-1$               & $1$       & $0$   & $-1$ & $-1$    \\ \hline
$\rchi_7$     & $3$   & $3$    & $0$               & $-1$               & $-1$       & $0$   & $1$ & $1$     \\ \hline
$\rchi_8$     & $4$   & $-4$    & $1$               & $0$               & $0$       & $-1$   & $0$ & $0$     \\ \hline
\end{tabular}
\end{center}
\end{table}}
The irreducible representations $\rchi_1,\rchi_2,\rchi_3,\rchi_6,\rchi_7$ are obtained by $S_4^\ast\rightarrow S_4^\ast/\{\pm\Id\}=S_4$. Note that $\chi_1$ and $\chi_2$ 
also arise from 
the action of $G^*$ on the invariants $\Phi\Psi$ and $\Omega$ (or $(\Phi^3+\Psi^3)/2$), respectively. The representation $\rchi_4$ is obtained by multiplication of $S_4^\ast$ on the column vectors of $\mathbb{C}^2$ (also equal to $\rchi_{V_1}$). The representation $\rchi_5$ is obtained by the tensor product of $\rchi_4$ with $\rchi_2$. The last representation $\rchi_8$ is realized as the action of $S_4^\ast$ on $V_{3}$. Under the natural basis $\{e_1,e_2\}$ of $\mathbb{C}^2$, we realize $\rchi_5$ as the character of the natural action on the column vectors given by $\rchi_2(A)A,~\forall A\in S_4^\ast$. 
Immediately, we obtain that 
\begin{small} \begin{align}\label{character table 2O}
\begin{split}
\rchi_1\cdot \rchi_4&=\rchi_4,\quad\rchi_2\cdot\rchi_4=\rchi_5,\quad\rchi_3\cdot\rchi_4=\rchi_8,\quad\rchi_4\cdot\rchi_4=\rchi_1+\rchi_7,\\
\rchi_5\cdot\rchi_4&=\rchi_2+\rchi_6,\quad\rchi_6\cdot\rchi_4=\rchi_5+\rchi_8,\quad\rchi_7\cdot\rchi_4=\rchi_4+\rchi_8,\quad\rchi_8\cdot\rchi_4=\rchi_3+\rchi_6+\rchi_7, 
\end{split}
\end{align}\end{small}
\!\!where the unit representation in $\rchi_4\cdot\rchi_4$ is given by $\Span\{e_1\otimes e_2-e_2\otimes e_1\}$. 

On the right hand side, we have the whole group $S_4$ double covered only by $S_4^\ast$ in $SU(2)$.

\begin{lemma}\label{two isoms of S4}
Up to conjugation, there are two isomorphisms $\gamma_i$ from $S_4^\ast$ to itself. Explicitly, $\gamma_1=\Id$, while $\gamma_2$ is determined by $\gamma_2(T_{1,1})=-T_{1,1}$ and $\gamma_2(T_{3,0})=-T_{3,0}$. Meanwhile, if $T$ lies in the conjugacy classes $\{T_{2,1}\},~\{T_{1,1}\},$ or $\{-T_{1,1}\}$, then $\gamma_2(T)=-T$; otherwise, $\gamma_2(T)=T$.
\end{lemma}

\begin{proof} Since the outer automorphism group of $S_4$ is trivial, the outer automorphism group of $S_4^*$ is ${\mathbb Z}_2$ with the generator given explicitly above, which is not inner since its character is different from that of $G^*$ acting on ${\mathbb C}^2$.
\end{proof}

\begin{theorem}
In the case $S_4$, let $g([x,y]):\mathbb{P}^1\rightarrow \mathbb{P}^3$ be a holomorphic curve satisfying diagram \eqref{diagram3}. 
up to M\"obius transformations on $\mathbb{P}^1$ and the $PSL_2$-equivalence, we have $\deg g = 25,$ $29,$ or $31$. Moreover, the following classification holds.
\begin{enumerate}
\item[\rm{(1)}] If $\deg g=25$, then 
{\small\begin{equation}\label{deg 25 S4}
g=\begin{pmatrix}
a_1 & a_2 & a_3\\
b_1 & b_2 & b_3
\end{pmatrix}\begin{pmatrix}
x\Omega^4 & x(\Phi\Psi)^3 & -(\Phi^3+\Psi^3)\Phi\Psi \frac{\partial \Omega}{\partial y} \\
y \Omega^4 & y(\Phi\Psi)^3  & (\Phi^3+\Psi^3)\Phi\Psi \frac{\partial \Omega}{\partial x}
\end{pmatrix}^T.
\end{equation}}
\!\!This class depends on two freedoms and belongs to the generally ramified family.
\item[\rm{(2)}] If $\deg g=29$, then 
{\small\begin{equation}\label{s4 deg 29}
g=\begin{pmatrix}
a_1 & a_2 & a_3 \\
b_1 & b_2 & b_3
\end{pmatrix}\begin{pmatrix}
x(\Phi^3+\Psi^3)(\Phi\Psi)^2 & -\frac{\partial \Omega^5 }{\partial y} & -(\Phi\Psi)^3 \frac{\partial \Omega }{\partial y} \\
y(\Phi^3+\Psi^3)(\Phi\Psi)^2 & \frac{\partial \Omega^5 }{\partial x} & (\Phi\Psi)^3 \frac{\partial \Omega }{\partial x} 
\end{pmatrix}^T,
\end{equation}}
\!\!where {\scriptsize$\begin{pmatrix}
a_1 & a_2 & a_3 \\
b_1 & b_2 & b_3
\end{pmatrix}\in G(2,3)$} is the unique plane which is perpendicular to 
{\small\begin{equation}\label{perp s4 deg 29}{\scriptsize
(x(\Phi^3+\Psi^3)(\Phi\Psi)^2, -\frac{\partial \Omega^5 }{\partial y}, -(\Phi\Psi)^3 \frac{\partial \Omega}{\partial y})|_{(x,y)=(t,1)},}
\end{equation}}
\!\!under the bilinear inner product. This class depends on one freedom $t$ and belongs to the generally ramified family.
\item[\rm{(3)}]  If $\deg g=31$, then 
\begin{equation}\label{deg31 s4}{\small
g=\begin{pmatrix}
a_1 & a_2 & a_3 \\
b_1 & b_2 & b_3
\end{pmatrix}\begin{pmatrix}
x\Omega^3(\Phi^3+\Psi^3) & -\Omega^3\Phi\Psi \frac{\partial \Omega}{\partial y} & -(\Phi\Psi)^3\frac{\partial ( \Phi\Psi)}{\partial y}\\
y\Omega^3(\Phi^3+\Psi^3) & \Omega^3\Phi\Psi \frac{\partial \Omega}{\partial x} & (\Phi\Psi)^3\frac{\partial (\Phi\Psi)}{\partial x}
\end{pmatrix}^T,}
\end{equation}
where {\scriptsize$\begin{pmatrix}
a_1 & a_2 & a_3 \\
b_1 & b_2 & b_3
\end{pmatrix}\in G(2,3)$} is the unique plane which is perpendicular to 
{\small\begin{equation}\label{perp s4 deg 31}{\scriptsize
(y\Omega^3(\Phi^3+\Psi^3), \Omega^3\Phi\Psi \frac{\partial \Omega}{\partial x}, (\Phi\Psi)^3\frac{\partial (\Phi\Psi)}{\partial x})|_{(x,y)=(t,1)},}
\end{equation}}
\!\!under the bilinear inner product. This class depends on one freedom $t$ and belongs to the exceptional transversal family.

\end{enumerate}
\end{theorem}

\begin{proof} 
We choose $\phi_2$ to be $\gamma_1$ stated in the Lemma \ref{two isoms of S4}. The other choice $\gamma_2$ is similar. 

It follows from  Corollary \ref{Cor4.1} that $24\leq \deg(g)\leq 36$.
Since the ramified multiplicities of the zeros of $\Omega$ are $4$, they are mapped to the vertices of the octahedron under $g$. Since the ramified multiplicities of the zeros of $\Phi\Psi$ are all $3$, they are mapped to the center of faces of octahedron under $g$. Since $(1-i)/\sqrt{2}$, one of the zeros of $(\Phi^3+\Psi^3)/2$, is invariant under $\tau\circ\tau\circ\sigma:w\rightarrow -i/w$, we see that  $g((1-i)/\sqrt{2})$ lies in the axis of $T_{2,3}=T_{3,0}^2T_{1,1}$, a half-turn around the axis through the midpoints of opposite edges. Hence the zeros of $(\Phi^3+\Psi^3)/2$ are all mapped to the midpoints of edges. Therefore, $\deg \mathcal{F}\geq 6$, and 
\[\deg \mathcal{Q}-\deg \mathcal{F}=2\deg g-6(\deg g-24)=144-4\deg g\geq 20.\] 
Thus $25\leq \deg g\leq 31$. So, the cases $\deg g=24,$ and $32$ through $36$ do not occur. Moreover, by Proposition \ref{poincareEven}, we can also rule out the cases when $\deg g$ is even.

\textbf{Case (1)}: Assume $\deg g=25$. Then $\deg \mathcal{F}=6$ and $\mathcal{F}=(\Omega)$. 
Through $\phi_1$, we have $\rchi_{V_{25}}=3\rchi_4+2\rchi_5+4\rchi_8$,  where  
\begin{scriptsize}$$3W_4=\Span\{y\Omega^4,-x\Omega^4\}\oplus \Span\{y(\Phi\Psi)^3,-x(\Phi\Psi)^3\}\oplus\Span\{(\Phi^3+\Psi^3)\Phi\Psi \frac{\partial \Omega}{\partial x},(\Phi^3+\Psi^3)\Phi\Psi \frac{\partial \Omega}{\partial y}\},$$\end{scriptsize}
is constructed by invariants $\Omega^4,~(\Phi\Psi)^3,~(\Psi^3+\Psi^3)\Phi\Psi$ of degree $24,24,20$ with characters $\rchi_1,~\rchi_1$ and $\rchi_2$, respectively ((1), (2) in Proposition \ref{key basis prop} are also used). They are equivalent to $\mathbb{C}^2$ with character $\rchi_4$ when the first polynomial maps to $e_1$ and the second maps to $e_2$. Similarly,  \begin{scriptsize}$$2W_5=\Span\{y\Omega^2(\Phi^3+\Psi^3),-x\Omega^2(\Phi^3+\Psi^3)\}\oplus\Span\{\Omega^2\Phi\Psi\frac{\partial \Omega}{\partial x},\Omega^2\Phi\Psi\frac{\partial \Omega}{\partial y}\}.$$\end{scriptsize} 
We exclude it as its vectors have common factors.
Tensoring $3W_4$ with $\mathbb{C}^2$, from 
\eqref{character table 2O} we obtain the following invariant subspace,
{\scriptsize \[
3W_1= \Span\{y\Omega^4\otimes e_2-(-x)\Omega^4\otimes e_1,y(\Phi\Psi)^3\otimes e_2-(-x)(\Phi\Psi)^3\otimes e_1, (\Phi^3+\Psi^3)\Phi\Psi \frac{\partial \Omega}{\partial x}\otimes e_2-(\Phi^3+\Psi^3)\Phi\Psi \frac{\partial \Omega}{\partial y}\otimes e_1\}.\]}
A plane $g(x,y)$ in $3W_1$ is spanned by the two rows of the matrix \eqref{deg 25 S4} and
\begin{small}$$\det g=6\Omega(\Phi\Psi)(\Phi^3+\Psi^3)(p_{13}\Omega^4+p_{23}(\Phi\Psi)^3).$$\end{small}
It follows from Lemma \ref{linear system of D2} that {\small$\mathcal{Q}=5(\Omega)+(\Phi\Psi)+(\frac{\Phi^3+\Psi^3}{2})$}, or {\small$(\Omega)+4(\Phi\Psi)+(\frac{\Phi^3+\Psi^3}{2})$}, or {\small$(\Omega)+(\Phi\Psi)+3(\frac{\Phi^3+\Psi^3}{2})$}, or the generic case {\small$(\Omega)+(\Phi\Psi)+(\frac{\Phi^3+\Psi^3}{2})+\varphi^{-1}(\varphi(t))$}, where {\small$\varphi^{-1}(\varphi(t))$} consists of $24$ distinct points constituting the principal orbit of $S_4$ through the point $t$. In all the above cases, we obtain the desired  curves 
downstairs by computing {\scriptsize$g\cdot uv(u^4-v^4)/xy(x^4-y^4)$} and using the invariant theory. It follows from Theorem~\ref{thm} that such curves always belong to the generally ramified family. 

\textbf{Case (2)}: Assume $\deg g=27$. Through $\phi_1$, we have 
{\small$\rchi_{V_{27}}=2\rchi_4+2\rchi_5+5\rchi_8$}, where
{\scriptsize\begin{align*}
2W_4&=\Span\{y\Omega(\Phi\Psi)(\Phi^3+\Psi^3),-x\Omega(\Phi\Psi)(\Phi^3+\Psi^3)\}\oplus \Span\{\Omega(\Phi\Psi)^2\frac{\partial \Omega}{\partial x},\Omega(\Phi\Psi)^2\frac{\partial \Omega}{\partial y}\},\\
2W_5&=\Span\{y\Omega^3(\Phi\Psi),-x\Omega^3(\Phi\Psi)\}\oplus \Span\{(\Phi^3+\Psi^3)\frac{\partial (\Phi\Psi)^2}{\partial x},(\Phi^3+\Psi^3)\frac{\partial (\Phi\Psi)^2}{\partial y}\}, 
\end{align*}}
\!\!both of which are excluded as vectors in them have common factors.

\textbf{Case (3)}: Assume $\deg g=29$. Then $\deg \mathcal{F}=30$ and $\mathcal{F}>(\Omega)$. Through $\phi_1$, we have  
{\small$\rchi_{V_{29}}=2\rchi_4+3\rchi_5+5\rchi_8$}, where 
\begin{scriptsize}$$\aligned&3W_5=\Span\{y(\Phi^3+\Psi^3)(\Phi\Psi)^2,-x(\Phi^3+\Psi^3)(\Phi\Psi)^2\}\oplus \Span\{\frac{\partial \Omega^5}{\partial x}, \frac{\partial \Omega^5}{\partial y}\}\oplus \Span\{(\Phi\Psi)^3\frac{\partial \Omega}{\partial x},(\Phi\Psi)^3\frac{\partial \Omega}{\partial y}\},\\
 &2W_4=\Span\{y(\Omega\Phi\Psi)^2,-x(\Omega\Phi\Psi)^2\}\oplus\Span\{\Omega^2(\Phi^3+\Psi^3)\frac{\partial \Omega}{\partial x},\Omega^2(\Phi^3+\Psi^3)\frac{\partial \Omega}{\partial y} \}.\endaligned$$\end{scriptsize} 
 We exclude $2W_4$ as its vectors have common factors.
 Tensoring $3W_5$ with $\mathbb{C}^2$, from 
 \eqref{character table 2O} we obtain the following invariant subspace,
\begin{scriptsize}
$$3W_2= \Span\{(\Phi^3+\Psi^3)(\Phi\Psi)^2(y\otimes e_2-(-x)\otimes e_1), \frac{\partial \Omega^5}{\partial x}\otimes e_2-\frac{\partial \Omega^5}{\partial y}\otimes e_1, (\Phi\Psi)^3(\frac{\partial \Omega}{\partial x}\otimes e_2-\frac{\partial \Omega}{\partial y}\otimes e_1)\}.$$\end{scriptsize}
A plane $g(x,y)$ in $3W_2$ is spanned by the two rows of the matrix of \eqref{s4 deg 29} with
\begin{small}$$\det g=6(\Phi^3+\Psi^3)(\Phi\Psi)^2\Omega(5p_{12}\Omega^4+p_{13}(\Phi\Psi)^3).$$\end{small}
It follows from Lemma \ref{linear system of D2} and the fact that the zeros of $\Phi\Psi$ and $\frac{\Phi^3+\Psi^3}{2}$ do not lie on the free lines, that {\small$\mathcal{Q}=5(\Omega)+2(\Phi\Psi)+({\Phi^3+\Psi^3})$}, 
or the generic case {\small$(\Omega)+2(\Phi\Psi)+({\Phi^3+\Psi^3})
+\varphi^{-1}(\varphi(t))$}, where {\small$\varphi^{-1}(\varphi(t))$} consists of $24$ distinct points constituting the principal orbit of $S_4$ through the point $t$. By Theorem~\ref{thm}, in all cases, the sextic curves $F$ downstairs are tangent to the $1$-dimensional orbit at $F(0)$.

We only discuss the generic case {\small$\mathcal{Q}=\Omega+2(\Phi\Psi)+({\Phi^3+\Psi^3}) 
+\varphi^{-1}(\varphi(t))$}, where $t$ lies in one of principal orbits of $S_4$, and {\small$\mathcal{F}=\Omega+\varphi^{-1}(\varphi(t))$}. So, we obtain that $g(t)$ lies on one of the six free lines. The action of multiplying {\small$T_{1,1}^T$} on the right gives a permutation $(3654)$, while the action of multiplying {\small$T_{3,0}^T$} on the right gives a permutation $(1325)$ on the six free lines. So, $S_4$ acts on the six free lines transitively. Therefore, interchanging $t$ with another points of {\small$\varphi^{-1}(\varphi(t))$}, we may assume that $g(t)$ lies on the free line $\mathcal{L}_1$. Thus, we obtain the perpendicular conclusion in item $2$ from \eqref{eq-L12}. 

The  curves $F$ dowstairs are computed by {\scriptsize$g\cdot uv(u^4-v^4)/xy(x^4-y^4)h(x,y)$} and invariant theory, where {\small$h(x,y):=5p_{12}\Omega^4+p_{13}(\Phi\Psi)^3$} is a polynomial of degree $24$ vanishing at 
$\varphi^{-1}(\varphi(t))$. 

The case {\small$\mathcal{Q}=5(\Omega)+2(\Phi\Psi)+(\Phi^3+\Psi^3)$} 
is obtained by letting $t\rightarrow 0$ in the generic case.

\textbf{Case (4)}: Assume that $\deg g=31$. Then $\deg \mathcal{F}=42$ and $\mathcal{F}>(\Omega)$. Through $\phi_1$, we have 
{\small$\rchi_{V_{31}}=3\rchi_4+3\rchi_5+5\rchi_8$},  where 
{\scriptsize\begin{align*}
3W_4&=\Span\{y\Omega^3(\Phi^3+\Psi^3),-x\Omega^3(\Phi^3+\Psi^3)\}\oplus  \Span\{\Omega^3\Phi\Psi\frac{\partial \Omega}{\partial x},\Omega^3\Phi\Psi\frac{\partial \Omega}{\partial y}\}\oplus \Span\{(\Phi\Psi)^3\frac{\partial (\Phi\Psi)}{\partial x},(\Phi\Psi)^3\frac{\partial (\Phi\Psi)}{\partial y}\},\\
3W_5&=\Span\{y\Omega(\Phi\Psi)^3,-x\Omega(\Phi\Psi)^3\}\oplus \Span\{y\Omega(\Phi^3+\Psi^3)^2,-x\Omega(\Phi^3+\Psi^3)^2\}\oplus \Span\{\Omega\Phi\Psi(\Phi^3+\Psi^3)\frac{\partial \Omega}{\partial x}, \Omega\Phi\Psi(\Phi^3+\Psi^3)\frac{\partial \Omega}{\partial y}\}.
\end{align*}}
\!\!Meanwhile, {\small$\rchi_{V_{31}\otimes \mathbb{C}^2}=\rchi_{V_{31}}\cdot \rchi_4=3\rchi_1+3\rchi_2+5\rchi_3+8\rchi_6+8\rchi_7$},
where $3W_1$ (respectively, $3W_2$) is obtained from {\small$3W_4\otimes \mathbb{C}^2$} (respectively, {\small$3W_5\otimes \mathbb{C}^2$}). Since $\Omega$ is a common factor of any vector in $3W_5$, we omit $3W_2$ in the following. The other invariant subspace $3W_1$ is given by the following 
{\scriptsize\[3W_1=\Span\{\Omega^3(\Phi^3+\Psi^3)(y\otimes e_2-(-x)\otimes e_1), \Omega^3\Phi\Psi (\frac{\partial \Omega}{\partial x}\otimes e_2-\frac{\partial \Omega}{\partial y}\otimes e_1), (\Phi\Psi)^3(\frac{\partial (\Phi\Psi)}{\partial x}\otimes e_2-\frac{\partial (\Phi\Psi)}{\partial y}\otimes e_1)\}.\]}

A plane $g(x,y)$ in $3W_1$ is spanned by the two rows of the matrix \eqref{deg31 s4} with
\begin{small}$$\det g=\Omega^3(\Phi\Psi)(\Phi^3+\Psi^3)(6p_{12}\Omega^4+(8p_{13}-4p_{23})(\Phi\Psi)^3).$$\end{small}
It follows from Lemma \ref{linear system of D2} and the fact that the zeros of $\Phi\Psi$ and $\Phi^3+\Psi^3$ do not lie on the free lines that {\small$\mathcal{Q}=7(\Omega)+(\Phi\Psi)+(\Phi^3+\Psi^3)$} in the non-generic case, or in the generic case {\small$3(\Omega)+(\Phi\Psi)+(\Phi^3+\Psi^3)+\varphi^{-1}(\varphi(t))$}, where {\small$\varphi^{-1}(\varphi(t))$} consists of $24$ distinct points constituting the principal orbit of $S_4$ through the point $t$. By Theorem~\ref{thm}, in all cases, the sextic curves downstairs belong to exceptional transversal family.

We only discuss the generic case when {\small$\mathcal{Q}=3(\Omega)+(\Phi\Psi)+(\Phi^3+\Psi^3)+\varphi^{-1}(\varphi(t))$}, where $t$ lies in one of principal orbits of $S_4$, and {\small$\mathcal{F}=3(\Omega)+\varphi^{-1}(\varphi(t))$}. So, we obtain that $g(t)$ lies on one of  the six free lines. Similar to the preceding case, we may assume that $g(t)$ lies on the free line $\mathcal{L}_1$. Then the perpendicular conclusion in item $3$ is obtained from \eqref{eq-L12}. 

The curves $F$ downstairs are computed by {\scriptsize$g\cdot uv(u^4-v^4)/\Omega^3 h(x,y)$} and invariant theory, where {\small$h(x,y):=6p_{12}\Omega^4+(8p_{13}-4p_{23})(\Phi\Psi)^3$} is a polynomial of degree $24$ whose zeros are all points of {\small$\varphi^{-1}(\varphi(t))$}.
The non-generic case {\small$\mathcal{Q}=7(\Omega)+(\Phi\Psi)+(\Phi^3+\Psi^3)$ is obtained by letting $t\rightarrow 0$} in the generic case.
\end{proof}

\vspace{2mm}

\noindent {\small Department of Mathematics, Washington University, St. Louis, MO 63130\\
School of Mathematical Sciences, Beihang University, Beijing 102206, China\\
School of Mathematical Sciences and LPMC, Nankai University, Tianjin 300071, China}

\begin{tiny}\noindent E-mail: chi@wustl.edu;~~~\phantom{,,}xiezhenxiao@buaa.edu.cn;~~~\phantom{,,}xuyan2014@mails.ucas.ac.cn.\end{tiny}

\end{document}